\documentclass{tac}
\usepackage{amsmath}
\usepackage{amssymb}
\usepackage[legacycolonsymbols]{mathtools}
\usepackage[utf8]{inputenc}
\usepackage[T1]{fontenc}
\usepackage{graphicx}
\usepackage{xcolor}
\usepackage{tikz}
\usetikzlibrary{arrows}
\usepackage{longtable}
\usetikzlibrary{calc}
\usetikzlibrary{arrows.meta}
\usepackage[toc,page]{appendix}

\usepackage{enumitem}

\usepackage[normalem]{ulem} 

\newcommand\reddout{\bgroup
  \markoverwith{\lower-0.2ex\hbox{\kern-.03em
  \textcolor{red}{\vbox{\hrule width.2em\kern0.45ex\hrule}}%
  \kern-.03em}}%
  \ULon}
\makeatother

\newcommand{\newtext}[1]{\textcolor{red}{#1}}
\newcommand{\oldtext}[1]{\reddout{#1}}
\newcommand{\modefy}[2]{\oldtext{#1} \newtext{#2}}

\usepackage{xy}

\input diagxy

\def\xypic{\hbox{\rm\Xy-pic}}

\usepackage[colorlinks=true]{hyperref}
\hypersetup{allcolors=[rgb]{0.1,0.1,0.4}}

\usepackage[absolute,overlay]{textpos}

\author{Masaki Fukuda, and Tommy Shu}

\thanks{I would like to express my sincere gratitude to my supervisor, Professor Yuji Terashima, for his invaluable guidance, insightful comments, and constant support throughout this research. The work of MF is supported by JST SPRING, Grant Number JPMJSP2114, and the work of TS is supported by the Tohoku University Advanced Graduate Program for AI and Electronics.}

\address{Department of Physics, Tohoku University\\
 \\[5pt]
 Department of Mathematics, Tohoku
University \\
}

\title{3-Crossed Modules, Quasi-Categories, and the Moore Complex}

\copyrightyear{2025}


\eaddress{masaki.fukuda.r8@dc.tohoku.ac.jp\CR seiyo.shu.r7@dc.tohoku.ac.jp}

\def\xypic{\hbox{\rm\Xy-pic}}

\newtheorem{dfn}{Definition}

\newtheorem{theorem}{Theorem}

\newtheoremrm{rem}{Remark}

\newcommand{\sym}[1]{|\text{\scriptsize $#1$}|}
\newcommand{\act}[1]{\,\sp{#1}\!}

\newcommand{\frakg}{{\mathfrak{g}}}
\newcommand{\frakh}{{\mathfrak{h}}}
\newcommand{\frakl}{{\mathfrak{l}}}
\newcommand{\frakm}{{\mathfrak{m}}}

\newcommand{\HLp}[2]{{\{\!\!\{#1,#2\}\!\!\}}}
\newcommand{\del}{\partial}

\let\pf\proof
\let\epf\endproof
\mathrmdef{Hom}
\mathbfdef{Set}
\renewcommand{\to}{\rightarrow}
\begin{document}

\begin{textblock*}{3cm}(15cm,1cm)
\raggedleft
TU-1294
\end{textblock*}
 
  \maketitle

\begin{abstract}

The established equivalence between 2-crossed modules and Gray 3-groups \cite{SarikayaUlualan2024} serves as a benchmark for higher-dimensional algebraic models. 
However, to the best of our knowledge, the established definitions of 3-crossed modules \cite{Arvasi2009} are not clearly suited for extending this equivalence.
In this paper, we propose an alternative formulation of a 3-crossed module, equipped with a new type of lifting, which is specifically designed to serve as a foundation for this higher-order categorical correspondence.
As the primary results of this paper, we validate this new structure. We prove that the simplicial set induced by our 3-crossed module forms a quasi-category. Furthermore, we show that the Moore complex of length 3 associated with a simplicial group naturally admits the structure of our 3-crossed module. This work establishes our definition as a robust candidate for modeling the next level in this algebraic-categorical program.
\end{abstract}


\tableofcontents
\clearpage

\section{Introduction}\label{sec-Introduction}

Crossed modules were introduced by Whitehead in 1949 as an algebraic model for homotopy 2-types. Conduché extended this notion to 2-crossed modules, providing an algebraic model for homotopy 3-types.

Crossed modules and 2-crossed modules are closely related to higher category theory. Noohi showed that the category of crossed modules is equivalent to the category of 2-groups \cite{Noohi2007}. Similarly, Sarikaya demonstrated that the category of 2-crossed modules is equivalent to that of Gray 3-groups \cite{SarikayaUlualan2024}. This equivalence establishes a crucial benchmark: a 'correct' definition of an n-crossed module should be algebraically rich enough to correspond to an $(n+1)$-dimensional higher categorical structure.

Crossed modules and 2-crossed modules have proven to be useful tools in the study of low-dimensional topology. For example, crossed modules play a key role in the definition of the twisted Yetter invariant, while 2-crossed modules appear in certain 4-dimensional invariants based on the classical 3BF theory \cite{RadenkovicVojinovic2022}. These two types of topological invariants can be viewed as generalizations of the Dijkgraaf–Witten invariant.

The Dijkgraaf–Witten invariant, constructed using a finite group and a triangulation of a manifold, is computed by counting the number of assignments of group elements to 1-simplices that satisfy certain compatibility conditions determined by the 2-simplices. In the case of the twisted Yetter invariant, a similar construction is used; however, instead of a single finite group, a finite crossed module is employed. In this setting, elements of one group in the crossed module are assigned to 1-simplices, while elements of the second group are assigned to 2-simplices, subject to compatibility conditions involving both 2-simplices and 3-simplices.

This idea naturally extends to 2-crossed modules, where an additional level of algebraic structure appears. In particular, the Peiffer lifting—a distinctive feature of 2-crossed modules—emerges in the compatibility conditions determined by the 4-simplices.

These two threads of inquiry—the algebraic modeling of homotopy types and the topological invariants—naturally motivate the search for a 3-crossed module.

A definition for \lq\lq3-crossed modules\rq\rq\ was proposed by Arvasi et al. \cite{Arvasi2009}. However, its structural properties and suitability for extending the aforementioned categorical equivalence \cite{SarikayaUlualan2024} are not fully clear. This ambiguity motivates the search for a robust alternative formulation.

This paper proposes such an alternative definition of a 3-crossed module. Our construction is specifically designed to serve as a foundation for extending the 2-crossed module and Gray 3-group equivalence to the next dimension.

In this paper, we first re-examine the 2-crossed module case to establish the properties required for this extension. We show that a 2-crossed module gives rise to a simplicial set, derived from the coloring conditions in low-dimensional topology, and prove that this simplicial set forms a quasi-category. We then introduce our new definition of a 3-crossed module, equipped with a new type of lifting, and demonstrate how this lifting naturally arises in extended coloring conditions. As the primary results of this paper, we prove that the simplicial set associated with our 3-crossed module also forms a quasi-category. Furthermore, we validate our definition by showing that the Moore complex of length 3 associated with a simplicial group admits the structure of our newly defined 3-crossed module.

The results presented in this paper support our definition, providing a foundation for a future paper where we will construct the corresponding one-dimension-higher Gray category and prove the anticipated equivalence.

\section{Preliminaries}\label{sec-Preliminaries}

\subsection{Simplicial sets and simplicial groups.}

In this section, we review some basic notions regarding simplicial sets and simplicial groups. The exposition here is based on \cite{RezkIntroQCats}.

We denote by $\Delta$ the category whose structure is as follows:
\begin{itemize}
  \item Objects are finite, non-empty, totally ordered sets of the form $[n] \coloneqq \{0 < 1 < \cdots < n\}$ for $n \ge 0$.
  \item Morphisms $\delta : [n] \rightarrow [m]$ are weakly monotonic functions, i.e., functions such that $x \le y$ implies $\delta(x) \le \delta(y)$.
\end{itemize}

We refer to morphisms in $\Delta$ as \textbf{simplicial operators}. Among these, there are distinguished morphisms called \textbf{face} and \textbf{degeneracy} operators:
\[
  d^n_i \coloneqq \langle 0, \dots, \hat{i}, \dots, n \rangle : [n - 1] \rightarrow [n], \quad 0 \le i \le n,
\]
\[
  s^n_i \coloneqq \langle 0, \dots, i, i, \dots, n \rangle : [n + 1] \rightarrow [n], \quad 0 \le i \le n.
\]

These face and degeneracy operators satisfy the following five identities:
\begin{equation}
\label{psgrp}
\left\{
  \begin{array}{ll}
    d^{n+1}_j \circ d^n_i = d^{n+1}_i \circ d^n_{j-1}, & \text{for } i < j; \\
    s^{n+1}_j \circ s^n_i = s^{n+1}_i \circ s^n_{j+1}, & \text{for } i \le j; \\
    s^n_j \circ d^{n+1}_i = \mathrm{id}_{[n]}, & \text{for } i = j \text{ or } i = j+1; \\
    s^n_j \circ d^{n+1}_i = d^n_i \circ s^{n-1}_{j-1}, & \text{for } i < j; \text{ and } \\
    s^n_j \circ d^{n+1}_i = d^n_{i-1} \circ s^{n-1}_j, & \text{for } i > j+1.
  \end{array}
\right.
\end{equation}

Let $\mathbf{Set}$ denote the category of sets, and let $\mathbf{Grp}$ denote the category of groups.  
A \textbf{simplicial set} is a contravariant functor $X : \Delta^{\mathrm{op}} \to \mathbf{Set}$. Similarly, a \textbf{simplicial group} is a contravariant functor $X : \Delta^{\mathrm{op}} \to \mathbf{Grp}$. We write $X_n$ for $X([n])$, and call it the set of \textbf{n-simplices} in X.

The \textbf{standard $n$-simplex} $\Delta^n$ is the simplicial set defined by
\[
\Delta^n \coloneqq \operatorname{Hom}_\Delta(-, [n]).
\]
Explicitly, this means that 
\[
(\Delta^n)_m = {\rm Hom}_\Delta ([m], [n]) = \{ \text{simplicial operators } a:[m] \rightarrow [n]  \},
\]
while the action of simplicial operators on cells of $\Delta^n$ is given by composition: $f: [m'] \rightarrow [m]$ sends $(a: [m] \rightarrow [n])\in (\Delta^n)_m $ to $(a \circ f : [m'] \rightarrow [n] \in (\Delta^n)_{m'})$.

\subsection{Quasi-categories}

We begin by defining \emph{horns}. For each $n \ge 1$ and $0 \le j \le n$, the \textbf{$j$-th horn} $\Lambda^n_j$ is a subcomplex of the standard $n$-simplex $\Delta^n$ defined by
\[
(\Lambda^n_j)_k = \{ f : [k] \to [n] \mid [n] \setminus \{ j \} \not\subset f([k]) \}.
\]
In other words, $\Lambda^n_j$ is the union of all $(n-1)$-dimensional faces of $\Delta^n$ except the $j$-th face:
\[
\Lambda^n_j = \bigcup_{i \neq j} \Delta^{[n] \setminus \{i\}} \subset \Delta^n.
\]
When $0 < j < n$, the horn $\Lambda^n_j \subset \Delta^n$ is called an \textbf{inner horn}.

A \textbf{quasi-category} is a simplicial set $C$ such that, for every map $f : \Lambda^n_j \to C$ from an inner horn, there exists its extension $g : \Delta^n \to C$. That is, $C$ is a quasi-category if the restriction map:
\[
\operatorname{Hom}(\Delta^n, C) \to \operatorname{Hom}(\Lambda^n_j, C)
\]
induced by inclusion $\Lambda^n_j \hookrightarrow \Delta^n$ is surjective for all $n \ge 2$ and $0 < j < n$. In other words, there always exists a dotted arrow completing the following commutative diagram.

\begin{center}
\begin{tikzpicture}[auto]
\node (a) at (0, 1.2) {$\Lambda^n_j$};
\node (x) at (3.2, 1.2) {$C$};
\node (b) at (0, 0) {$\Delta^n$};  
\draw[->] (a) to node {} (x);
\draw[>->] (a) to node {} (b);
\draw[->, dashed] (b) to node[swap] {} (x);
\end{tikzpicture}
\end{center}

\subsection{2-crossed modules}

We now define the notion of a 2-crossed module. Our definition differs slightly from that given in \cite{MartinsPicken2011}. In particular, we use the notation ${}^a b$ for the right action, instead of writing it as  $a \triangleright b$.

A \textbf{2-crossed module} is given by a complex of groups:
\[
L \xrightarrow{\partial} H \xrightarrow{\partial} G 
\]
together with the following three structures:
\begin{itemize}
  \item a left action of $G$ by automorphisms on both $H$ and $L$ (with the action on $G$ given by conjugation);
  \item a left action of $H$ by automorphisms on $L$ (with the action on $H$ also given by conjugation); and
  \item a function $\{-,-\} : H \times H \rightarrow L$ (called the Peiffer lifting);
\end{itemize}
which are required to satisfy the following ten axioms:
\begin{enumerate}
  \item  $\partial\circ \partial =0$;
  \item  $\partial(^gh)= {^g}\partial(h),\ \partial({^g}l)= {^g}\partial(l),\ \partial(^hl)= {^h}\partial(l)$, for each $g \in G$, $h \in H$, and $l \in L$;
  \item  $^g \{h_2,h_1\} = \{\act{g} h_2, \act{g}h_1\}$, for each $g\in G$ and $h_1,h_2 \in H$;
  \label{pfeq}
  \item $\partial  \{h_2, h_1\} = h_2 h_1 h_2^{-1} \act{\partial h_2}h_1^{-1}$, for each $h_1,h_2 \in H$;
  \item $\act{\del l}l' = l l' l^{-1}$, for each $l,l' \in L$;
\label{2cm_pf_id}
  \item $\{\partial l_2,\partial l_1\} = l_2l_1l_2^{-1}l_1^{-1}$, for each $l_1,l_2 \in L$;
\label{2cm_not_need_pl}
  \item $\{h_3h_2,h_1\}={^{h_3}}\{h_2, h_1\}\{h_3, {^{\partial h_2}}h_1 \}$, for each $h_1,h_2, h_3 \in H$;
  \label{lpl}
  \item $\{h_3,h_2 h_1\}=\{h_3, h_2\} {^{(^{\partial h_3} h_2)}}\{h_3, h_1\}$, for each $h_1,h_2, h_3 \in H$;
  \label{rpl}
  \item ${^h}l = l\{\partial l^{-1}, h \}$, for each $h \in H$, and $l\in L$;
  \label{acthl}
  \item ${^{\partial h}}l = {^h}l\{h, \partial l^{-1} \}$, for each $h\in H$, and $l \in L$; and
  \label{actgl}
  \item $l^{-1} {^{\partial h}}l=\{\partial l^{-1}, h \}\{h, \partial l^{-1} \}$, for each $h \in H$, and $l \in L$.
  \label{2cm_not_need}
\end{enumerate}

\begin{rem}
In our context of 2-crossed modules, property \ref{2cm_not_need} follows from properties \ref{acthl} and \ref{actgl}. Similarly, property \ref{2cm_not_need_pl} can be derived by applying properties \ref{2cm_pf_id} and \ref{acthl} to the expression $\act{\del l_2}l_1$. We include these properties explicitly to emphasize that our structure satisfies the same axioms as standard 2-crossed modules.
\end{rem}

\section{2-crossed modules and simplicial sets}\label{sec-2-crossed module and sset}
Before presenting the definition of 3-crossed modules and proving that the simplicial set associated with a 3-crossed module is a quasi-category, we construct a simplicial set using a 2-crossed module and show that it is a quasi-category. Most of the discussion in this section serves as a simplified case of the 3-crossed module.

Let $W \coloneqq (L \xrightarrow{\partial} H \xrightarrow{\partial} G,\{-,-\} )$ be a 2-crossed module. We construct a simplicial set $M_W$ by the following data:

\begin{itemize}
    \item For each $[n] \in \Delta$, the set $M_W([n])$ consists of all the tuples $(g(-,-), h(-,-,-), l(-,$ $-,-,-))$, where:
    \begin{enumerate}
        \item $g: [n] \times_{\le} [n] \to G$ is a function assigning each ordered pair in $[n]$ to an element of $G$;
        \item $h:[n] \times_{\le} [n] \times_{\le} [n] \to H$ is a function assigning  each ordered triple in $[n]$ to an element of $H$; and
        \item $l: [n] \times_{\le} [n] \times_{\le} [n] \times_{\le} [n] \to L$ is a function assigning each ordered quadruple in $[n]$ to an element of $L$;
    \end{enumerate}
    and the three fuctions $g,h$ and $l$ satisfy the following conditions:

\begin{enumerate}[label=(\alph*)]
\item for each $i \le j \le k$ in $[n]$, we have the idenntity 
\begin{equation}
g_{ik} = \partial(h_{ijk})g_{jk}g_{ij};
\label{r2-simplex}
\end{equation}
\item for each $i \le j \le k \le m$ in $[n]$, we have the idenntity 
\begin{equation}
\partial (l_{ijkm})  h_{ikm}{^{g_{km}}}h_{ijk} = h_{ijm}h_{jkm};
\label{r3-simplex}
\end{equation}
\item for each $i \le j \le k \le m \le p$ in $[n]$, we have the idenntity 
\begin{equation}
 {^{h_{ijp}}}l_{jkmp}l_{ijmp}{^{h_{imp}}}\left({^{g_{mp}}}l_{ijkm}\right) = l_{ijkp}{^{h_{ikp}}}\{h_{kmp},{^{g_{mp}g_{km}}}h_{ijk} \}^{-1}l_{ikmp};
\label{r4-simplex}
\end{equation}
 \item for each $i \in [n]$, we have $g_{ii} = e_G$;
        
 \item for each $i \le j$ in $[n]$, we have $h_{iij} = h_{ijj} = e_H$; and
        
 \item for each $i \le j \le k$ in $[n]$, we have $l_{iijk} = l_{ijjk} = l_{ijkk} = e_L$.
    \end{enumerate}
    
    Here, we simply denote  $g(i,j)$, $h(i,j,k)$, and $l(i,j,k,m)$ by $g_{ij}$, $h_{ijk}$, and $l_{ijkm}$, respectively.
    
\item For a simplicial operater $\delta : [m] \rightarrow [n]$, we define $M_W(\delta): (M_W)_n \rightarrow (M_W)_m$ as follows: for each $(g,h,l) \in (M_W)_n$, 
\[
M_W(\delta)
\begin{pmatrix}
g(-,-)\\
h(-,-,-)\\
l(-,-,-,-)\\
\end{pmatrix}
\coloneqq
\begin{pmatrix}
g(\delta(-),\delta(-))\\
h(\delta(-),\delta(-),\delta(-))\\
l(\delta(-),\delta(-),\delta(-),\delta(-))\\
\end{pmatrix}
.\]
\end{itemize}

\noindent
We will show that simplicial set $M_W$ is a quasi-category in the following theorem.

\begin{theorem}\label{2cm and qusi}
Let $W \coloneqq (L \xrightarrow{\partial} H \xrightarrow{\partial} G,\{-,-\} )$ be a 2-crossed module. Then the simplicial set $M_W$ is a quasi-category.
\end{theorem}

In order to prove Theorem~\ref{2cm and qusi}, we need the following two lemmas.

\begin{lemma}
\label{lemma1}
Let \( W \coloneqq (L \xrightarrow{\partial} H \xrightarrow{\partial} G, \{-,-\}) \) be a 2-crossed module. For any \( n \ge 1 \) and any inner horn \( \Lambda^n_j \), let \( \phi \in \operatorname{Hom}(\Lambda^n_j, M_W) \). Then, for any \( m \ge 1 \) and any \( f \in (\Lambda^n_j)_m \), the following identity holds:
\[
\phi_m(f) 
\coloneqq
\begin{pmatrix}
\phi^g_m(f)(-,-) \\
\phi^h_m(f)(-,-,-) \\
\phi^l_m(f)(-,-,-,-)
\end{pmatrix}
=
\begin{pmatrix}
\phi^g_1(f(-,-))(0,1) \\
\phi^h_2(f(-,-,-))(0,1,2) \\
\phi^l_3(f(-,-,-,-))(0,1,2,3)
\end{pmatrix}.
\]
\end{lemma}

For any $m\ge 1$, any $k+1$ integers  \( 0\le x_0 \le x_1 \le \dots \le x_k \le m\)
and
each \( f \in (\Lambda^n_j)_m \),
clearly there exists a function \(f' \in (\Lambda^n_j)_k \) 
satisfying $f'(i) = f(x_i)$ for each $i$, $0\leq i\leq k$.
We renamed such a function $f'$ to  \( f(x_0, x_1, \dots, x_k) \in (\Lambda^n_j)_k \) 
for improving readability.

\begin{proof}
We will prove the case of \( \phi^l_m(f)(i,j,k,p) = \phi^l_3(f(i,j,k,p))(0,1,2,3) \). The other components \( \phi^g_m \) and \( \phi^h_m \) can be handled in the same way.

Let \( \phi \in \operatorname{Hom}(\Lambda^n_j, M_W) \) and let \( f \in (\Lambda^n_j)_m \) for \( m \ge 1 \). We consider \( \phi^l_m(f)(i,j,k,p) \) for \( i \le j \le k \le p \) in \( [m] \).

Define a function \( \langle ijkp \rangle : [3] \to [m] \) by:
\[
\langle ijkp \rangle(x) \coloneqq
\begin{cases}
i & \text{if } x = 0, \\
j & \text{if } x = 1, \\
k & \text{if } x = 2, \\
p & \text{if } x = 3.
\end{cases}
\]
Using the defining property of a natural transformation $\phi$, we obtain the following commutative diagram:
\begin{center}
\begin{tikzpicture}[scale=4] 
       \coordinate (0) at (0,0); 
 	\coordinate (1) at (1,0); 
	 \coordinate (2) at (0,-1);
	 \coordinate (3) at (1,-1);
	 \draw (0) node[] {$(\Lambda^n_j)_m$};
	 \draw (1) node[] {$(M_W)_m$};
	 \draw (2) node[] {$(\Lambda^n_j)_3$};
	 \draw (3) node[] {$(M_W)_3$};
 \draw[->]  ($(0)+(0.2,0)$) to  ($(1)+(-0.2,0)$);
  \draw[->]  ($(2)+(0.2,0)$) to ($(3)+(-0.2,0)$);
  \draw[->]  ($(0)+(0,-0.2)$) to ($(2)+(0,0.2)$);
\draw[->]  ($(1)+(0,-0.2)$) to ($(3)+(0,0.2)$);
 \draw (0.5,0) node[above] {$\phi_m$};  
 \draw (0.5,-1) node[below] {$\phi_3$};
 \draw (0,-0.5) node[left] {$(\Lambda^n_j)(\langle ijkp\rangle)$};
 \draw (1,-0.5) node[right] {$M_W(\langle ijkp\rangle)$};
 \draw[<-] (0.6,-0.5) arc[start angle=0, end angle=-270, radius=0.1];
\end{tikzpicture}
\end{center}

From this diagram, we obtain
\begin{align*}
\phi^l_m(f)(i,j,k,p) 
&= \left(M_W(\langle ijkp \rangle)\left(\phi_m (f)\right)\right)(0,1,2,3) \\
&= \left( \phi_3\left(\Lambda^n_j (\langle ijkp\rangle)(f) \right) \right)(0,1,2,3) \\
&= \phi^l_3 (\langle f(i)f(j)f(k)f(p) \rangle)(0,1,2,3) .
\end{align*}
This completes the proof.
\end{proof}

By examining the proof above, we observe that the argument does not depend on the fact that we are using an inner horn \( \Lambda^n_j \). The same reasoning applies when \( \Lambda^n_j \) is replaced by the standard simplex \( \Delta^n \). For this reason, we obtain the following corollary.

\begin{corollary}
Let \( W \coloneqq (L \xrightarrow{\partial} H \xrightarrow{\partial} G, \{-,-\}) \) be a 2-crossed module. For any \( n \ge 1 \) and any standard simplex \( \Delta^n \), let \( \phi \in \operatorname{Hom}( \Delta^n, M_W) \). Then, for any \( m \ge 1 \) and any \( f \in (\Delta^n)_m \), the following identity holds:
\[
\phi_m(f) 
\coloneqq
\begin{pmatrix}
\phi^g_m(f)(-,-) \\
\phi^h_m(f)(-,-,-) \\
\phi^l_m(f)(-,-,-,-)
\end{pmatrix}
=
\begin{pmatrix}
\phi^g_1(f(-,-))(0,1) \\
\phi^h_2(f(-,-,-))(0,1,2) \\
\phi^l_3(f(-,-,-,-))(0,1,2,3)
\end{pmatrix}.
\]
\end{corollary}

By using Lemma~\ref{lemma1}, we can conclude that any $\phi \in \operatorname{Hom}(\Lambda^n_j, M_W)$ is determined by $\phi_1, \phi_2$, and $\phi_3$.
Let $\psi \in \operatorname{Hom}(\Lambda^2_1, M_W)$. Then, it is possible that $\psi^l_3(\langle 0111 \rangle)(0,1,2,3)$ occurs, by Lemma~\ref{lemma1}.
However, we want $\psi$ to be determined solely by $\psi_1$, since $\Lambda^2_1$ contains only 1-simplices.

The following lemma resolves the issue that arises when degenerate simplices are involved.

\begin{lemma}
\label{lemma2}

Let \( W \coloneqq (L \xrightarrow{\partial} H \xrightarrow{\partial} G, \{-,-\}) \) be a 2-crossed module. For any \( n \ge 1 \) and any inner horn \( \Lambda^n_j \), let \( \phi \in \operatorname{Hom}(\Lambda^n_j, M_W) \). Then for any map $ \langle a_0 \dots a_i a_{i+1}\dots a_m  \rangle \in (\Lambda^n_j)_m $ with \(1\le m \le 3 \), the following equation holds whenever $a_i = a_{i+1}$ for any $0\le i  \le m-1$:
\[
\phi_m (\langle a_0 \dots a_i a_{i+1}\dots a_m \rangle) (0,1,\dots, m) = \phi_{m-1} (\langle a_0 \dots a_i \hat{a}_{i+1}\dots a_m \rangle) (0,1,\dots, i,i,\dots,m-1)
\]
where the $\hat{a}_{i+1}$ indicates that ${a}_{i+1}$ omitted.

\end{lemma}

\begin{proof}
This follows directly from Lemma~\ref{lemma1}.
\end{proof}

Now we can prove Theorem~\ref{2cm and qusi}.

\begin{proof}[of Theorem~\ref{2cm and qusi}]
What we need to prove is that every \( \psi \in \operatorname{Hom}(\Lambda^n_j, M_W) \) can be extended to  
\( \tilde{\psi} \in \operatorname{Hom}(\Delta^n, M_W) \) for all \( n \geq 2 \) and \( 0 < j < n \).  
By using Lemmas~\ref{lemma1} and~\ref{lemma2}, and properties (d)–(e) of \( M_W \),  
to determine \( \tilde{\psi} \in \operatorname{Hom}(\Delta^n, M_W) \),  
it suffices to construct  
$
\tilde{\psi}_2(\langle ij \rangle)(0,1), \ 
\tilde{\psi}_3(\langle ijk \rangle)(0,1,2), \quad \text{and} \ 
\tilde{\psi}_4(\langle ijkp \rangle)(0,1,2,3)
$
for each \( 0 \leq i \leq j \leq k \leq p \leq n \).  
There are four cases to consider.

\begin{enumerate}
\item ${\rm Hom}(\Lambda^2_1, M_W)$ to ${\rm Hom}(\Delta^2, M_W)$

The images of $\Lambda^2_1$  and $\Delta^2$ are illstrated in Figrue~\ref{fig:l^2_1}.

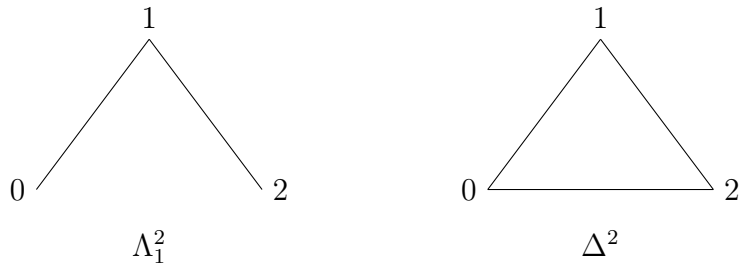
\begin{figure}[htbp]
\begin{center}
\begin{tikzpicture}[scale=2]
\label{fig:l^2_1}
\coordinate[label=left:0]  (0) at (0,0);
\coordinate[label=above:1]  (1) at (0.75,1);
\coordinate[label=right:2]  (2) at (1.5,0);
\draw (0) -- (1);
\draw (1) -- (2);
\draw (0.75,-0.2) node [below] {$\Lambda^2_1$};

\begin{scope}[xshift=3.0cm]
\coordinate[label=left:0]  (0) at (0,0);
\coordinate[label=above:1]  (1) at (0.75,1);
\coordinate[label=right:2]  (2) at (1.5,0);
\draw (0) -- (1);
\draw (1) -- (2);
\draw (0) -- (2);
\draw (0.75,-0.2) node [below] {$\Delta^2$};
\end{scope}

\end{tikzpicture}
\end{center}
\caption{$\Lambda^2_1$ and $\Delta^2$}
\label{l^2_1}
\end{figure}

For any \( \psi \in \operatorname{Hom}(\Lambda^2_1, M_W) \), the map \( \psi \) is determined by  
\( \psi_1(\langle 01 \rangle)(0,1) \) and \( \psi_1(\langle 12 \rangle)(0,1) \).
To extend \( \psi \) to \( \tilde{\psi} \in \operatorname{Hom}(\Delta^2, M_W) \),  
it suffices to construct  
\( \tilde{\psi}_1(\langle 02 \rangle)(0,1) \) and \( \tilde{\psi}_2(\langle 012 \rangle)(0,1,2) \).
We construct these as follows:  
\begin{align*}
\tilde{\psi}_1(\langle 01 \rangle)(0,1) &\coloneqq {\psi}_1(\langle 01 \rangle)(0,1)\\
\tilde{\psi}_1(\langle 12 \rangle)(0,1) &\coloneqq {\psi}_1(\langle 12 \rangle)(0,1)\\
\tilde{\psi}_1(\langle 02 \rangle)(0,1) &\coloneqq \psi_1(\langle 12\rangle)(0,1)\cdot \psi_1(\langle 01\rangle)(0,1)\\
\tilde{\psi}_2(\langle 012 \rangle)(0,1,2) &\coloneqq e_H
.
\end{align*}
It is easy to see that \( \tilde{\psi} \) is an element of \( \operatorname{Hom}(\Delta^2, M_W) \).

\item ${\rm Hom}(\Lambda^3_i, M_W)$ to ${\rm Hom}(\Delta^3, M_W)$ for $1 \le i \le 2$

We will prove the case \( i = 2 \). The case \( i = 1 \) can be handled in a similar manner.
By comparing \( \Lambda^3_2 \) and \( \Delta^3 \), the difference lies in the presence of the 2-simplex \( \langle 013 \rangle \) and the 3-simplex \( \langle 0123 \rangle \).  
To extend a map \( \psi \in \operatorname{Hom}(\Lambda^3_2, M_W) \) to  
\( \tilde{\psi} \in \operatorname{Hom}(\Delta^3, M_W) \),  
we need to construct $\tilde{\psi}_2(\langle 013 \rangle)(0,1,2)$ and $\tilde{\psi}_3(\langle 0123 \rangle)(0,1,2,3)$. We construct them as follows:
\begin{align*}
\tilde{\psi}_1(\langle ij \rangle)(0,1) &\coloneqq {\psi}_1(\langle ij \rangle)(0,1) \ {\rm\ for\ all\ 0\le i\le j \le 3}\\
\tilde{\psi}_2(\langle 012 \rangle)(0,1,2) &\coloneqq {\psi}_2(\langle 012 \rangle)(0,1,2) \\
\tilde{\psi}_2(\langle 123 \rangle)(0,1,2) &\coloneqq {\psi}_2(\langle 123 \rangle)(0,1,2) \\
\tilde{\psi}_2(\langle 023 \rangle)(0,1,2) &\coloneqq {\psi}_2(\langle 023 \rangle)(0,1,2) \\
\tilde{\psi}_2(\langle 013 \rangle)(0,1,2) &\coloneqq h_{023}{^{g_{23}}}h_{012}h^{-1}_{234}\\
\tilde{\psi}_3(\langle 0123 \rangle)(0,1,2,3) &\coloneqq e_L.
\end{align*}


Here, we write \( g_{ij} \) and \( h_{ijk} \) as follows:
\begin{align*}
g_{ij}&\coloneqq \psi_1(\langle ij \rangle)(0,1)\\
h_{ijk}&\coloneqq \psi_2(\langle ijk \rangle)(0,1,2).
\end{align*}

\item ${\rm Hom}(\Lambda^4_i, M_W)$ to ${\rm Hom}(\Delta^4, M_W)$ for $1 \le i \le 3$

We will prove the case \( i = 3 \). The cases \( i = 1, 2 \) can be handled in a similar manner.
The difference between \( \Lambda^4_3 \) and \( \Delta^4 \) is the presence of the 3-simplex \( \langle 0124 \rangle \) and the 4-simplex \( \langle 01234 \rangle \).
By the properties of \( M_W \), to extend \( \psi \in \operatorname{Hom}(\Lambda^4_3, M_W) \) to  
\( \tilde{\psi} \in \operatorname{Hom}(\Delta^4, M_W) \),  
it suffices to construct only  $\tilde{\psi}_3(\langle 0124 \rangle)(0,1,2,3)$.
The following construction will be sufficient:  
\begin{align*}
\tilde{\psi}_1(\langle ij \rangle)(0,1) &\coloneqq {\psi}_1(\langle ij \rangle)(0,1)  \text{ for all } 0\le i\le j \le 3,\\
\tilde{\psi}_2(\langle ijk \rangle)(0,1,2) &\coloneqq {\psi}_2(\langle ijk \rangle)(0,1,2)  \text{ for all }  0\le i\le j \le k\le 4,\\
\tilde{\psi}_3(\langle 1234 \rangle)(0,1,2) &\coloneqq {\psi}_3(\langle 1234 \rangle)(0,1,2),\\
\tilde{\psi}_3(\langle 0234 \rangle)(0,1,2) &\coloneqq {\psi}_3(\langle 0234 \rangle)(0,1,2),\\
\tilde{\psi}_3(\langle 0134 \rangle)(0,1,2) &\coloneqq {\psi}_3(\langle 0134 \rangle)(0,1,2),\\
\tilde{\psi}_3(\langle 0123 \rangle)(0,1,2) &\coloneqq {\psi}_3(\langle 0123 \rangle)(0,1,2), \text{ and }\\
\tilde{\psi}_3(\langle 0123 \rangle)(0,1,2) &\coloneqq {}^{h_{014}} l_{1234} \, l_{0134} \, {}^{h_{034}}\left( {}^{g_{34}} l_{0123}\right) \, l^{-1}_{0234} \, {}^{h_{024}} \{ h_{234}, {}^{g_{34} g_{23}} h_{012} \}.
\end{align*}
Here, we write \( g_{ij} \), \( h_{ijk} \), and \( l_{ijkm} \)  as follows: 
\begin{align*}
g_{ij}&\coloneqq \psi_1(\langle ij \rangle)(0,1),\\
h_{ijk}&\coloneqq \psi_2(\langle ijk \rangle)(0,1,2), \text{ and}\\
l_{ijkm}&\coloneqq \psi_3(\langle ijkm \rangle)(0,1,2,3).
\end{align*}

\item ${\rm Hom}(\Lambda^n_i, M_W)$ to ${\rm Hom}(\Delta^n, M_W)$ for $5\le n$ and $1 \le i \le n-1$

The difference between \( \Lambda^n_i \) and \( \Delta^n \) is the presence or absence of the \((n-1)\)-simplex \( \langle 01 \dots \hat{i} \dots n \rangle \) and the \( n \)-simplex \( \langle 01 \dots n \rangle \).
By the properties of \( M_W \), any \( \psi \in \operatorname{Hom}(\Lambda^n_i, M_W) \) can be extended directly to  
\( \tilde{\psi} \in \operatorname{Hom}(\Delta^n, M_W) \) as follows:
\begin{align*}
\tilde{\psi}_1(\langle ij \rangle)(0,1) &\coloneqq {\psi}_1(\langle ij \rangle)(0,1)  \text{ for all }  0\le i\le j \le 3, \\
\tilde{\psi}_2(\langle ijk \rangle)(0,1,2) &\coloneqq {\psi}_2(\langle ijk \rangle)(0,1,2)  \text{ for  all } 0\le i\le j \le k\le 4, \text{ and}\\
\tilde{\psi}_3(\langle ijkm \rangle)(0,1,2,3) &\coloneqq {\psi}_3(\langle ijkm \rangle)(0,1,2,3) \text{ for all } 0\le i\le j \le k\le m \le n.
\end{align*}
\end{enumerate}
This completes the proof.
\end{proof}

\section{3-crossed modules and simplicial sets}\label{sec-3-crossed module and sset}

In this section, we define a 3-crossed module. After defining the notion of a 3-crossed module, we construct a simplicial set associated with it and prove that this simplicial set is a quasi-category.

A 3-crossed module involves four distinct groups, four types of actions, and six types of liftings.

\begin{dfn}{(3-crossed module)}
\label{d3cm}
Let $G,H,L,M$ be groups. A {\bf 3-crossed module} is a complex of groups:
\[
 M\xrightarrow{\partial}L \xrightarrow{\partial} H \xrightarrow{\partial} G 
\]
together with the following structures:
\begin{itemize}
  \item a left action of $G$ by automorphisms on  $H$, $L$ and $M$ (with the action on $G$ given by conjugation),
  \item a left action of $H$ by automorphisms on $L$ and $M$ (with the action on $H$ also given by conjugation),
  \item a left action of $L$ by automorphisms on $M$(with the action on $L$ also given by conjugation),
  \item a function $\{-,-\} : H \times H \rightarrow L$ (called the Peiffer lifting),
  \item a function $\{-,-,-\} : H \times H \times H  \rightarrow M$ (called the Left-Homanian),
  \item a function $\{-,-,-\}' : H \times H \times H  \rightarrow M$ (called the Right-Homanian),
  \item a function $\{-,-\}_{HL} : H \times L \rightarrow M$ (called the HL-Peiffer lifting),
  \item a function $\{-,-\}'_{HL} : H \times L \rightarrow M$ (called the HL'-Peiffer lifting),  and
  \item a function $\{-,-\}_{LL} : L \times L \rightarrow M$ (called the LL-Peiffer lifting),         
\end{itemize}
which are required to satisfy the following axioms:
\begin{enumerate}
  \item Each $\partial$ is homomorphism.
  \item  $\partial\circ \partial =0$.
  \item For each $g \in G$, $h \in H$, $l \in L$, and $m \in M$,
  \begin{align*}
\left\{
    \begin{aligned}
    & \partial({}^gh)= {{}^g}\partial(h) \\
    & \partial({}^gl)= {{}^g}\partial(l)  \\
    & \partial({}^gm)= {{}^g}\partial(m) .
    \end{aligned}
\right.
\end{align*}
  \item For each $h \in H$, $l \in L$, and $m \in M$,
  \begin{align*}
\left\{
    \begin{aligned}
    & \partial({}^hl)= {{}^h}\partial(l)  \\
    & \partial({}^hm)= {{}^h}\partial(m) .
    \end{aligned}
\right.
\end{align*}
  \item For each $l \in L$ and $m \in M$,
\[
\partial({}^lm)= {{}^l}\partial(m).
\]
  \item For each $g\in G$ and $h_1,h_2 \in H$,
  \[
   ^g \{h_2,h_1\} = \{^g h_2,^g h_1\} .
  \]
  \item For each $g\in G$ and $h_1,h_2,h_3 \in H$,
  \begin{align*}
\left\{
    \begin{aligned}
   &{}^g \{h_3,h_2,h_1\} = \{{}^g h_3,{}^g h_2, {}^g h_1\}\\
   &{}^g \{h_3,h_2,h_1\}' = \{{}^g h_3,{}^g h_2, {}^g h_1\}'. 
    \end{aligned}
\right.
\end{align*}
  \item For each $g\in G$ and $l_1,l_2 \in L$,
  \begin{align*}
{}^g \{l_2,l_1\}_{LL} = \{{}^g l_2,{}^g l_1\}_{LL}.
\end{align*}
 \item For each $h\in H$ and $l_1,l_2 \in L$,
  \begin{align*}
{}^h \{l_2,l_1\}_{LL} = \{{}^h l_2,{}^h l_1\}_{LL}.
\end{align*}

\item For each $l\in L$ and $m \in M$,
 \[
 \act{\del l}m = \act{l}m\{l,\partial m^{-1}\}_{LL}  .
 \] 
\label{act_u_ll_dellm}
\item For each $l\in L$ and $m \in M$,
 \[
 \act{l}m = m\{\partial m^{-1}, l\}_{LL} .
 \] 
\label{act_u_ll_lm}
  \item For each $m,m'\in M$,
  \[
  \act{\partial m}m' = mm'm^{-1} .
  \]
\label{act_delm_m_3cm}

  \item For each $h\in H$ and $l \in L$,
  \[
  \partial\{h,l\}_{HL}\act{h}l = l\{\partial l^{-1},h\}.
  \]
  \label{dhl}
  
    \item For each $h\in H$ and $l \in L$,
  \[
  \partial\{h,l\}'_{HL}\act{\partial h}l = \act{h}l\{h,\partial l^{-1}\}.
  \]
\label{dhlp}

  \item For each $g\in H$,$h\in H$, and $l \in L$,
  \begin{align*}
\left\{
    \begin{aligned}
{}^g \{h,l\}_{HL} &= \{{}^g h, {}^g l\}_{HL}\\
{}^g \{h,l\}_{HL}' &= \{{}^g h, {}^g l\}_{HL}'.
    \end{aligned}
\right.
\end{align*}

  \item For each $h\in H$ and $m \in M$,
 \[
 \act{h}m = m\{h,\partial m^{-1}\}_{HL}  .
 \] 
\label{act_h_m_3cm}

 \item For each $l,l'\in L$,
 \[
\partial\{l,l'\}_{LL}\act{\partial l}l' = ll'l^{-1}.
 \]
 \label{dll}
  \item  For each $h_1,h_2, h_3 \in L$,
  \[
  \{l_3l_2,l_1\}_{LL}={^{l_3}}\{l_2, l_1\}_{LL}\{l_3, {^{\partial l_2}}l_1 \}_{LL}.
  \] 
  \label{lll}
  \item For each $l_1,l_2, l_3 \in L$,
  \[
  \{l_3,l_2 l_1\}_{LL}=\{l_3, l_2\}_{LL} {^{(^{\partial l_3} l_2)}}\{l_3, l_1\}_{LL}.
  \] 
  \label{rll}
  \item For each $h\in H$ and $m \in M$,
 \[
 \act{\partial h}m = \act{h}m\{h,\partial m^{-1}\}'_{HL}  .
 \] 
 \label{dhmnl}
  \item For each $h_1,h_2,h_3 \in H$,
\[
\{h_3h_2,h_1\} = \partial\{h_3,h_2,h_1\}\ {{}^{h_3}\{h_2,h_1\}}\{h_3, {}^{\partial h_2}h_1\}.
\]  
\label{dhmnr}
  \item For each $h_1,h_2,h_3 \in H$,
\[
\{h_3,h_2h_1\} = \partial\{h_3,h_2,h_1\}' \{h_3,h_2\}{}^{\left({}^{\partial h_3}h_2\right) }\{h_3, h_1\}.
\]
  \item For each $h_1,h_2 \in H$ and $l \in L$,
\[
\{h_2h_1,l\}_{HL} = \act{l}\{\partial l^{-1},h_2,h_1\}'\{h_2,l\}_{HL} \act{h_2}\{h_1,l\}_{HL}  .
\]  
\label{rl:hhl}

  \item For each $h \in H$ and $l_1,l_2 \in L$,
\[
\{h,l_2l_1\}_{HL} = \act{l_2l_1}\{\partial l_1^{-1},\partial l_2^{-1},h\}\act{l_2l_1}\{l_1^{-1},\{\partial l_2^{-1},h\}\}_{LL} ^{-1}\{h,l_2\}_{HL} \act{\left(\act{h}l_2\right)}\{h,l_1\}_{HL}  .
\] 
  \item For each $h_1,h_2 \in H$ and $l \in L$,
\[
\{h_2h_1,l\}'_{HL} = \act{\left(\act{h_2h_1}l\right)}\{h_2,h_1,\partial l^{-1}\}\act{h_2}\{h_1,l\}'_{HL}\{h_2,\act{\partial h_1}l\}'_{HL} .
\]   
  \item For each $h \in H$ and $l_1,l_2 \in L$,
\[
\{h,l_2l_1\}'_{HL} = \act{\left(\act{h}(l_2l_1)\right)}\{h,\partial l_1^{-1},\partial l_2^{-1}\}'\act{\left(\act{h}l_2\right)}\{h,l_1\}'_{HL}\act{\left(\act{h}l_2\act{\partial h}l_1\right)}\{\act{\partial h}l_1^{-1},\{h,\partial l_2^{-1}\}\}_{LL}^{-1}\{h,l_1\}'_{HL}  .
\] 
\label{rrHL}
 \item For each $h_1,h_2,h_3,h_4 \in H$,
 \[
 \{h_4h_3,h_2,h_1\}\act{\left(\act{h_4h_3}\{h_2,h_1\}\right)}\{h_4,h_3,\act{\partial h_2}h_1\} = \{h_4,h_3h_2,h_1\}\act{h_4}\{h_3,h_2,h_1\} .
 \]
  \label{lhmnr}
  
  \item For each $h_1,h_2,h_3,h_4 \in H$,
 \[
\{h_4,h_3h_2,h_1\}'\{h_4,h_3,h_2\}' = \{h_4,h_3,h_2h_1\}'\act{\{h_4,h_3\}}(\act{\left(\act{\partial h_4}h_3\right)}\{h_4,h_2,h_1\}') .
 \]
 \label{rhmnr}

    \item For each $h_1,h_2,h_3,h_4 \in H$,
 \begin{align*}
&\{h_4,h_3,h_2h_1\}\act{\left(\act{h_4}\{h_3,h_2h_1\}\right)}\{h_4,\act{\partial h_3}h_2,\act{\partial h_3}h_1\}'\act{h_4}\{h_3,h_2,h_1\}' \\
  &= \{h_4h_3,h_2,h_1\}'\act{\{h_4h_3,h_2\}}(\act{\act{\partial (h_4h_3)}h_2}\{h_4,h_3,h_1\})\{h_4,h_3,h_2\} \\
   &\hspace{1em}\times\act{\left(\act{h_4}\{h_3,h_2\}\right)}\{\{h_4,\act{\partial h_3}h_2\},\act{\left(\act{\partial(h_4h_3)}h_2h_4\right)}\{h_3,h_1\}\}_{LL}.
\end{align*}
 \label{lrhmnr}
 
  \item For each $h_1,h_2,h_3 \in H$,
  \label{yanBt}
 \begin{align}
  \nonumber
 &\act{\act{h_3}\{h_2,h_1\}}(\{h_3,\act{\partial h_2}h_1,h_2\}'^{-1}\{h_3,\partial\{h_2,h_1\}^{-1},h_2h_1\}')\{h_3,\{h_2,h_1\}\}' \\
  \nonumber
  \hspace{1em} &\{\act{\partial h_3}\{h_2,h_1\},\act{\partial\act{\partial h_3}\{h_2,h_1\}^{-1}}\{h_3,h_2h_1\}\}\{h_3,h_2,h_1\}' \\
  \nonumber
  &= \{h_3,h_2,h_1\}^{-1}\act{\{h_3h_2,h_1\}}\{\act{\partial(h_3h_2)}h_1,\{h_3,h_2\}\}^{-1}\{\{h_3,h_2\},\act{\partial\{h_3,h_2\}^{-1}}\{h_3h_2,h_1\}\}^{-1} \\
  \nonumber
  &\hspace{1em} \act{\{h_3,h_2\}}(\{\partial\{h_3,h_2\}^{-1},h_3h_2,h_1\}^{-1}\{\act{\partial h_3}h_2,h_3,h_1\})
\end{align}


    \item For each $l,l' \in L$,
    \[
    \act{l}\{\partial l',l^{-1}\}_{HL}\act{ll'l^{-1}l'^{-1}}\{l',l\}_{LL}= \act{\left(\act{\partial l}l'\right)}\{\partial l,l'^{-1}\}'_{HL}\{l,l'\}^{-1}_{LL}.
    \]
     \item For each $l\in L$ and $m \in M$,
    \[
    \{\partial m,l\}_{LL}\{l,\partial m\}_{LL} = \{\partial l,\partial m\}_{HL} = \{\partial l,\partial m\}'^{-1}_{HL}.
    \]
    \label{rllandhl}
        \item For each $m,m' \in M$,
    \[
    \{\partial m,\partial m'\}_{LL} = mm'm^{-1}m'^{-1}.
    \]   
\label{3cm_not_need}
\end{enumerate}
\end{dfn}

\begin{rem}
Properties \ref{rllandhl} and \ref{3cm_not_need} can be derived from the other axioms of a 3-crossed module. Specifically, property \ref{rllandhl} is obtained by applying properties \ref{act_u_ll_dellm} and \ref{act_u_ll_lm} to the term $\act{\del l}m$, along with property \ref{act_h_m_3cm} in $\act{\del l}m$ and property \ref{dhmnl} in $\act{\del \circ \del l}m$. Similarly, property \ref{3cm_not_need} follows from properties \ref{act_u_ll_lm} and \ref{act_delm_m_3cm} applied to $\act{\del m'}m$. We include these properties explicitly for the reader's convenience.
\end{rem}

There are six types of liftings. One is the normal Peiffer lifting. The LL-Peiffer lifting corresponds to the normal Peiffer lifting for the complex of groups \( M \xrightarrow{\partial} L \xrightarrow{\partial} H \).
The HL-Peiffer lifting, HL'-Peiffer lifting, and the left and right Homanian are new types of liftings that arise specifically in 3-crossed modules.
As can be seen in the 3-crossed module property~\ref{dhl}, the HL-Peiffer lifting is a twisted version of the 2-crossed module property~\ref{acthl}. According to the 3-crossed module property~\ref{dhlp}, the HL'-Peiffer lifting is a twist of the 2-crossed module property~\ref{actgl}.
The left Homanian corresponds to a twist of the relation in the 2-crossed module property~\ref{lpl}, while the right Homanian corresponds to a twist of the relation in property~\ref{rpl}.

Some of the relations in a 3-crossed module can be understood through diagrams that can be described in the language of category theory.  
See Appendix~\ref{app_dia_3cm} for details.

We investigate the behavior of liftings when one of the arguments is the unit element.
The following lemma describes this situation.
\begin{lemma}
\label{lemma_of_unit_element}
Let $G,H,L,M$ be groups, and let
$
 M\xrightarrow{\partial}L \xrightarrow{\partial} H \xrightarrow{\partial} G 
$ be 3-crossed module. Then the six types of liftings satisfy the following equations:
\begin{align}
\{h_2, h_1\}
  &= e_L
  && \text{if $h_2 = e_H$ or $h_1 = e_H$,}
  \label{eq:einpl} \\
 \{l_2, l_1\}
  &= e_M
  && \text{if $l_2 = e_L$ or $l_1 = e_L$,}
  \label{eq:einll} \\
\{h_3, h_2, h_1\}
  &= \{h_3, h_2, h_1\}' = e_M
  && \text{if $h_3 = e_H$ or $h_2 = e_H$ or $h_1 = e_H$,}
  \label{eq:einhmn} \\
\{h, l\}_{HL}
  &= \{h, l\}'_{HL} = e_M
  && \text{if $h = e_H$ or $l = e_L$.}
  \label{eq:einhl}
\end{align}
Here, $e_X$ denotes the unit element of the group $X$. 
\end{lemma}
\begin{proof}
Equation~(\ref{eq:einll}) follows from Properties \ref{lll} and \ref{rll} of a 3-crossed module applied to $\{e_L e_L, l\}_{LL}$ and $\{l, e_L e_L\}_{LL}$. Indeed,
\begin{equation}
\begin{aligned}
\{e_L e_L , l\}_{LL} &= \{e_L  , l\}_{LL}\{e_L , l\}_{LL},\\
\{l, e_L e_L \}_{LL} &= \{l, e_L\}_{LL} \{l, e_L\}_{LL}.
\end{aligned}
\label{p_einll}
\end{equation}

We next show that, for each $h\in H$,
\begin{equation}
\begin{aligned}
\{e_H, e_H, h\} = e_M,\\
\{h, e_H, e_H\}' = e_M.
\label{hyp_of_unit}
\end{aligned}
\end{equation}
Equation~(\ref{hyp_of_unit}) is obtained by computing $\act{h}(mm')$ and $\act{\del h}(mm')$ and then setting $m=m'=e_M$. We compute $\act{h}(mm')$ as follows:
\begin{align}
  \nonumber
  \act{h}(mm') &= mm'\{h,\partial (mm')^{-1}\}_{HL} \\
  \nonumber
  &= mm'\act{\partial (mm)'^{-1}}\{e_H,e_H,h\}\act{\partial (mm')^{-1}}\{\partial m,\{e_H,h\}\}_{LL}^{-1}\{h,\partial m'^{-1}\}_{HL}\act{\partial \act{h}m'^{-1}}\{h,\partial m^{-1}\}_{HL} \\
  \nonumber
  &= \{e_H,e_H,h\}\{\partial m,\{e_H,h\}\}_{LL}^{-1}mm'\{h,\partial m'^{-1}\}_{HL}\act{\partial \act{h}m'^{-1}}\{h,\partial m^{-1}\}_{HL} \\
  \label{eq:act_hmm}
  &= \{e_H,e_H,h\}\{\partial m,\{e_H,h\}\}_{LL}^{-1}m\act{h}m'\act{\partial \act{h}m'^{-1}}\{h,\partial m^{-1}\}_{HL} \\
  \nonumber
  &= \{e_H,e_H,h\}\{\partial m,\{e_H,h\}\}_{LL}^{-1}m\{h,\partial m^{-1}\}_{HL}\act{h}m' \\
  \nonumber
  &= \{e_H,e_H,h\}\{\partial m,\{e_H,h\}\}_{LL}^{-1}\act{h}m\act{h}m'.
\end{align}
Similarly, $\act{\del h}(mm')$ can be computed in an analogous manner:
\begin{align}
  \nonumber
  \act{\partial h}(mm') &= \act{h}(mm')\{h,\partial (mm')^{-1}\}_{HL} ' \\
  \nonumber
  &= \act{h}(mm')\act{\partial \act{h}(mm')^{-1}}\{h,e_H,e_H\}'\act{\partial \act{h}m'^{-1}}\{h,\partial m^{-1}\}_{HL} '\act{\partial (\act{h}m'^{-1}\act{\partial h}m^{-1})}\{\partial m,\{h,e_H\}\}_{LL} ^{-1}\{h,\partial m'^{-1}\}_{HL} ' \\
  \nonumber
  &= \{h,e_H,e_H\}'\act{h}m\{h,\partial m^{-1}\}_{HL}'\act{\partial (\act{\partial h}m^{-1})}\{\partial m,\{h,e_H\}\}_{LL}^{-1}\act{h}m'\{h,\partial m'^{-1}\}_{HL}' \\
  \label{eq:act_gmm}
  &= \{h,e_H,e_H\}'\act{\partial h}m\act{\partial (\act{\partial h}m^{-1})}\{\partial m,\{h,e_H\}\}_{LL}^{-1}\act{\partial h}m'.
\end{align}
Using Equations~(\ref{eq:act_hmm}) and~(\ref{eq:act_gmm}), together with $m=m'=e_M$ and Equation~(\ref{eq:einll}), we obtain Equation~(\ref{hyp_of_unit}).

Next, by substituting the unit element appropriately into Properties \ref{dhmnr}-\ref{rrHL} of a 3-crossed module, we obtain Equations~(\ref{eq:P11h1h})-(\ref{eq:h11h1}):
\begin{align}
  \label{eq:P11h1h}
  e_L &= \partial\{e_H,e_H,h\}\{e_H,h\}, \\
  \label{eq:Ph11h1}
  e_L &= \partial\{h,e_H,e_H\}'\{h,e_H\}, \\
  \label{eq:Pl111l}
  e_M &= \act{l}\{\partial l^{-1},e_H,e_H\}'\{e_H,l\}_{HL}, \\
  \label{eq:11hh1}
  e_M &= \{e_H,e_H,h\}\{h,e_L\}_{HL},  \\
  \label{eq:11Pl1l}
  e_M &= \act{l}\{e_H,e_H,\partial l^{-1}\}\{e_H,l\}'_{HL},  \\
  \label{eq:h11h1}
  e_M &= \{h,e_L,e_L\}'\{h,e_L\}'_{HL}. 
\end{align}
Using Equations~(\ref{hyp_of_unit}) and (\ref{eq:P11h1h})-(\ref{eq:h11h1}), we obtain Equations~(\ref{eq:einpl}) and~(\ref{eq:einhl}) . 

Now, by setting $h_3=h_2=h_1=e_H$ and $h_4=h_3=h_2=e_H$ in Proprety~\ref{lhmnr} of 3-crossed module, we obtain Equations~(\ref{2ulhmn1}) and (\ref{2ulhmn2}), respectively:
\begin{align}
\label{2ulhmn1}
\{h, e_H, e_H \} = e_M,\\ 
\label{2ulhmn2}
\{e_H, h, e_H \} = e_M.
\end{align}
Similarly, setting $h_4=h_3=h_2=e_H$ and $h_4=h_2=h_1=e_H$ in Proprety~\ref{rhmnr} of a 3-crossed module yields Equations~(\ref{2urhmn1}) and (\ref{2urhmn2}), respectively:
\begin{align}
\label{2urhmn1}
\{e_H, h, e_H \}' = e_M,\\ 
\label{2urhmn2}
\{e_H, e_H, h \}' = e_M.
\end{align}
Using Equations~(\ref{eq:einpl}), (\ref{eq:einll}), (\ref{eq:einhl}), (\ref{hyp_of_unit}), and (\ref{2ulhmn1})-(\ref{2urhmn2}) where necessary, and applying Property~\ref{lhmnr} of a 3-crossed module with $h_4 = h_3 = e_H$ and with $h_3= h_2 = e_H$, we obtain Equation~(\ref{1ulhmn1}) and (\ref{1ulhmn2}), respectively:
  \begin{align}
\label{1ulhmn1}
\{e_H, h_2, h_1 \} = e_M,\\ 
\label{1ulhmn2}
\{h_4, e_H, h_1 \} = e_M.
\end{align}

Similarly, applying Property~\ref{rhmnr} of a 3-crossed module with $h_3 = h_2 = e_H$ and with $h_2= h_1 = e_H$, we obtain Equation~(\ref{1urhmn1}) and (\ref{1urhmn2}), respectively:
\begin{align}
\label{1urhmn1}
\{h_4, e_H, h_1 \}' = e_M,\\ 
\label{1urhmn2}
\{h_4, h_3, e_H \}' = e_M.
\end{align}

Finaly, applying Property~\ref{lrhmnr} of a 3-crossed module with $h_2 = h_1 = e_H$ and with $h_4= h_3 = e_H$, we obtain Equation~(\ref{1ulhmn3}) and (\ref{1urhmn3}), respectively:
\begin{align}
\label{1ulhmn3}
\{h_4, h_3, e_H \} = e_M,\\ 
\label{1urhmn3}
\{e_H, h_2, h_1 \}' = e_M.
\end{align}
This completes the proof.

\end{proof}

These twists naturally appear in the simplicial set associated with a 3-crossed module.  
We now construct the simplicial set associated with a 3-crossed module. This can be done by adding additional structure to the simplicial set associated with a 2-crossed module.

Let \( T \coloneqq  (M \xrightarrow{\partial} L \xrightarrow{\partial} H \xrightarrow{\partial} G) \) be a 3-crossed module.  
The simplicial set \( M_T \) associated with this 3-crossed module consists of the following data:

\begin{itemize}
    \item For each $[n] \in \Delta$, the set $M_W([n])$ consists of quadruples 
    \[ (g(-,-), h(-,-,-), l(-,-,-,-),m(-,-,-,-,-)) ,\]
    where:
    \begin{enumerate}
        \item $g: [n] \times_{\le} [n] \to G$ is a function assigning each ordered pair in $[n]$ to an element of $G$,
        \item $h: [n] \times_{\le} [n] \times_{\le} [n] \to H$ is a function assigning each ordered triple in $[n]$ to an element of $H$,  
        \item $l: [n] \times_{\le} [n] \times_{\le} [n] \times_{\le} [n] \to L$ is a function assigning each ordered quadruple in $[n]$ to an element of $L$, and
        \item $m: [n] \times_{\le} [n] \times_{\le} [n] \times_{\le} [n] \times_{\le} [n] \to M$ is a function assigning each ordered quintuple in $[n]$ to an element of $M$,
    \end{enumerate}
    satisfying the following conditions:

\begin{enumerate}[label=(\alph*)]
\item for each $i \le j \le k$ in $[n]$, we have the idenntity 
\begin{equation}
g_{ik} = \partial(h_{ijk})g_{jk}g_{ij},
\label{r2-simplex m}
\end{equation}
\item for each $i \le j \le k \le p$ in $[n]$, we have the identity 
\begin{equation}
\partial (l_{ijkp})  h_{ikp}{^{g_{kp}}}h_{ijk} = h_{ijp}h_{jkp},
\label{r3-simplex m}
\end{equation}
\item for each $i \le j \le k \le p \le q$ in $[n]$, we have the idenntity 
\begin{equation}
 \partial(m_{ijkpq}) {^{h_{ijq}}}l_{jkpq}l_{ijpq}{^{h_{ipq}}}\left({^{g_{pq}}}l_{ijkp}\right) = l_{ijkq}{^{h_{ikq}}}\{h_{kpq},{^{g_{pq}g_{kq}}}h_{ijk} \}^{-1}l_{ikpq} ,
\label{r4-simplex m}
\end{equation}
\item for each $i \le j \le k  \le p \le q\le x$ in $[n]$, we have the idenntity 

\label{r5-simplex m}
\begin{equation}
\label{r5-simplex m:eq}
\begin{aligned}
   \act{L_D L_E}M_5 \cdot \act{L_D L_E L_F}M_4 &\cdot M_3 \cdot \act{L_A}M_2 \cdot  \act{L_A L_B L_C} M_1\\
   =
   \act{}M'_{13}
   &\cdot \act{L_K L_N}M'_{12}
   \cdot \act{L_K }M'_{11}
   \cdot \act{L_K L_L}M'_{10}
   \cdot \act{L_K L_L L_M}M'_9\\
  & \cdot \act{}M'_8 
   \cdot \act{L_G L'_I}M'_7
   \cdot \act{L_G L'_I}M'_6
   \cdot \act{L_G }M'_5\\
   &\cdot \act{L_G L_H L_I L_J}M'_4
   \cdot \act{L_G L_H}M'_3
   \cdot \act{}M'_2
   \cdot \act{L_A L_B}M'_1.
   \end{aligned}
   \end{equation}

Here each term will be written as follows:

\noindent
\begin{align*}
     L_A &= \act{h_{ijx}h_{jkx}}l_{kpqx}\\
   L_B &= \act{h_{ijx}}l_{jkqx}\\
   L_C &=l_{ijqx}\\
   L_D &=l_{ijkx}\\
   L_E &=\act{h_{ikx}(\act{g_{kx}}h_{ijk})}l_{kpqx}\\
   L_F &=\act{h_{ikq}}\{h_{kqx}, \act{g_{qx}g_{kq}}h_{ijk} \}^{-1}\\
   M_j &= \act{h_{iqx}}\left(\act{g_{qx}}m_{ijkpq} \right)\\
   M_k &= m_{ijkqx}\\
   M_p &= \{l_{ijkx}, \act{h_{ikx}(\act{g_{kx}}h_{ijk})}l_{kpqx} \}_{L,L}\\
   M_q &= \left\{l_{ikqx},\act{h_{iqx}(\act{g_{qx}}h_{ikq})}\{\act{g_{qx}}h_{kpq},\act{g_{qx}g_{pq}g_{kp}}h_{ijk} \} \right\}^{-1}_{L,L}\\
   M_x &=\act{h_{ikx}}\left(\act{\{h_{kqx},\act{g_{qx}g_{kq}}h_{ijk} \}^{-1}(\act{h_{kqx}}\{\act{g_{qx}}h_{kpq},\act{g_{qx}g_{pq}g_{kp}}h_{ijk} \}^{-1})} \{h_{kqx}, \act{g_{qx}}h_{kpq}, \act{g_{qx}g_{pq}g_{kp}}h_{ijk} \}_L^{-1}\right)
    \end{align*}
   
   \begin{align*}
   L_G &= \act{h_{ijx}}l_{jkpx}\\
   L_H &= \act{h_{ijx}h_{jpx}}\{h_{pqx},\act{g_{qx}g_{pq}}h_{jkp} \}^{-1}\\
   L_I &=L'_I =l_{ijpx}\\
   L_J &=\act{h_{ipx}}\{h_{pqx}, \act{g_{qx}g_{pq}}h_{ijp}\}^{-1} \\
   L_K &=l_{ijkx} \\
   L_L &=\act{h_{ikx}}\{h_{kpx}, \act{g_{px}g_{kp}}h_{ijk} \}^{-1}\\
   L_M &= l_{ikpx} \\
   L_N &= \act{h_{ikx}}\{h_{kpx}h_{pqx}, \act{g_{qx}g_{px}g_{kp}}h_{ijk} \}^{-1}\\
   M'_j &= \{l_{ijqx}, \act{h_{iqx}(\act{g_{qx}h_{ijq}})}(\act{g_{qx}}l_{jkpq}) \}^{-1}_{L,L}\\
   M'_k &= m_{jkpqx}\\
   M'_p &= m_{ijpqx}\\
   M'_q &= \{l_{ipqx}, \act{h_{iqx}(\act{g_{qx}h_{ipq}})}(\act{g_{qx}g_{pq}}l_{ijkp}) \}^{-1}_{L,L}\\
   M'_x &=\{l_{ijpx}, \act{h_{ipx}(\act{g_{px}h_{ijq}})}\{h_{pqx} \act{g_{qx}g_{pq} }h_{jkp}  \}\}_{L,L} \\
   M'_6 &=\act{h_{ipx}}\left(\act{(\act{(\act{g_{px}}h_{ijp})}\{h_{pqx}, \act{g_{qx}g_{pq}}h_{jkp} \}^{-1}\{h_{pqx}, \act{g_{qx}g_{pq}}h_{ijp} \}^{-1} ) }\{h_{pqx}, \act{g_{qx}g_{pq}}h_{ijp},\act{g_{qx}g_{pq}}h_{jkp} \}'^{-1}_{R}  \right)\\
   M'_7 &= \act{  (\act{g_{px}}l_{ijkp}\{h_{pqx}, \act{g_{qx}g_{pq}}(h_{ikp}\act{g_{kp}}h_{ijk}) \}^{-1})      }
  \left(  \act{(\act{h_{pqx}}(\act{g_{qx}g_{pq}}l_{ijkp}))}\{h_{pqx}, \act{g_{qx}g_{pq}}(\del l_{ijkp}),\act{g_{qx}g_{pq}}(h_{ikp}\act{g_{kp}}h_{ijk}) \}\right) \\
  &\ \ \ \ \cdot \act{  (\act{g_{px}}l_{ijkp}\{h_{pqx}, \act{g_{qx}g_{pq}}(h_{ikp}\act{g_{kp}}h_{ijk}) \}^{-1})      }\left(\{h_{pqx}, \act{g_{qx}g_{pq}}l_{ijkp}^{-1}\}'_{HL}
  \right)    \\
  &\ \ \ \ \cdot \{ \act{g_{px}}l_{ijkp}, \{h_{pqx}, \act{g_{qx}g_{pq}}(h_{ikp}\act{g_{kp}}h_{ijk}) \}^{-1}\}_{L,L}\\
     M'_8 &=m_{ijkpx}\\
   M'_9 &= \act{h_{ipx}}\left(\act{(\act{(\act{g_{px}}h_{ikp})}\{h_{pqx}, \act{g_{qx}g_{pq}g_{kp}}h_{ijk} \}^{-1}\{h_{pqx}, \act{g_{qx}g_{pq}}h_{ikp} \}^{-1} ) }\{h_{pqx}, \act{g_{qx}g_{pq}}h_{ikp},\act{g_{qx}g_{pq}g_{kp}}h_{ijk} \}'_{R}  \right)\\
     M'_ji &=\{l_{ikpx}, \act{h_{ipx}(\act{g_{px}}h_{ikp})}\{h_{pqx}, \act{g_{qx}g_{pq}g_{kp}}h_{ijk}\}^{-1} \}_{L,L}^{-1}\\
        M'_jj &=\act{h_{ikx}}\left(\act{\{h_{kpx},\act{g_{px}g_{kp}}h_{ijk} \}^{-1}(\act{h_{kpx}}\{h_{pqx},\act{g_{qx}g_{pq}g_{kp}}h_{ijk} \}^{-1})} \{h_{kpx}, h_{pqx}, \act{g_{qx}g_{pq}g_{kp}}h_{ijk} \}_L^{-1}\right)\\
           M'_jk &=m_{ikpqx}^{-1}\\
           M'_jp &= \act{ \act{(\act{g_{px}}h_{ijk})}l_{kpqx}\{h_{kqx}(\act{g_{qx}}h_{kpq}),\act{g_{qx}g_{pq}g_{kp}}h_{ijk}\}^{-1}l_{kpqx}^{-1}}\{\del l_{kpqx}, h_{kqx}\act{g_{qx}}h_{kpq}, \act{g_{qx}g_{pq}g_{kp}}h_{ijk} \} \\
            &\ \ \ \ \cdot \act{\left(\act{(\act{g_{px}}h_{ijk}) }l_{kpqx}\right)}\left( \act{l_{kpqx}^{-1}}\{l_{kpqx},\{h_{kqx}(\act{g_{qx}}h_{kpq}),\act{g_{qx}g_{pq}g_{kp}}h_{ijkp}  \} \}_{L,L}\cdot \{ \act{g_{kx}}h_{ijk}, l_{kpqx}^{-1}\}_{HL}   \right)
    \end{align*}

 \item For each $i \in [n]$, we have $g_{ii} = e_G$,
        
 \item For each $i \le j$ in $[n]$, we have $h_{iij} = h_{ijj} = e_H$,
        
 \item For each $i \le j \le k$ in $[n]$, we have $l_{iijk} = l_{ijjk} = l_{ijkk} = e_L$.
 
 \item For each $i \le j \le k \le p$ in $[n]$, we have $m_{iijkp} = m_{ijjkp} = m_{ijkkp} = m_{ijkpp}= e_m$.
    \end{enumerate}
    
  Here, we abbreviate $g(i,j)$, $h(i,j,k)$, and $l(i,j,k,m)$ as $g_{ij}$, $h_{ijk}$, and $l_{ijkm}$, respectively.  If you want to understand the reasoning behind equations~(\ref{r2-simplex m})–(\ref{r5-simplex m:eq}), see Appendix~\ref{cocycle} for details.
    
\item For a simplicial operater $\delta : [m] \rightarrow [n]$, $M_T(\delta): (M_T)_n \rightarrow (M_T)_m$ is defined by: for each $(g,h,l,m) \in (M_T)_n$, $M_T(\delta)$,
\[
M_T(\delta)
\begin{pmatrix}
g(-,-)\\
h(-,-,-)\\
l(-,-,-,-)\\
m(-,-,-,-,-)\\
\end{pmatrix}
\coloneqq
\begin{pmatrix}
g(\delta(-),\delta(-))\\
h(\delta(-),\delta(-),\delta(-))\\
l(\delta(-),\delta(-),\delta(-),\delta(-))\\
m(\delta(-),\delta(-),\delta(-),\delta(-),\delta(-))
\end{pmatrix}
.\]
\end{itemize}

The simplicial set associated with a 3-crossed module is also a quasi-category.  
The proof is similar to the proof for the simplicial set associated with a 2-crossed module.
With a slight extension, the argument of Lemma~\ref{lemma1} also applies to \( M_T \).

\begin{lemma}
\label{lemma1_3cm}
Let \( T \coloneqq (M \xrightarrow{\partial} L \xrightarrow{\partial} H \xrightarrow{\partial} G \) be a 3-crossed module. For any \( n \ge 1 \) and any inner horn \( \Lambda^n_j \), let \( \phi \in \operatorname{Hom}(\Lambda^n_j, M_W) \). Then, for any \( k \ge 1 \) and any \( f \in (\Lambda^n_j)_k \), the following identity holds:
\[
\phi_k(f) 
\coloneqq
\begin{pmatrix}
\phi^g_k(f)(-,-) \\
\phi^h_k(f)(-,-,-) \\
\phi^l_k(f)(-,-,-,-)\\
\phi^m_k(f)(-,-,-,-,-)
\end{pmatrix}
=
\begin{pmatrix}
\phi^g_1(f(-,-))(0,1) \\
\phi^h_2(f(-,-,-))(0,1,2) \\
\phi^l_3(f(-,-,-,-))(0,1,2,3)\\
\phi^m_4(f(-,-,-,-))(0,1,2,3,4)
\end{pmatrix}.
\]
\end{lemma}
\begin{proof}
The proof is similar to that of Lemma~\ref{lemma1}.
\end{proof}

By using Lemmas~\ref{lemma2} and~\ref{lemma1_3cm}, we can prove $M_T$ is quasi-category in similar way of $M_S$.

\begin{theorem}
Let \( T \coloneqq (M \xrightarrow{\partial} L \xrightarrow{\partial} H \xrightarrow{\partial} G, \{-,-\}) \) be a 3-crossed module. Then the simplicial set $M_T$ is a quasi-category.
\end{theorem}
\begin{proof}
The proof is essentially the same as that of Theorem~\ref{2cm and qusi}.
\end{proof}

\section{Example of 3-crossed module}

\subsection{Construction from a Moore complex}
\label{EX of 3-crossed module}

By \cite{AkcaPak2010}, we can construct a 2-crossed module from a simplicial group with Moore complex of length 2. In this section, we show that a 3-crossed module can similarly be constructed from a simplicial group with a Moore complex of length 3, which can be viewed as a natural extension of the 2-crossed module case.

\begin{dfn}
Let X be a simplicial group. Simplicial group with Moore complex of length 3 is defined by the following chain complex:
\[
XN_3/{\operatorname{Im} (Xd^4_4)} \xrightarrow{Xd^3_3} XN_2 \xrightarrow{Xd^2_2} XN_1 \xrightarrow{Xd^1_1} XN_0 .
\]
Here, we define $XN_n = \displaystyle\bigcap_{n-1}^{i=0} \operatorname{Ker}(Xd_i^n)$ and ${\rm Im}(Xd^4_4) \coloneqq Xd^4_4(XN_4)$.
\end{dfn}
 
 We now begin the construction of a 3-crossed module from a simplicial group with Moore complex of length 3. To achieve this, we need to define all the required structures: the group actions, the Peiffer lifting, the HL-Peiffer lifting, the HL'-Peiffer lifting, the LL-Peiffer lifting, the Left-Homanian, and the right-Homanian. 
 
\begin{theorem}
 Let X be a simplicial group, and suppose its Moore complex of length 3 is given by the following:
 \[
 M\coloneqq XN_3/{\operatorname{Im} (Xd^4_4)} \xrightarrow{\partial \coloneqq Xd^3_3} L\coloneqq XN_2  \xrightarrow{\partial \coloneqq Xd^2_2} H\coloneqq XN_1 \xrightarrow{\partial \coloneqq Xd^1_1}  G \coloneqq XN_0.
 \]

$M \xrightarrow{\partial} L \xrightarrow{\partial} H \xrightarrow{\partial} G$ gives rise to a 3-crossed module with the following data.

\begin{enumerate}[label=(\roman*)]

\item For each $g \in G$,$h \in H$,$l \in L$, and $m \in M$, action will be determind by following:
\begin{align*}
\act{g}h&\coloneqq Xs^0_0(g)\cdot h\cdot Xs^0_0(g^{-1}) ,\\
\act{g}l &\coloneqq   Xs^1_1\circ Xs^0_0(g)\cdot l\cdot Xs^1_1\circ Xs^0_0(g^{-1}),\\
\act{g}m& \coloneqq  Xs^2_2\circ Xs^1_1\circ Xs^0_0(g)\cdot m\cdot Xs^2_2\circ Xs^1_1\circ Xs^0_0(g^{-1}), \\
\act{h}l &\coloneqq  Xs^1_1(h)\cdot l\cdot Xs^1_1(h^{-1}),\\
\act{h}m& \coloneqq  Xs^2_2\circ Xs^1_1(h)\cdot m\cdot Xs^2_2\circ Xs^1_1(h^{-1}), \text{ and}\\
\act{l}m &\coloneqq  Xs^2_2(l)\cdot m\cdot Xs^2_2(l^{-1}).\\
\end{align*}

\item For each $h_1, h_2 \in H$, Peiffer lifting will be determind by following:
\begin{align*}
Xs_1^1(h_1)\cdot Xs_1^1(h_2)\cdot Xs_1^1(h_1^{-1})\cdot Xs_0^1(h_1)\cdot Xs_1^1(h_2^{-1})\cdot Xs_0^1(h_1^{-1}).
\end{align*}

\item For each $l_1, l_2 \in L$, LL-Peiffer lifting will be determind by following:
\begin{align*}
Xs_2^2(l_1)\cdot Xs_2^2(l_2)\cdot Xs_2^2(l_1^{-1})\cdot Xs_1^2(l_1)\cdot Xs_2^2(l_2^{-1})\cdot Xs_1^2(l_1^{-1}).
\end{align*}

\item For each $h \in H$, and $l \in L$, HL-Peiffer lifting will be determind by following:
\begin{align*}
&Xs_2^2(l_1)\cdot Xs_1^2(l_1^{-1})\cdot Xs_2^2\circ Xs_1^1(h_1)\cdot Xs_1^2(l_1)\cdot Xs_0^2(l_1^{-1})\\
&\cdot Xs_2^2\circ Xs_1^1(h_1^{-1})\cdot Xs_0^2(l_1)\cdot Xs_2^2\circ Xs_1^1(h_1)\\
&\cdot Xs_2^2(l_1^{-1})\cdot Xs_2^2\circ Xs_1^1(h_1^{-1}).
\end{align*}

\item For each $h \in H$, and $l \in L$, HL'-Peiffer lifting will be determind by following:
 \begin{align*}
 &Xs^2_2\circ Xs^1_1(h_1)\cdot Xs^2_2(l_1)\cdot Xs^2_1(l_1^{-1})\cdot Xs^2_2\circ Xs^1_1(h_1^{-1})\\
 &\cdot Xs^2_2\circ Xs^1_0(h_1)\cdot Xs^2_1(l_1)\cdot Xs^2_2\circ Xs^1_0(h_1^{-1})\\
 &\cdot Xs^2_1\circ Xs^1_0(h_1)\cdot Xs^2_2(l_1^{-1})\cdot Xs^2_1\circ Xs^1_0(h_1^{-1}).
 \end{align*}

\item For each $h_1, h_2, h_3 \in H$, left-Homanian will be determind by following:
\begin{align*}
Xs^2_2\circ Xs^1_1(h_1)\cdot Xs^2_2\circ Xs^1_1(h_2)\cdot Xs^2_2\circ Xs^1_1(h_3)\cdot Xs^2_2\circ Xs^1_1(h_2^{-1})\\
\cdot Xs^2_2\circ Xs^1_1(h_1^{-1})\cdot Xs^2_2\circ Xs^1_0(h_1)\cdot Xs^2_2\circ Xs^1_0(h_2)\cdot Xs^2_2\circ Xs^1_1(h_3^{-1})\\
\cdot Xs^2_2\circ Xs^1_0(h_2^{-1})\cdot Xs^2_1\circ Xs^1_0(h_2)\cdot Xs^2_2\circ Xs^1_1(h_3)\cdot Xs^2_1\circ Xs^1_0(h_2^{-1})\\
\cdot Xs^2_2\circ Xs^1_0(h_1^{-1})\cdot Xs^2_2\circ Xs^1_1(h_1)\cdot Xs^2_1\circ Xs^1_0(h_2)\cdot Xs^2_2\circ Xs^1_1(h_3^{-1})\\
\cdot Xs^2_1\circ Xs^1_0(h_2^{-1})\cdot Xs^2_2\circ Xs^1_0(h_2)\cdot Xs^2_2\circ Xs^1_1(h_3)\cdot Xs^2_2\circ Xs^1_0(h_2^{-1})\\
\cdot Xs^2_2\circ Xs^1_1(h_2)\cdot Xs^2_2\circ Xs^1_1(h_3^{-1})\cdot Xs^2_2\circ Xs^1_1(h_2^{-1})\cdot Xs^2_2\circ Xs^1_1(h_1^{-1}).
\end{align*}

\item For each $h_1, h_2, h_3 \in H$, right-Homanian will be determind by following:
\begin{align*}
&Xs^2_2\circ Xs^1_1(h_1)\cdot Xs^2_2\circ Xs^1_1(h_2)\cdot Xs^2_2\circ Xs^1_1(h_3)\cdot Xs^2_2\circ Xs^1_1(h_1^{-1})\\
&\cdot Xs^2_2\circ Xs^1_0(h_1)\cdot Xs^2_2\circ Xs^1_1(h_3^{-1})\cdot Xs^2_2\circ Xs^1_1(h_2^{-1})\cdot Xs^2_2\circ Xs^1_0(h_1^{-1})\\
&\cdot Xs^2_1\circ Xs^1_0(h_1)\cdot Xs^2_2\circ Xs^1_1(h_2)\cdot Xs^2_1\circ Xs^1_0(h_1^{-1})\cdot Xs^2_2\circ Xs^1_0(h_1)\\
&\cdot Xs^2_2\circ Xs^1_1(h_3)\cdot Xs^2_2\circ Xs^1_0(h_1^{-1})\cdot Xs^2_2\circ Xs^1_1(h_1)\cdot Xs^2_2\circ Xs^1_1(h_3^{-1})\\
&\cdot Xs^2_2\circ Xs^1_1(h_1^{-1})\cdot Xs^2_1\circ Xs^1_0(h_1)\cdot Xs^2_2\circ Xs^1_1(h_2^{-1})\cdot Xs^2_1\circ Xs^1_0(h_1^{-1})\\
&\cdot Xs^2_2\circ Xs^1_0(h_1)\cdot Xs^2_2\circ Xs^1_1(h_2)\cdot Xs^2_2\circ Xs^1_0(h_1^{-1})\cdot Xs^2_2\circ Xs^1_1(h_1)\\
&\cdot Xs^2_2\circ Xs^1_1(h_2^{-1})\cdot Xs^2_2\circ Xs^1_1(h_1^{-1}).
\end{align*}

\end{enumerate}
\end{theorem}

\begin{proof}
To prove that $M \xrightarrow{\partial} L \xrightarrow{\partial} H \xrightarrow{\partial} G$  forms a 3-crossed module, we must verify all the axioms of a 3-crossed module. Here, we verify property \ref{dhmnl}, which states that 
\[
\act{\partial h}m = \act{h}m\{h,\partial m^{-1}\}'_{HL}.
\]
The left-hand side and the right-hand side can be written as follows:

\begin{align*}
\act{\partial h}m &=
Xs^2_2\circ Xs^1_1\circ Xs^0_0\circ Xd^1_1(h)\cdot m\cdot Xs^2_2\circ Xs^1_1\circ Xs^0_0\circ Xd^1_1(h^{-1}),\\
\act{h}m\{h,\partial m^{-1}\}'_{HL} &= 
Xs^2_2\circ Xs^1_1(h)\cdot m\cdot Xs^2_2\circ Xs^1_1(h^{-1})\cdot Xs^2_2\circ Xs^1_1(h)\\
&\cdot Xs^2_2\circ Xd^3_3(m^{-1})\cdot Xs^2_1\circ Xd^3_3(m)\cdot Xs^2_2\circ Xs^1_1(h^{-1})\\
&\cdot Xs^2_2\circ Xs^1_0(h)\cdot Xs^2_1\circ Xd^3_3(m^{-1})\cdot Xs^2_2\circ Xs^1_0(h^{-1})\\
&\cdot Xs^2_1\circ Xs^1_0(h)\cdot Xs^2_2\circ Xd^3_3(m)\cdot Xs^2_1\circ Xs^1_0(h^{-1}).
\end{align*}

To prove this equation, it suffices to show that $\act{h}m\{h,\partial m^{-1}\}'_{HL}\act{\partial h}m^{-1} \in {\rm Im}Xd^4_4 $. By using identity (\ref{psgrp}), we can express $\act{\partial h}m$ and $ \act{h}m\{h,\partial m^{-1}\}'_{HL}$ as follows:
\begin{align*}
\act{\partial h}m &=
Xd^4_4\circ Xs^3_2\circ Xs^2_1\circ Xs^1_0(h)\cdot Xd^4_4\circ Xs^3_3(m)\cdot Xd^4_4\circ Xs^3_2\circ Xs^2_1\circ Xs^1_0(h^{-1})\\
&=Xd^4_4\circ \left(Xs^3_2\circ Xs^2_1\circ Xs^1_0(h)\cdot  Xs^3_3(m)\cdot  Xs^3_2\circ Xs^2_1\circ Xs^1_0(h^{-1})\right)\\
&\coloneqq Xd^4_4 (A)
\end{align*}
 and
\begin{align*}
\act{h}m\{h,\partial m^{-1}\}'_{HL} &= 
Xd^4_4\circ Xs^3_3\circ Xs^2_2\circ Xs^1_1(h)\cdot Xd^4_4\circ Xs^3_3(m)\\
&\cdot Xd^4_4\circ Xs^3_3\circ Xs^2_2\circ Xs^1_1(h^{-1})\cdot Xd^4_4\circ Xs^3_3\circ Xs^2_2\circ Xs^1_1(h)\cdot Xd^4_4\circ Xs^3_2(m^{-1})\\
&\cdot Xd^4_4\circ Xs^3_1(m)\cdot Xd^4_4\circ Xs^3_3\circ Xs^2_2\circ Xs^1_1(h^{-1})\cdot Xd^4_4\circ Xs^3_3\circ Xs^2_2\circ Xs^1_0(h)\\
&\cdot Xd^4_4\circ Xs^3_1(m^{-1})\cdot Xd^4_4\circ Xs^3_3\circ Xs^2_2\circ Xs^1_0(h^{-1})\cdot Xd^4_4\circ Xs^3_3\circ Xs^2_1\circ Xs^1_0(h)\\
&\cdot Xd^4_4\circ Xs^3_2(m)\cdot Xd^4_4\circ Xs^3_3\circ Xs^2_1\circ Xs^1_0(h^{-1})\\
&=Xd^4_4\circ
\left(
    \begin{aligned}
&\ \ Xs^3_3\circ Xs^2_2\circ Xs^1_1(h)\cdot Xs^3_3(m)\\
&\cdot  Xs^3_3\circ Xs^2_2\circ Xs^1_1(h^{-1})\cdot  Xs^3_3\circ Xs^2_2\circ Xs^1_1(h)\cdot  Xs^3_2(m^{-1})\\
&\cdot  Xs^3_1(m)\cdot  Xs^3_3\circ Xs^2_2\circ Xs^1_1(h^{-1})\cdot  Xs^3_3\circ Xs^2_2\circ Xs^1_0(h)\\
&\cdot  Xs^3_1(m^{-1})\cdot  Xs^3_3\circ Xs^2_2\circ Xs^1_0(h^{-1})\cdot  Xs^3_3\circ Xs^2_1\circ Xs^1_0(h)\\
&\cdot  Xs^3_2(m)\cdot  Xs^3_3\circ Xs^2_1\circ Xs^1_0(h^{-1})
    \end{aligned}
\right)\\
&\coloneqq Xd^4_4(B),
\end{align*}
 respectively.

From these expressions, we observe that $A, B \in X_4$. Thus, to prove $\act{h}m\{h,\partial m^{-1}\}'_{HL}\act{\partial h}m^{-1} \in {\rm Im}(Xd^4_4)$, it suffices to show that $AB^{-1} \in {\rm Ker}(Xd^4_i)$ for each $i=0,1,2,3$. We will compute $Xd^4_i(A)$ and $Xd^4_i(B)$, and show that $Xd^4_i(A)=Xd^4_i(B)$ for all $i=0,1,2,3$.

\begin{enumerate}[label=(\roman*)]
\item For i=0, by applying identity (\ref{psgrp}), we obtain the following expressions for $Xd^4_0(A)$ and $Xd^4_0(B)$:

\begin{align*}
Xd^4_0 (A) &=
Xs^2_1\circ Xs^1_0(h)\cdot Xs^2_1\circ Xs^1_0(h^{-1})\\
&=e, \text{ and} \\
Xd^4_0 (B) &=
Xs^2_2\circ Xs^1_1(h)\cdot Xs^2_2\circ Xs^1_1(h^{-1})\\
&\cdot Xs^2_2\circ Xs^1_0(h)\cdot Xs^2_2\circ Xs^1_0(h^{-1})\\
&=e.
\end{align*}

\item For i=1, by applying identity (\ref{psgrp}), we obtain the following expressions for $Xd^4_1(A)$ and $Xd^4_1(B)$:

\begin{align*}
Xd^4_1 (A) &=
Xs^2_1\circ Xs^1_0(h)\cdot Xs^2_1\circ Xs^1_0(h^{-1})\\
&=e, \text{ and}\\
Xd^4_1 (B) &=
Xs^2_2\circ Xs^1_1(h)\cdot Xs^2_2\circ Xs^1_1(h^{-1})\cdot Xs^2_2\circ Xs^1_1(h)\\
&\cdot m\cdot Xs^2_2\circ Xs^1_1(h^{-1})\cdot Xs^2_2\circ Xs^1_1(h)\cdot m^{-1}\\
&\cdot Xs^2_2\circ Xs^1_1(h^{-1})\cdot Xs^2_2\circ Xs^1_0(h)\cdot Xs^2_2\circ Xs^1_0(h^{-1})\\
&=e.
\end{align*}

\item For i=2, by applying identity (\ref{psgrp}), we obtain the following expressions for $Xd^4_2(A)$ and $Xd^4_2(B)$:
\begin{align*}
Xd^4_2 (A) &=
Xs^2_1\circ Xs^1_0(h)\cdot Xs^2_1\circ Xs^1_0(h^{-1})\\
&=e, \text{ and}\\
Xd^4_2 (B) &=
Xs^2_2\circ Xs^1_1(h)\cdot Xs^2_2\circ Xs^1_1(h^{-1})\cdot Xs^2_2\circ Xs^1_1(h)\\
&\cdot m^{-1}\cdot m\cdot Xs^2_2\circ Xs^1_1(h^{-1})\cdot Xs^2_2\circ Xs^1_0(h)\\
&\cdot m^{-1}\cdot Xs^2_2\circ Xs^1_0(h^{-1})\cdot Xs^2_2\circ Xs^1_0(h)\\
&\cdot m\cdot Xs^2_2\circ Xs^1_0(h^{-1})\\
&=e.
\end{align*}

\item For i=3, by applying identity (\ref{psgrp}), we obtain the following expressions for $Xd^4_3(A)$ and $Xd^4_3(B)$:

\begin{align*}
Xd^4_3 (A) &=
Xs^2_1\circ Xs^1_0(h)\cdot m\cdot Xs^2_1\circ Xs^1_0(h^{-1}), \text{ and}\\
Xd^4_3 (B) &=
Xs^2_2\circ Xs^1_1(h)\cdot m\cdot Xs^2_2\circ Xs^1_1(h^{-1})\\
&\cdot Xs^2_2\circ Xs^1_1(h)\cdot m^{-1}\cdot Xs^2_2\circ Xs^1_1(h^{-1})\\
&\cdot Xs^2_2\circ Xs^1_0(h)\cdot Xs^2_2\circ Xs^1_0(h^{-1})\\
&\cdot Xs^2_1\circ Xs^1_0(h)\cdot m\cdot Xs^2_1\circ Xs^1_0(h^{-1})\\
&=Xs^2_1\circ Xs^1_0(h)\cdot m\cdot Xs^2_1\circ Xs^1_0(h^{-1}).
\end{align*}
\end{enumerate}

By performing calculations for all $i=0,1,2,3$, we conclude that \( Xd^4_i(A) = Xd^4_i(B) \) for all \( i = 0,1,2,3 \), and therefore $\act{h}m\{h,\partial m^{-1}\}'_{HL}\act{\partial h}m^{-1} \in {\rm Im}(Xd^4_4)$, which establishes that property \ref{dhmnl} holds:
 \[
 \act{\partial h}m = \act{h}m\{h,\partial m^{-1}\}'_{HL}.
 \]

The remaining properties of a 3-crossed module can be verified in a similar manner\footnote{Although each axiom can be proven by hand, we have developed and used a computer program to verify that all the axioms of a 3-crossed module are satisfied.}.
\end{proof}

\subsection{Construction from a 2-crossed module}

In this section, we will give a way to construct a 3-crossed module from a 2-crossed module. This is a standard construction, and an extension of the well-known methods for obtaining a crossed module from a group or a 2-crossed module from a crossed module.

Let $C\xrightarrow{\partial}B \xrightarrow{\partial} A$ be a 2-crossed module. We will set the groups $G,H,L$ and $M$ as follows:
\[
G \coloneqq A, \quad H \coloneqq B\times A, \quad L \coloneqq C\times B, \quad M \coloneqq C.
\]
For $G$ and $M$, the group product is that of $A$ and $C$, respectively. For $H$ and $L$, the group product is defined component-wise; i.e., for $h_2=(b_2, a_2), h_1=(b_1, a_1) \in H$ and $l_2=(c_2, b_2), l_1=(c_1, b_1) \in l$, we define the group products $h_2 \cdot_H h_1$ and $l_2 \cdot_L l_1$ as following:
\begin{equation}
\label{ex2:prod}
\begin{aligned}
h_2 \cdot_H h_1 &\coloneqq (b_2 \cdot_B b_1, a_2 \cdot_A a_1), \\
l_2 \cdot_L l_1 &\coloneqq (c_2 \cdot_C c_1, b_2 \cdot_B b_1).
\end{aligned}
\end{equation}

We define the maps $\del:M \rightarrow L, \del:L \rightarrow H$ and $\del:H \rightarrow G$ as follows:
\begin{equation}
\label{ex2:del}
\begin{aligned}
\del : &M \longrightarrow L\\
   &c  \longmapsto (c, e_B), \\
\del : &L \longrightarrow H\\
   &(c,b)  \longmapsto (b, e_A),\\
\del : &H \longrightarrow G\\
   &(b,a)  \longmapsto a.         
\end{aligned}
\end{equation}
Here, $e_X$ denotes the identity element in the group $X$.

To equip this sequence with the structure of a 3-crossed module, we must define the required actions and six types of liftings. We define these as follows.

\textbf{Actions.}

The actions are derived from the 2-crossed module structure. Let $g=a_g \in G, h=(b_h, a_h) \in H, l=(c_l, b_l) \in L,$ and $m = c_m \in M$. We define:
\begin{equation}
\begin{aligned}
  \act{g}h &\coloneqq  (\act{a_g}b_h,\act{a_g}a_h), & \act{g}l &\coloneqq  (\act{a_g}c_l,\act{a_g}b_l), & \act{g}m &\coloneqq  \act{a_g}c_m, \\
  && \act{h}l &\coloneqq  (\act{b_h}c_l,\act{b_g}b_l), & \act{h}m &\coloneqq  \act{b_h}c_m, \\
  &&&& \act{l}m &\coloneqq  \act{c_l}c_m.
\end{aligned}
\end{equation}

\textbf{Peiffer lifting.}

Since $C\xrightarrow{\partial}B \xrightarrow{\partial} A$ is a 2-crossed module, it is equipped with a Peiffer lifting, which we denote $\{-,-\}_B : B \times B \to C$. We define the new Peiffer lifting $\{-,-\} : H \times H \to L$ using this. For $h_1=(b_1, a_1)$ and $h_2=(b_2, a_2)$ in $H$, we define:
\begin{align}
\{h_2, h_1\} \coloneqq (\{b_2, b_1\}_B, b_2b_1b_2^{-1} \act{a_1}b_1^{-1} ).
\end{align}

\textbf{Left-Homanian.}

We define the Left-Homanian $\{-, -, -\} : H \times H \times H \to M$ using the 2-crossed module's lifting $\{-,-\}_B$.  For $h_3 = (b_3, a_3), h_2 = (b_2, a_2), h_1 = (b_1, a_1) \in H$, it is defined as:
\begin{align}
\label{ex_lhmn}
\{h_3,h_2,h_1\} \coloneqq \{b_3b_2,b_1\}_B\{b_3,\act{a_2}b_1\}_B^{-1}\act{b_3}\{b_2,b_1\}_B^{-1}.
\end{align}
This expression is related to the 2-crossed module axiom (Property~\ref{lpl}). If $a_2 = \del (b_2)$ were satisfied, the right-hand side would equal $e_C$ by that property. However, in $H = B \times A$, $a_2$ is not necessarily equal to $\del(b_2)$, so this lifting is not generally trivial.

\textbf{Right-Homanian.}

 The Right-Homanian is defined in a similar manner. For $h_3 = (b_3, a_3), h_2 = (b_2, a_2), h_1 = (b_1, a_1) \in H$:
\begin{align}
\label{ex_rhmn}
\{h_3,h_2,h_1\}' \coloneqq \{b_3,b_2b_1\}_B\act{(\act{a_3}b_2)}\{b_3,b_1\}_B^{-1}\{b_3,b_2\}_B^{-1}.
\end{align}

\textbf{HL-Peiffer lifting.}

This lifting $\{-,-\}_{HL} : H \times L \to M$ is defined as:
For $l_1 = (c_1, b_1) \in L$ and $h_2 = (b_2, a_2) \in H$,
\begin{align}
\{h_2,l_1\}_{HL} \coloneqq c_1\{b_1^{-1},b_2\}_B\act{b_2}c_1^{-1}.
\end{align}

\textbf{HL'-Peiffer lifting.}

This lifting $\{-,-\}'_{HL} : H \times L \to M$ is defined as:
For $l_1 = (c_1, b_1) \in L$ and $h_2 = (b_2, a_2) \in H$,
\begin{align}
\{h_2,l_1\}'_{HL} = \act{b_2}c_1\{b_2,b_1^{-1}\}_B\act{a_2}c_1^{-1}.
\end{align}

\textbf{LL-Peiffer lifting.}

This lifting $\{-,-\} : L \times L \to M$ is defined as:
For $l_1 = (c_1, b_1), l_2 = (c_2, b_2)  \in L$,
\begin{align}
\label{ex2:ll}
\{l_2,l_1\}_{LL} = c_2 c_1 c_2^{-1}\act{b_2}c_1^{-1}.
\end{align}

With equations~(\ref{ex2:prod})through~(\ref{ex2:ll}), we have defined all structures required for a 3-crossed module: the groups, maps, actions, and six types of liftings. The remaining work is to prove that these definitions satisfy all the axioms of a 3-crossed module. These structures do indeed satisfy the axioms, leading to the following theorem:

\begin{theorem}
Let $C\xrightarrow{\partial}B \xrightarrow{\partial} A$ be a 2-crossed module. Let $G \coloneqq A$, $H \coloneqq B\times A$, $L \coloneqq C\times B$, and $M \coloneqq C$, with group structures as defined in \eqref{ex2:prod}.
Then $M \xrightarrow{\partial} L \xrightarrow{\partial} H \xrightarrow{\partial} G$ forms a 3-crossed module with the maps, actions, and liftings defined in \eqref{ex2:del} through \eqref{ex2:ll}.
\end{theorem}
\begin{proof}
The proof consists of checking all the properties of a 3-crossed module. This is a straightforward but lengthy calculation.

As an example, we will show that this structure satisfies the 3-crossed module axiom for Property~\ref{rl:hhl}. Property~\ref{rl:hhl} is as follows: for $h_1 = (b_1, a_1), h_2 = (b_2, a_2)  \in H$ and $l = (c_l, b_l) \in L$,
\begin{align*}
\{h_2h_1,l\}_{HL} = \act{l}\{\partial l^{-1},h_2,h_1\}'\{h_2,l\}_{HL} \act{h_2}\{h_1,l\}_{HL}  .
\end{align*}
We can show this holds by the following calculation:
\begin{align*}
  \nonumber
  (\text{LHS}) &= \{(b_2,a_2)(b_1,a_1),(c_l,b_l)\}_{HL} \\
  \nonumber
  &= \{(b_2b_1,a_2a_1),(c_l,b_l)\}_{HL} \\
  \nonumber
  &= c_l\{b_l^{-1},b_2b_1\}\act{b_2b_1}c_l^{-1} \\
  \nonumber
  (\text{RHS}) &= \act{(c_l,b_l)}\{\partial (c_l,b_l)^{-1},(b_2,a_2),(b_1,a_1)\}'\{(b_2,a_2),(c_l,b_l)\}_{HL}\act{(b_2,a_2)}\{(b_1,a_1),(c_l,b_l)\}_{HL} \\
  \nonumber
  &= \act{(c_l,b_l)}\{(b_l^{-1},e_A),(b_2,a_2),(b_1,a_1)\}'c_l\{b_l^{-1},b_2\}\act{b_2}c_l^{-1}\act{(b_2,a_2)}(c_l\{b_l^{-1},b_1\}\act{b_1}c_l^{-1}) \\
  \nonumber
  &= \act{(c_l,b_l)}(\{b_l^{-1},b_2b_1\}\act{b_2}\{b_l^{-1},b_1\}^{-1}\{b_l^{-1},b_2\}^{-1})c_l\{b_l^{-1},b_2\}\act{b_2}c_l^{-1}\act{b_2}c_l\act{b_2}\{b_l^{-1},b_1\}\act{b_2b_1}c_l^{-1} \\
  &= c_l(\{b_l^{-1},b_2b_1\}\act{b_2}\{b_l^{-1},b_1\}^{-1}\{b_l^{-1},b_2\}^{-1})c_l^{-1}c_l\{b_l^{-1},b_2\}\act{b_2}c_l^{-1}\act{b_2}c_l\act{b_2}\{b_l^{-1},b_1\}\act{b_2b_1}c_l^{-1} \\
  \nonumber
  &= c_l\{b_l^{-1},b_2b_1\}\act{b_2b_1}c_l^{-1} = (\text{LHS})
\end{align*}
The other properties of a 3-crossed module can be verified in a similar manner.
\end{proof}

\section{Conclusion and Future Work}

In this paper, we have proposed an alternative formulation of a 3-crossed module. Our definition, constructed from four groups and six kinds of liftings, was motivated by the goal of extending the established equivalence between 2-crossed modules and Gray 3-groups \cite{SarikayaUlualan2024}.

As the primary contribution of this paper, we validated this new definition. We proved that the simplicial set associated with our 3-crossed module forms a quasi-category. We also demonstrated its robustness by showing how it naturally extends a 2-crossed module and, crucially, that it is equivalent to the structure found in the Moore complex of length 3 of a simplicial group.

The established equivalence in \cite{SarikayaUlualan2024} serves as a benchmark. We conjecture that our formulation of a 3-crossed module is the \lq correct\rq \ one for this program, positing that there exists a one-dimension-higher Gray category structure whose category is equivalent to the category of our 3-crossed modules. Our immediate future work is to construct such a category and prove this anticipated equivalence.

\begin{appendices}
\section{3-crossed module and diagram}
\label{app_dia_3cm}
Here we will treat the 3-crossed module using diagrams.
This approach is useful for computing and understanding the operators of a 3-crossed module.
Throughout this section, we assume that $ M\xrightarrow{\partial}L \xrightarrow{\partial} H \xrightarrow{\partial} G $ is 3-crossed module. We use two types of diagrams: cube-type and lattice-type.
The cube-type diagram is useful for understanding the structure of a 3-crossed module, as it gives an intuitive picture of the extension of a Gray category.
The lattice-type diagram, on the other hand, is convenient for explicit calculations.

For $g_1, g_2 \in G$, product $g_2 \#_1 g_1 \coloneqq g_2\cdot g_1$ can be represented as in Figure~\ref{d:1-mor_comp}.

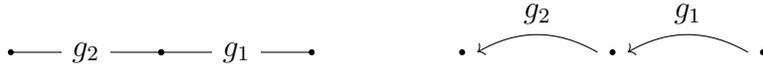
\begin{figure}[htbp]
\begin{center}
  \begin{tikzpicture}[scale=2]
    \begin{scope}[xshift=0.0cm]
      \draw (0,0) to node[sloped] [fill=white] {$g_2$} (1,0);
      \fill[black] (0,0) circle[radius=0.02];
      \fill[black] (1,0) circle[radius=0.02];
    \end{scope}
    \begin{scope}[xshift=1.0cm]
      \draw (0,0) to node [fill=white] {$g_1$} (1,0);
      \fill[black] (0,0) circle[radius=0.02];
      \fill[black] (1,0) circle[radius=0.02];
    \end{scope}
    \begin{scope}[xshift=3.0cm]
      \fill[black] (0,0) circle[radius=0.02];
      \fill[black] (1,0) circle[radius=0.02];
      \fill[black] (2,0) circle[radius=0.02];
      \draw[<-] (0.1,0) to [out=30,in=150] node[midway] (g2) {} (0.9,0)   ;
	      \node[above] at (g2) {$g_2$};
	\draw[<-] (1.1,0) to [out=30,in=150] node[midway] (g1) {} (1.9,0)   ;
	      \node[above] at (g1) {$g_1$};
    \end{scope}
  \end{tikzpicture}
\end{center}
\caption{Compositon of $G$}
\label{d:1-mor_comp}
\end{figure}

The left diagram is the cube-type representation, while the right diagram is the lattice-type. For the group $G$, these two diagrams look almost identical, but for $H$ and $L$ they appear quite different. An element $H$ can be regarded as a map defined by
\begin{align*}
h: &G \longrightarrow G\\
   &g  \longmapsto \partial(h)g .
\end{align*}
for each $h \in H$.

This correspondence is illustrated in Figure~\ref{d:2-mor}:
the left figure shows the cube-type diagram, and the right one shows the lattice-type.
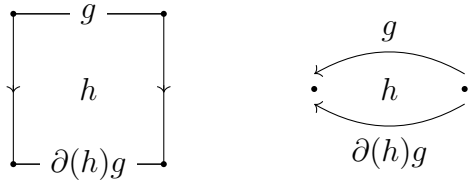
\begin{figure}
\begin{center}
  \begin{tikzpicture}[scale=2]
    \begin{scope}[xshift=0.0cm]
      \draw (0,0) to node[sloped] {\tikz{\draw[->](0,0)--(0.1,0);}} (0,-1);
      \draw (1,0) to node[sloped] {\tikz{\draw[->](0,0)--(0.1,0);}} (1,-1);
      \draw (0,0) to node[sloped] {\tikz{\draw[->](0,0)--(0.1,0);}} (1,0);
      \draw (0,-1) to node  {\tikz{\draw[-](0,0)--(0.1,0);}} (1,-1);
      \node at (1/2,-1/2) {$h$};
      \draw (0,0) to node [fill=white] {$g$} (1,0);
      \draw (0,-1) to node [fill=white] {$\del(h)g$} (1,-1);
      \fill[black] (0,0) circle[radius=0.02];
      \fill[black] (1,0) circle[radius=0.02];
      \fill[black] (0,-1) circle[radius=0.02];
      \fill[black] (1,-1) circle[radius=0.02];
    \end{scope}
    \begin{scope}[xshift=2.0cm, yshift=-0.5cm]
      \fill[black] (0,0) circle[radius=0.02];
      \fill[black] (1,0) circle[radius=0.02];
      \draw[<-] (0,0.1) to [out=30,in=150] node[midway] (g) {} (1,0.1)   ;
	      \node[above] at (g) {$g$};
	\draw[<-] (0,-0.1) to [out=-30,in=-150] node[midway] (g') {} (1,-0.1)   ;
	      \node[below] at (g') {$\del(h) g$};
	      \node at ($(g)!0.5!(g')$) {$h$};
    \end{scope}
  \end{tikzpicture}
\end{center}
\caption{Two types of $H$}
\label{d:2-mor}
\end{figure}

In this section, it is convenient to work with the cube-type representation; hence, we mainly use the cube-type diagrams.
The lattice-type diagrams are helpful for understanding the color conditions.
Occasionally, we present both cube-type and lattice-type diagrams for comparison, but in most cases we only use the cube-type in this section.
 
From now on, we denote the unit element $e_g$ of $G$ without a label, as shown in Figure~\ref{d:1-mor_e}.
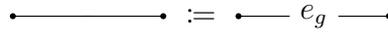
\begin{figure}
 \begin{center}
  \begin{tikzpicture}[scale=2]
    \begin{scope}[xshift=0.0cm]
      \draw (0,0) to node[sloped] {\tikz{\draw[-](0,0)--(0.1,0);}} (1,0);
      \fill[black] (0,0) circle[radius=0.02];
      \fill[black] (1,0) circle[radius=0.02];
    \end{scope}
    \begin{scope}[xshift=1.5cm]
      \draw (0,0) to node [fill=white] {$e_g$} (1,0);
      \fill[black] (0,0) circle[radius=0.02];
      \fill[black] (1,0) circle[radius=0.02];
    \end{scope}
    \node at (1.25, 0) {$\coloneqq$};
  \end{tikzpicture}
\end{center}
\caption{Unit element of $G$}
\label{d:1-mor_e}
\end{figure}

If the top edge of a square is determined, we can compute the bottom edge accordingly, so we may omit writing the label for the bottom edge.
For $h_1, h_2 \in H$, the product $h_2\#_2 h_1 \coloneqq h_2\cdot h_1$ is represented in Figure~\ref{d:2-mor_vertical}.
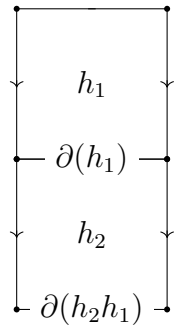
\begin{figure}
\begin{center}
  \begin{tikzpicture}[scale=2]
    \begin{scope}[xshift=0.0cm]
      \draw (0,0) to node[sloped] {\tikz{\draw[->](0,0)--(0.1,0);}} (0,-1);
      \draw (1,0) to node[sloped] {\tikz{\draw[->](0,0)--(0.1,0);}} (1,-1);
      \draw (0,0) to node[sloped] {\tikz{\draw[-](0,0)--(0.1,0);}} (1,0);
      \draw (0,-1) to node  {\tikz{\draw[-](0,0)--(0.1,0);}} (1,-1);
      \node at (1/2,-1/2) {$h_1$};
      \fill[black] (0,0) circle[radius=0.02];
      \fill[black] (1,0) circle[radius=0.02];
      \fill[black] (0,-1) circle[radius=0.02];
      \fill[black] (1,-1) circle[radius=0.02];
    \end{scope}
    \begin{scope}[yshift=-1.0cm]
      \draw (0,0) to node[sloped] {\tikz{\draw[->](0,0)--(0.1,0);}} (0,-1);
      \draw (1,0) to node[sloped] {\tikz{\draw[->](0,0)--(0.1,0);}} (1,-1);
      \draw (0,0) to node [fill=white] {$\del(h_1)$} (1,0);
      \draw (0,-1) to node [fill=white] {$\del(h_2 h_1)$} (1,-1);
      \node at (1/2,-1/2) {$h_2$};
      \fill[black] (0,0) circle[radius=0.02];
      \fill[black] (1,0) circle[radius=0.02];
      \fill[black] (0,-1) circle[radius=0.02];
      \fill[black] (1,-1) circle[radius=0.02];
    \end{scope}
  \end{tikzpicture}
\end{center}
\caption{Vertical compositon of $H$}
\label{d:2-mor_vertical}
\end{figure}

This product $h_2\#_2 h_1$ is called the {\bf vertical composion} in $H$ which is widely used in higher categories such as 2-categories and Gray categories. 
For each $g \in G$, we denote by ${\rm id}_g \in H$ an element satisfying ${\rm id}_g (g) = g$. Although ${\rm id}_g$ is not uniqe, the unit element of $H$ is always of this form. We depict ${\rm id}_g \in H$ in Figure~\ref{d:2-mor_unit}.
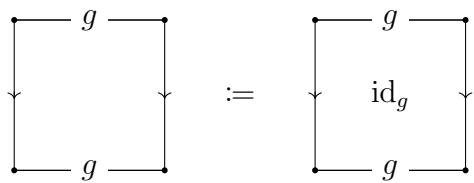
\begin{figure}
\begin{center}
  \begin{tikzpicture}[scale=2]
    \begin{scope}[xshift=0.0cm]
      \draw (0,0) to node[sloped] {\tikz{\draw[->](0,0)--(0.1,0);}} (0,-1);
      \draw (1,0) to node[sloped] {\tikz{\draw[->](0,0)--(0.1,0);}} (1,-1);
      \draw (0,0) to node[sloped] {\tikz{\draw[->](0,0)--(0.1,0);}} (1,0);
      \draw (0,-1) to node  {\tikz{\draw[-](0,0)--(0.1,0);}} (1,-1);
      \node at (1/2,-1/2) {$$};
      \draw (0,0) to node [fill=white] {$g$} (1,0);
      \draw (0,-1) to node [fill=white] {$g$} (1,-1);
      \fill[black] (0,0) circle[radius=0.02];
      \fill[black] (1,0) circle[radius=0.02];
      \fill[black] (0,-1) circle[radius=0.02];
      \fill[black] (1,-1) circle[radius=0.02];
    \end{scope}
     \begin{scope}[xshift=2.0cm]
      \draw (0,0) to node[sloped] {\tikz{\draw[->](0,0)--(0.1,0);}} (0,-1);
      \draw (1,0) to node[sloped] {\tikz{\draw[->](0,0)--(0.1,0);}} (1,-1);
      \draw (0,0) to node[sloped] {\tikz{\draw[->](0,0)--(0.1,0);}} (1,0);
      \draw (0,-1) to node  {\tikz{\draw[-](0,0)--(0.1,0);}} (1,-1);
      \node at (1/2,-1/2) {${\rm id}_g$};
      \draw (0,0) to node [fill=white] {$g$} (1,0);
      \draw (0,-1) to node [fill=white] {$g$} (1,-1);
      \fill[black] (0,0) circle[radius=0.02];
      \fill[black] (1,0) circle[radius=0.02];
      \fill[black] (0,-1) circle[radius=0.02];
      \fill[black] (1,-1) circle[radius=0.02];
    \end{scope}
        \node at (1.5,-1/2) {$\coloneqq$};
  \end{tikzpicture}
\end{center}
\caption{Unit element of $H$}
\label{d:2-mor_unit}
\end{figure}

We will omit the label for ${\rm id}_g \in H$ in the following diagrams.

For $g_1, g_2 \in G$ and $h_1, {\rm id}_{g_2} \in H$, and assume that $h_1(g_1) = g'_1$. Then there exists a {\bf horizontal compositon} ${\rm id}_{g_2} \#_1 h_1 (g_2\#_1 g_1 ) = g_2\#_1 g'_1$, which is illustrated in Figure~\ref{d:2-mor_L_hori}. 

\begin{figure}
\begin{center}
  \begin{tikzpicture}[scale=2]
    \begin{scope}[xshift=0.0cm]
      \draw (0,0) to node[sloped] {\tikz{\draw[->](0,0)--(0.1,0);}} (0,-1);
      \draw (1,0) to node[sloped] {\tikz{\draw[->](0,0)--(0.1,0);}} (1,-1);
      \draw (0,0) to node[sloped] {\tikz{\draw[->](0,0)--(0.1,0);}} (1,0);
      \draw (0,-1) to node  {\tikz{\draw[-](0,0)--(0.1,0);}} (1,-1);
      \node at (1/2,-1/2) {$$};
      \draw (0,0) to node [fill=white] {$g_2$} (1,0);
      \draw (0,-1) to node [fill=white] {$g_2$} (1,-1);
      \fill[black] (0,0) circle[radius=0.02];
      \fill[black] (1,0) circle[radius=0.02];
      \fill[black] (0,-1) circle[radius=0.02];
      \fill[black] (1,-1) circle[radius=0.02];
    \end{scope}
    \begin{scope}[xshift=1.0cm]
      \draw (0,0) to node[sloped] {\tikz{\draw[->](0,0)--(0.1,0);}} (0,-1);
      \draw (1,0) to node[sloped] {\tikz{\draw[->](0,0)--(0.1,0);}} (1,-1);
      \draw (0,0) to node[sloped] {\tikz{\draw[->](0,0)--(0.1,0);}} (1,0);
      \draw (0,-1) to node  {\tikz{\draw[-](0,0)--(0.1,0);}} (1,-1);
      \node at (1/2,-1/2) {$h_1$};
      \draw (0,0) to node [fill=white] {$g_1$} (1,0);
      \draw (0,-1) to node [fill=white] {$g_1'$} (1,-1);
      \fill[black] (0,0) circle[radius=0.02];
      \fill[black] (1,0) circle[radius=0.02];
      \fill[black] (0,-1) circle[radius=0.02];
      \fill[black] (1,-1) circle[radius=0.02];
    \end{scope}
     \begin{scope}[xshift=3.0cm, yshift=-0.5cm]
      \fill[black] (0,0) circle[radius=0.02];
      \fill[black] (1,0) circle[radius=0.02];
       \fill[black] (2,0) circle[radius=0.02];
      \draw[<-] (1.1,0.1) to [out=30,in=150] node[midway] (g1) {} (1.9,0.1)   ;
	      \node[above] at (g1) {$g_1$};
	\draw[<-] (1.1,-0.1) to [out=-30,in=-150] node[midway] (g1') {} (1.9,-0.1)   ;
	      \node[below] at (g1') {$g_1'$};
	      \node at ($(g1)!0.5!(g1')$) {$h_1$};
	\draw[<-] (0.1,0) to [out=0,in=0] node[midway] (g2) {} (0.8,0);
	 \node[above] at (0.5,0) {$g_2$};
    \end{scope}
  \end{tikzpicture}
\end{center}
\caption{Horizontal compositon of $H$}
\label{d:2-mor_L_hori}
\end{figure}

We can reinterpret this horizontal composition ${\rm id}_{g_2} \#_1 h_1$ in categorical language as $\act{g_2}h_1$ when written in the language of groups, according to the following formula:
\begin{align*}
\act{g_2}h_1 (g_2\cdot g_1) &= \partial(\act{g_2}h_1 )g_2\cdot g_1\\
&= g_2 \cdot \partial(h_1) \cdot g_2^{-1} \cdot g_2 \cdot g_1\\
&= g_2 \cdot \partial(h_1) \cdot g_1\\
&= g_2 \cdot h_1(g_1)\\
&= g_2 \cdot g'_1 \\
&= {\rm id}_{g_2} \#_1 h_1 (g_2\#_1 g_1 ).
\end{align*}

For $g_1, g_2 \in G$ and $ {\rm id}_{g_1}, h_2\in H$, assume that $h_2(g_2) = g'_2$. There is a switched version of the horizontal composition:  $h_2 \#_1 {\rm id}_{g_1}  (g_2\#_1 g_1 ) = g'_2\#_1 g_1$. We represent $h_2 \#_1 {\rm id}_{g_1}$ in Figure~\ref{d:2-mor_R_hori}.

\begin{figure}
\begin{center}
  \begin{tikzpicture}[scale=2]
    \begin{scope}[xshift=0.0cm]
      \draw (0,0) to node[sloped] {\tikz{\draw[->](0,0)--(0.1,0);}} (0,-1);
      \draw (1,0) to node[sloped] {\tikz{\draw[->](0,0)--(0.1,0);}} (1,-1);
      \draw (0,0) to node[sloped] {\tikz{\draw[->](0,0)--(0.1,0);}} (1,0);
      \draw (0,-1) to node  {\tikz{\draw[-](0,0)--(0.1,0);}} (1,-1);
      \node at (1/2,-1/2) {$h_2$};
      \draw (0,0) to node [fill=white] {$g_2$} (1,0);
      \draw (0,-1) to node [fill=white] {$g_2'$} (1,-1);
      \fill[black] (0,0) circle[radius=0.02];
      \fill[black] (1,0) circle[radius=0.02];
      \fill[black] (0,-1) circle[radius=0.02];
      \fill[black] (1,-1) circle[radius=0.02];
    \end{scope}
    \begin{scope}[xshift=1.0cm]
      \draw (0,0) to node[sloped] {\tikz{\draw[->](0,0)--(0.1,0);}} (0,-1);
      \draw (1,0) to node[sloped] {\tikz{\draw[->](0,0)--(0.1,0);}} (1,-1);
      \draw (0,0) to node[sloped] {\tikz{\draw[->](0,0)--(0.1,0);}} (1,0);
      \draw (0,-1) to node  {\tikz{\draw[-](0,0)--(0.1,0);}} (1,-1);
      \node at (1/2,-1/2) {$$};
      \draw (0,0) to node [fill=white] {$g_1$} (1,0);
      \draw (0,-1) to node [fill=white] {$g_1$} (1,-1);
      \fill[black] (0,0) circle[radius=0.02];
      \fill[black] (1,0) circle[radius=0.02];
      \fill[black] (0,-1) circle[radius=0.02];
      \fill[black] (1,-1) circle[radius=0.02];
    \end{scope}
    \begin{scope}[xshift=3.0cm, yshift=-0.5cm]
      \fill[black] (0,0) circle[radius=0.02];
      \fill[black] (1,0) circle[radius=0.02];
       \fill[black] (2,0) circle[radius=0.02];
      \draw[<-] (0.1,0.1) to [out=30,in=150] node[midway] (g1) {} (0.9,0.1)   ;
	      \node[above] at (g1) {$g_2$};
	\draw[<-] (0.1,-0.1) to [out=-30,in=-150] node[midway] (g1') {} (0.9,-0.1)   ;
	      \node[below] at (g1') {$g_2'$};
	      \node at ($(g1)!0.5!(g1')$) {$h_2$};
	\draw[<-] (1.1,0) to [out=0,in=0] node[midway] (g2) {} (1.8,0);
	 \node[above] at (1.5,0) {$g_1$};
    \end{scope}
  \end{tikzpicture}
\end{center}
\caption{Horizontal compositon of $H$}
\label{d:2-mor_R_hori}
\end{figure}

By the same discussion as for ${\rm id}_{g_2} \#_1 h_1$, the composition $h_2 \#_1 {\rm id}_{g_1}$ can be reinterpreted as $h_2$, according to the following formula:
\begin{align*}
h_2 (g_2\cdot g_1) &=  \partial(h_2)g_2\cdot g_1\\
&=h_2(g_2)\cdot h_1 \\
&=g'_2 \cdot h_1 \\
&= h_2 \#_1 {\rm id}_{g_1}  (g_2\#_1 g_1 ) = g'_2\#_1 g_1
\end{align*}

For $h_1, h_2 \in H$ and $g_1, g_2 \in G$, assume that $h_1(g_1) = g'_1$ and $h_2(g_2) = g'_2$. We would like to define the horizontal composition $h_2\#_1 h_1 (g_2\#_1 g_1 ) = g'_2\#_1 g'_1$ as illustrated in Figure~\ref{d:not_well_hor}, but this composition is not well-defined.
\begin{figure}
\begin{center}
  \begin{tikzpicture}[scale=2]
    \begin{scope}[xshift=0.0cm]
      \draw (0,0) to node[sloped] {\tikz{\draw[->](0,0)--(0.1,0);}} (0,-1);
      \draw (1,0) to node[sloped] {\tikz{\draw[->](0,0)--(0.1,0);}} (1,-1);
      \draw (0,0) to node[sloped] {\tikz{\draw[->](0,0)--(0.1,0);}} (1,0);
      \draw (0,-1) to node  {\tikz{\draw[-](0,0)--(0.1,0);}} (1,-1);
      \node at (1/2,-1/2) {$h_2$};
      \draw (0,0) to node [fill=white] {$g_2$} (1,0);
      \draw (0,-1) to node [fill=white] {$g_2'$} (1,-1);
      \fill[black] (0,0) circle[radius=0.02];
      \fill[black] (1,0) circle[radius=0.02];
      \fill[black] (0,-1) circle[radius=0.02];
      \fill[black] (1,-1) circle[radius=0.02];
    \end{scope}
    \begin{scope}[xshift=1.0cm]
      \draw (0,0) to node[sloped] {\tikz{\draw[->](0,0)--(0.1,0);}} (0,-1);
      \draw (1,0) to node[sloped] {\tikz{\draw[->](0,0)--(0.1,0);}} (1,-1);
      \draw (0,0) to node[sloped] {\tikz{\draw[->](0,0)--(0.1,0);}} (1,0);
      \draw (0,-1) to node  {\tikz{\draw[-](0,0)--(0.1,0);}} (1,-1);
      \node at (1/2,-1/2) {$h_1$};
      \draw (0,0) to node [fill=white] {$g_1$} (1,0);
      \draw (0,-1) to node [fill=white] {$g_1'$} (1,-1);
      \fill[black] (0,0) circle[radius=0.02];
      \fill[black] (1,0) circle[radius=0.02];
      \fill[black] (0,-1) circle[radius=0.02];
      \fill[black] (1,-1) circle[radius=0.02];
    \end{scope}
  \end{tikzpicture}
\end{center}
\caption{Image of $h_2\#_1 h_1$}
\label{d:not_well_hor}
\end{figure}
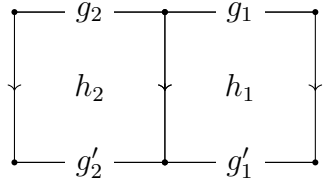

\begin{figure}
\begin{center}
  \begin{tikzpicture}[scale=2]
    \begin{scope}[xshift=0.0cm]
      \draw (0,0) to node[sloped] {\tikz{\draw[->](0,0)--(0.1,0);}} (0,-1);
      \draw (1,0) to node[sloped] {\tikz{\draw[->](0,0)--(0.1,0);}} (1,-1);
      \draw (0,0) to node[sloped] {\tikz{\draw[->](0,0)--(0.1,0);}} (1,0);
      \draw (0,-1) to node  {\tikz{\draw[-](0,0)--(0.1,0);}} (1,-1);
      \node at (1/2,-1/2) {$h_2$};
      \draw (0,0) to node [fill=white] {$g_2$} (1,0);
      \draw (0,-1) to node [fill=white] {$g_2'$} (1,-1);
      \fill[black] (0,0) circle[radius=0.02];
      \fill[black] (1,0) circle[radius=0.02];
      \fill[black] (0,-1) circle[radius=0.02];
      \fill[black] (1,-1) circle[radius=0.02];
    \end{scope}
    \begin{scope}[xshift=1.0cm]
      \draw (0,0) to node[sloped] {\tikz{\draw[->](0,0)--(0.1,0);}} (0,-1);
      \draw (1,0) to node[sloped] {\tikz{\draw[->](0,0)--(0.1,0);}} (1,-1);
      \draw (0,0) to node[sloped] {\tikz{\draw[->](0,0)--(0.1,0);}} (1,0);
      \draw (0,-1) to node  {\tikz{\draw[-](0,0)--(0.1,0);}} (1,-1);
      \node at (1/2,-1/2) {$$};
      \draw (0,0) to node [fill=white] {$g_1$} (1,0);
      \draw (0,-1) to node [fill=white] {$g_1$} (1,-1);
      \fill[black] (0,0) circle[radius=0.02];
      \fill[black] (1,0) circle[radius=0.02];
      \fill[black] (0,-1) circle[radius=0.02];
      \fill[black] (1,-1) circle[radius=0.02];
    \end{scope}
        \begin{scope}[xshift=0.0cm, yshift = -1.0cm]
      \draw (0,0) to node[sloped] {\tikz{\draw[->](0,0)--(0.1,0);}} (0,-1);
      \draw (1,0) to node[sloped] {\tikz{\draw[->](0,0)--(0.1,0);}} (1,-1);
      \draw (0,0) to node[sloped] {\tikz{\draw[->](0,0)--(0.1,0);}} (1,0);
      \draw (0,-1) to node  {\tikz{\draw[-](0,0)--(0.1,0);}} (1,-1);
      \node at (1/2,-1/2) {$$};
      \draw (0,0) to node [fill=white] {$g_2'$} (1,0);
      \draw (0,-1) to node [fill=white] {$g_2'$} (1,-1);
      \fill[black] (0,0) circle[radius=0.02];
      \fill[black] (1,0) circle[radius=0.02];
      \fill[black] (0,-1) circle[radius=0.02];
      \fill[black] (1,-1) circle[radius=0.02];
    \end{scope}
    \begin{scope}[xshift=1.0cm, yshift = -1.0cm]
      \draw (0,0) to node[sloped] {\tikz{\draw[->](0,0)--(0.1,0);}} (0,-1);
      \draw (1,0) to node[sloped] {\tikz{\draw[->](0,0)--(0.1,0);}} (1,-1);
      \draw (0,0) to node[sloped] {\tikz{\draw[->](0,0)--(0.1,0);}} (1,0);
      \draw (0,-1) to node  {\tikz{\draw[-](0,0)--(0.1,0);}} (1,-1);
      \node at (1/2,-1/2) {$h_1$};
      \draw (0,0) to node [fill=white] {$g_1$} (1,0);
      \draw (0,-1) to node [fill=white] {$g_1'$} (1,-1);
      \fill[black] (0,0) circle[radius=0.02];
      \fill[black] (1,0) circle[radius=0.02];
      \fill[black] (0,-1) circle[radius=0.02];
      \fill[black] (1,-1) circle[radius=0.02];
    \end{scope}
    
       \begin{scope}[xshift=3.0cm, yshift=-1.0cm]
 \fill[black] (0,0) circle[radius=0.02];
      \fill[black] (1,0) circle[radius=0.02];
      \fill[black] (2,0) circle[radius=0.02];
      \draw[<-] (0.1,0.1) to [out=30,in=150] node[midway] (g2) {} (0.9,0.1)   ;
	      \node[above] at (g2) {$g_2$};
	\draw[<-] (0.1,-0.1) to [out=-30,in=-150] node[midway] (g2') {} (0.9,-0.1)   ;
	      \node[below] at (g2') {$g_2'$};
	      \node at ($(g2)!0.5!(g2')$) {$(h_2,1)$};
	\draw[<-] (1.1,0.1) to [out=30,in=150] node[midway] (g1) {} (1.9,0.1)   ;
	      \node[above] at (g1) {$g_1$};
	\draw[<-] (1.1,-0.1) to [out=-30,in=-150] node[midway] (g1') {} (1.9,-0.1)   ;
	      \node[below] at (g1') {$g_1'$};
	      \node at ($(g1)!0.5!(g1')$) {$(h_1,2)$};      
    \end{scope}
    
  \end{tikzpicture}
\end{center}
\caption{Two types of $({\rm id}_{g'_2}\#_1 h_1) \#_2 (h_2\#_1 {\rm id}_{g_1})$}
\label{d:h1h2}
\end{figure}
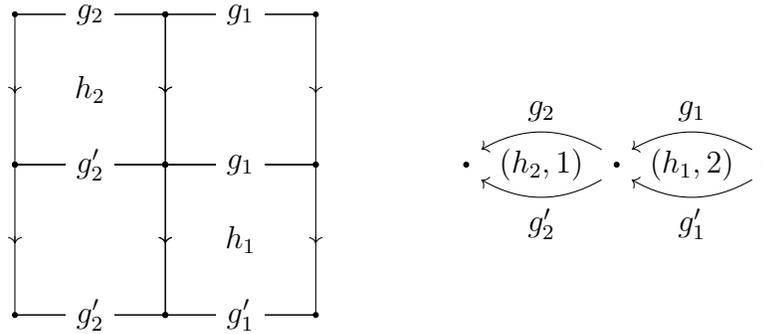

The horizontal composition $h_2\#_1 h_1$ can be expressed in two equivalent ways: 
\[
({\rm id}_{g'_2}\#_1 h_1) \#_2 (h_2\#_1 {\rm id}_{g_1})
\]
and 
\[(h_2\#_1 {\rm id}_{g'_1}) \#_2 ({\rm id}_{g_2}\#_1 h_1).
\] 
These two compositions are illustrated in Figures~\ref{d:h1h2} and~\ref{d:h2h1}, respectively.

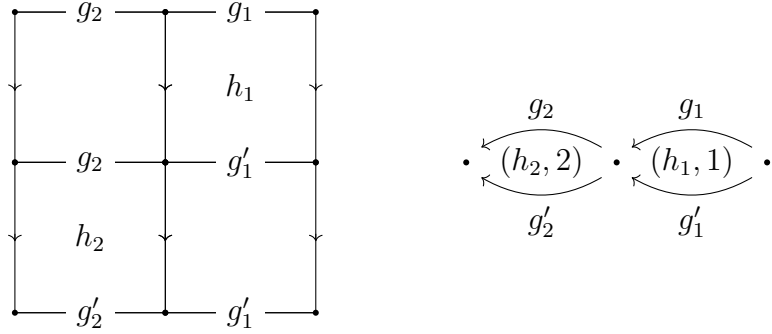
\begin{figure}[h]
\begin{center}
  \begin{tikzpicture}[scale=2]
    \begin{scope}[xshift=0.0cm]
      \draw (0,0) to node[sloped] {\tikz{\draw[->](0,0)--(0.1,0);}} (0,-1);
      \draw (1,0) to node[sloped] {\tikz{\draw[->](0,0)--(0.1,0);}} (1,-1);
      \draw (0,0) to node[sloped] {\tikz{\draw[->](0,0)--(0.1,0);}} (1,0);
      \draw (0,-1) to node  {\tikz{\draw[-](0,0)--(0.1,0);}} (1,-1);
      \node at (1/2,-1/2) {$$};
      \draw (0,0) to node [fill=white] {$g_2$} (1,0);
      \draw (0,-1) to node [fill=white] {$g_2$} (1,-1);
      \fill[black] (0,0) circle[radius=0.02];
      \fill[black] (1,0) circle[radius=0.02];
      \fill[black] (0,-1) circle[radius=0.02];
      \fill[black] (1,-1) circle[radius=0.02];
    \end{scope}
    \begin{scope}[xshift=1.0cm]
      \draw (0,0) to node[sloped] {\tikz{\draw[->](0,0)--(0.1,0);}} (0,-1);
      \draw (1,0) to node[sloped] {\tikz{\draw[->](0,0)--(0.1,0);}} (1,-1);
      \draw (0,0) to node[sloped] {\tikz{\draw[->](0,0)--(0.1,0);}} (1,0);
      \draw (0,-1) to node  {\tikz{\draw[-](0,0)--(0.1,0);}} (1,-1);
      \node at (1/2,-1/2) {$h_1$};
      \draw (0,0) to node [fill=white] {$g_1$} (1,0);
      \draw (0,-1) to node [fill=white] {$g_1'$} (1,-1);
      \fill[black] (0,0) circle[radius=0.02];
      \fill[black] (1,0) circle[radius=0.02];
      \fill[black] (0,-1) circle[radius=0.02];
      \fill[black] (1,-1) circle[radius=0.02];
    \end{scope}
        \begin{scope}[xshift=0.0cm, yshift = -1.0cm]
      \draw (0,0) to node[sloped] {\tikz{\draw[->](0,0)--(0.1,0);}} (0,-1);
      \draw (1,0) to node[sloped] {\tikz{\draw[->](0,0)--(0.1,0);}} (1,-1);
      \draw (0,0) to node[sloped] {\tikz{\draw[->](0,0)--(0.1,0);}} (1,0);
      \draw (0,-1) to node  {\tikz{\draw[-](0,0)--(0.1,0);}} (1,-1);
      \node at (1/2,-1/2) {$h_2$};
      \draw (0,0) to node [fill=white] {$g_2$} (1,0);
      \draw (0,-1) to node [fill=white] {$g_2'$} (1,-1);
      \fill[black] (0,0) circle[radius=0.02];
      \fill[black] (1,0) circle[radius=0.02];
      \fill[black] (0,-1) circle[radius=0.02];
      \fill[black] (1,-1) circle[radius=0.02];
    \end{scope}
    \begin{scope}[xshift=1.0cm, yshift = -1.0cm]
      \draw (0,0) to node[sloped] {\tikz{\draw[->](0,0)--(0.1,0);}} (0,-1);
      \draw (1,0) to node[sloped] {\tikz{\draw[->](0,0)--(0.1,0);}} (1,-1);
      \draw (0,0) to node[sloped] {\tikz{\draw[->](0,0)--(0.1,0);}} (1,0);
      \draw (0,-1) to node  {\tikz{\draw[-](0,0)--(0.1,0);}} (1,-1);
      \node at (1/2,-1/2) {$$};
      \draw (0,0) to node [fill=white] {$g_1'$} (1,0);
      \draw (0,-1) to node [fill=white] {$g_1'$} (1,-1);
      \fill[black] (0,0) circle[radius=0.02];
      \fill[black] (1,0) circle[radius=0.02];
      \fill[black] (0,-1) circle[radius=0.02];
      \fill[black] (1,-1) circle[radius=0.02];
    \end{scope}
    \begin{scope}[xshift=3.0cm, yshift=-1.0cm]
 \fill[black] (0,0) circle[radius=0.02];
      \fill[black] (1,0) circle[radius=0.02];
      \fill[black] (2,0) circle[radius=0.02];
      \draw[<-] (0.1,0.1) to [out=30,in=150] node[midway] (g2) {} (0.9,0.1)   ;
	      \node[above] at (g2) {$g_2$};
	\draw[<-] (0.1,-0.1) to [out=-30,in=-150] node[midway] (g2') {} (0.9,-0.1)   ;
	      \node[below] at (g2') {$g_2'$};
	      \node at ($(g2)!0.5!(g2')$) {$(h_2,2)$};
	\draw[<-] (1.1,0.1) to [out=30,in=150] node[midway] (g1) {} (1.9,0.1)   ;
	      \node[above] at (g1) {$g_1$};
	\draw[<-] (1.1,-0.1) to [out=-30,in=-150] node[midway] (g1') {} (1.9,-0.1)   ;
	      \node[below] at (g1') {$g_1'$};
	      \node at ($(g1)!0.5!(g1')$) {$(h_1,1)$};      
    \end{scope}
  \end{tikzpicture}
\end{center}
\caption{Two types of $(h_2\#_1 {\rm id}_{g'_1}) \#_2 ({\rm id}_{g_2}\#_1 h_1)$}
\label{d:h2h1}
\end{figure}

In the lattice-type representation, we must distinguish which element of $H$ acts first. To do this, we introduce an ordered pair, whose first component is the element of $H$, and whose second component is a natural number indicating the order in which the elements of $H$ are applied.

From the discussion above, we obtain 
\[
({\rm id}_{g'_2}\#_1 h_1) \#_2 (h_2\#_1 {\rm id}_{g_1}) \coloneqq \act{g'_2}h_1\cdot h_2 = \act{\partial(h_2)g_2}h_1\cdot h_2,
\]
and 
\[
(h_2\#_1 {\rm id}_{g'_1}) \#_2 ({\rm id}_{g_2}\#_1 h_1) \coloneqq h_2 \cdot \act{g_2}h_1.
\]
As in the case of $H$, each element of $L$ can be regarded as a map
\[
l : H \longrightarrow H, \qquad h \longmapsto \partial(l)\,h ,
\]
This correspondence is illustrated in Figure~\ref{d:l}.
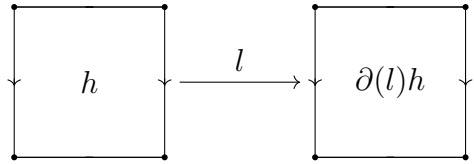
\begin{figure}
\begin{center}
  \begin{tikzpicture}[scale=2]
    \begin{scope}[xshift=0.0cm]
      \draw (0,0) to node[sloped] {\tikz{\draw[->](0,0)--(0.1,0);}} (0,-1);
      \draw (1,0) to node[sloped] {\tikz{\draw[->](0,0)--(0.1,0);}} (1,-1);
      \draw (0,0) to node[sloped] {\tikz{\draw[-](0,0)--(0.1,0);}} (1,0);
      \draw (0,-1) to node  {\tikz{\draw[-](0,0)--(0.1,0);}} (1,-1);
      \node at (1/2,-1/2) {$h$};
      \draw (0,0) to node {\tikz{\draw[-](0,0)--(0.1,0);}} (1,0);
      \draw (0,-1) to node {\tikz{\draw[-](0,0)--(0.1,0);}} (1,-1);
      \fill[black] (0,0) circle[radius=0.02];
      \fill[black] (1,0) circle[radius=0.02];
      \fill[black] (0,-1) circle[radius=0.02];
      \fill[black] (1,-1) circle[radius=0.02];
      \draw[->] (1.1,-1/2) to (1.9,-1/2);
      \node at (1.5,-1/2) [above]{$l$};
    \end{scope}
    \begin{scope}[xshift=2.0cm]
      \draw (0,0) to node[sloped] {\tikz{\draw[->](0,0)--(0.1,0);}} (0,-1);
      \draw (1,0) to node[sloped] {\tikz{\draw[->](0,0)--(0.1,0);}} (1,-1);
      \draw (0,0) to node[sloped] {\tikz{\draw[-](0,0)--(0.1,0);}} (1,0);
      \draw (0,-1) to node  {\tikz{\draw[-](0,0)--(0.1,0);}} (1,-1);
      \node at (1/2,-1/2) {$\del(l)h$};
      \draw (0,0) to node {\tikz{\draw[-](0,0)--(0.1,0);}} (1,0);
      \draw (0,-1) to node {\tikz{\draw[-](0,0)--(0.1,0);}} (1,-1);
      \fill[black] (0,0) circle[radius=0.02];
      \fill[black] (1,0) circle[radius=0.02];
      \fill[black] (0,-1) circle[radius=0.02];
      \fill[black] (1,-1) circle[radius=0.02];
	\end{scope}
  \end{tikzpicture}
\end{center}
\caption{Diagram of $L$}
\label{d:l}
\end{figure}

We denote the vertical composition of $l_2,l_1 \in L$ as $l_2 \#_3 l_1$. In diagrams, the vertical composition $l_2 \#_3 l_1$ is represented in Figure~\ref{d:ver_of_l}.
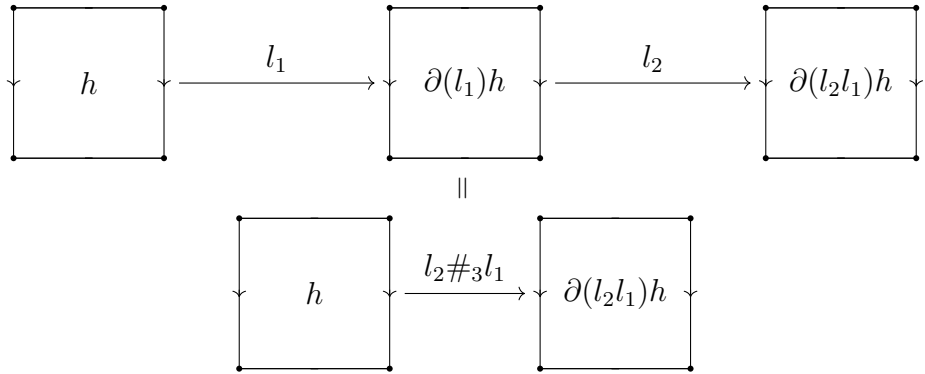
\begin{figure}[!]
\begin{center}
  \begin{tikzpicture}[scale=2]
    \begin{scope}[xshift=0.0cm]
      \draw[->] (1.1,-1/2) to (2.4,-1/2);
      \node at (1.75,-1/2) [above]{$l_1$};
      \draw (0,0) to node[sloped] {\tikz{\draw[->](0,0)--(0.1,0);}} (0,-1);
      \draw (1,0) to node[sloped] {\tikz{\draw[->](0,0)--(0.1,0);}} (1,-1);
      \draw (0,0) to node[sloped] {\tikz{\draw[-](0,0)--(0.1,0);}} (1,0);
      \draw (0,-1) to node  {\tikz{\draw[-](0,0)--(0.1,0);}} (1,-1);
      \node at (1/2,-1/2) {$h$};
      \draw (0,0) to node {\tikz{\draw[-](0,0)--(0.1,0);}} (1,0);
      \draw (0,-1) to node {\tikz{\draw[-](0,0)--(0.1,0);}}(1,-1);
      \fill[black] (0,0) circle[radius=0.02];
      \fill[black] (1,0) circle[radius=0.02];
      \fill[black] (0,-1) circle[radius=0.02];
      \fill[black] (1,-1) circle[radius=0.02];   
	  \end{scope}
	  \begin{scope}[xshift=2.5cm]
      \draw[->] (1.1,-1/2) to (2.4,-1/2);
      \node at (1.75,-1/2) [above]{$l_2$};
      \draw (0,0) to node[sloped] {\tikz{\draw[->](0,0)--(0.1,0);}} (0,-1);
      \draw (1,0) to node[sloped] {\tikz{\draw[->](0,0)--(0.1,0);}} (1,-1);
      \draw (0,0) to node[sloped] {\tikz{\draw[-](0,0)--(0.1,0);}} (1,0);
      \draw (0,-1) to node  {\tikz{\draw[-](0,0)--(0.1,0);}} (1,-1);
      \node at (1/2,-1/2) {$\del(l_1)h$};
      \draw (0,0) to node {\tikz{\draw[-](0,0)--(0.1,0);}} (1,0);
      \draw (0,-1) to node {\tikz{\draw[-](0,0)--(0.1,0);}}(1,-1);
      \fill[black] (0,0) circle[radius=0.02];
      \fill[black] (1,0) circle[radius=0.02];
      \fill[black] (0,-1) circle[radius=0.02];
      \fill[black] (1,-1) circle[radius=0.02];   
	  \end{scope}
	    \begin{scope}[xshift=5.0cm]
      \draw (0,0) to node[sloped] {\tikz{\draw[->](0,0)--(0.1,0);}} (0,-1);
      \draw (1,0) to node[sloped] {\tikz{\draw[->](0,0)--(0.1,0);}} (1,-1);
      \draw (0,0) to node[sloped] {\tikz{\draw[-](0,0)--(0.1,0);}} (1,0);
      \draw (0,-1) to node  {\tikz{\draw[-](0,0)--(0.1,0);}} (1,-1);
      \node at (1/2,-1/2) {$\del(l_2 l_1)h$};
      \draw (0,0) to node {\tikz{\draw[-](0,0)--(0.1,0);}} (1,0);
      \draw (0,-1) to node {\tikz{\draw[-](0,0)--(0.1,0);}}(1,-1);
      \fill[black] (0,0) circle[radius=0.02];
      \fill[black] (1,0) circle[radius=0.02];
      \fill[black] (0,-1) circle[radius=0.02];
      \fill[black] (1,-1) circle[radius=0.02];  
	  \end{scope}
	  \node at (3,-1.2) { \rotatebox{90}{$=$}};
	  \begin{scope}[xshift=1.5cm ,yshift=-1.4cm]
      \draw (0,0) to node[sloped] {\tikz{\draw[->](0,0)--(0.1,0);}} (0,-1);
      \draw (1,0) to node[sloped] {\tikz{\draw[->](0,0)--(0.1,0);}} (1,-1);
      \draw (0,0) to node[sloped] {\tikz{\draw[-](0,0)--(0.1,0);}} (1,0);
      \draw (0,-1) to node  {\tikz{\draw[-](0,0)--(0.1,0);}} (1,-1);
      \node at (1/2,-1/2) {$h$};
      \draw (0,0) to node {\tikz{\draw[-](0,0)--(0.1,0);}} (1,0);
      \draw (0,-1) to node {\tikz{\draw[-](0,0)--(0.1,0);}} (1,-1);
      \fill[black] (0,0) circle[radius=0.02];
      \fill[black] (1,0) circle[radius=0.02];
      \fill[black] (0,-1) circle[radius=0.02];
      \fill[black] (1,-1) circle[radius=0.02];
      \draw[->] (1.1,-1/2) to (1.9,-1/2);
      \node at (1.5,-1/2) [above]{$l_2 \#_3 l_1$};
    \begin{scope}[xshift=2.0cm]
      \draw (0,0) to node[sloped] {\tikz{\draw[->](0,0)--(0.1,0);}} (0,-1);
      \draw (1,0) to node[sloped] {\tikz{\draw[->](0,0)--(0.1,0);}} (1,-1);
      \draw (0,0) to node[sloped] {\tikz{\draw[-](0,0)--(0.1,0);}} (1,0);
      \draw (0,-1) to node  {\tikz{\draw[-](0,0)--(0.1,0);}} (1,-1);
      \node at (1/2,-1/2) {$\del(l_2 l_1)h$};
      \draw (0,0) to node {\tikz{\draw[-](0,0)--(0.1,0);}} (1,0);
      \draw (0,-1) to node {\tikz{\draw[-](0,0)--(0.1,0);}} (1,-1);
      \fill[black] (0,0) circle[radius=0.02];
      \fill[black] (1,0) circle[radius=0.02];
      \fill[black] (0,-1) circle[radius=0.02];
      \fill[black] (1,-1) circle[radius=0.02];
	\end{scope}
	\end{scope}
  \end{tikzpicture}
\end{center}
\caption{Vertical compositon of $L$}
\label{d:ver_of_l}
\end{figure}

For each $g_1,g_1',g_2,g_2'\in G$, $h_1,h_1', h_2,h_2'\in H$, and $l_1,l_2 \in L$, such that
\begin{equation*}
    \begin{aligned}
h_1(g_1) &= g_1',\ h_1'(g_1) = g_1'\\
h_2(g_2) &= g_2',\ h_2'(g_2) = g_2'\\
l_1(h_1)&= h_1',\ l_2(h_2) = h_2' .
    \end{aligned}
\end{equation*}
the elements ${\rm id}_{{\rm id}_{g_2}}\#_1 l_1$, and $l_2 \#_1{\rm id}_{{\rm id}_{g_1}}$ can be determined  in a manner analogous to the case of $H$ , and can be represented 
in Figures~\ref{d:hor1_of_l_1} and~\ref{d:hor1_of_l_2}, respectively.
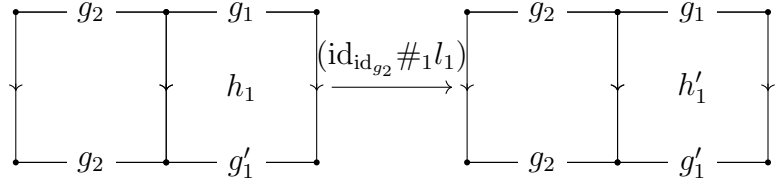
\begin{figure}
\begin{center}
  \begin{tikzpicture}[scale=2]
    \begin{scope}[xshift=0.0cm]
      \draw (0,0) to node[sloped] {\tikz{\draw[->](0,0)--(0.1,0);}} (0,-1);
      \draw (1,0) to node[sloped] {\tikz{\draw[->](0,0)--(0.1,0);}} (1,-1);
      \draw (0,0) to node[sloped] {\tikz{\draw[-](0,0)--(0.1,0);}} (1,0);
      \draw (0,-1) to node  {\tikz{\draw[-](0,0)--(0.1,0);}} (1,-1);
      \node at (1/2,-1/2) {$$};
      \draw (0,0) to node [fill=white] {$g_2$} (1,0);
      \draw (0,-1) to node[fill=white] {$g_2$} (1,-1);
      \fill[black] (0,0) circle[radius=0.02];
      \fill[black] (1,0) circle[radius=0.02];
      \fill[black] (0,-1) circle[radius=0.02];
      \fill[black] (1,-1) circle[radius=0.02];
       \begin{scope}[xshift=1.0cm]
      \draw (0,0) to node[sloped] {\tikz{\draw[->](0,0)--(0.1,0);}} (0,-1);
      \draw (1,0) to node[sloped] {\tikz{\draw[->](0,0)--(0.1,0);}} (1,-1);
      \draw (0,0) to node[sloped] {\tikz{\draw[-](0,0)--(0.1,0);}} (1,0);
      \draw (0,-1) to node  {\tikz{\draw[-](0,0)--(0.1,0);}} (1,-1);
      \node at (1/2,-1/2) {$h_1$};
      \draw (0,0) to node [fill=white] {$g_1$} (1,0);
      \draw (0,-1) to node [fill=white] {$g_1'$} (1,-1);
      \fill[black] (0,0) circle[radius=0.02];
      \fill[black] (1,0) circle[radius=0.02];
      \fill[black] (0,-1) circle[radius=0.02];
      \fill[black] (1,-1) circle[radius=0.02]; 
       \draw[->] (1.1,-1/2) to (1.9,-1/2);  
      \node at (1.5,-1/2) [above]{$({\rm id}_{{\rm id}_{g_2}}\#_1 l_1)$};
	\end{scope}
    \begin{scope}[xshift=3.0cm]
      \draw (0,0) to node[sloped] {\tikz{\draw[->](0,0)--(0.1,0);}} (0,-1);
      \draw (1,0) to node[sloped] {\tikz{\draw[->](0,0)--(0.1,0);}} (1,-1);
      \draw (0,0) to node[sloped] {\tikz{\draw[-](0,0)--(0.1,0);}} (1,0);
      \draw (0,-1) to node  {\tikz{\draw[-](0,0)--(0.1,0);}} (1,-1);
      \node at (1/2,-1/2) {$$};
      \draw (0,0) to node [fill=white] {$g_2$} (1,0);
      \draw (0,-1) to node [fill=white] {$g_2$}(1,-1);
      \fill[black] (0,0) circle[radius=0.02];
      \fill[black] (1,0) circle[radius=0.02];
      \fill[black] (0,-1) circle[radius=0.02];
      \fill[black] (1,-1) circle[radius=0.02];
	\end{scope}
	 \begin{scope}[xshift=4.0cm]
      \draw (0,0) to node[sloped] {\tikz{\draw[->](0,0)--(0.1,0);}} (0,-1);
      \draw (1,0) to node[sloped] {\tikz{\draw[->](0,0)--(0.1,0);}} (1,-1);
      \draw (0,0) to node[sloped] {\tikz{\draw[-](0,0)--(0.1,0);}} (1,0);
      \draw (0,-1) to node  {\tikz{\draw[-](0,0)--(0.1,0);}} (1,-1);
      \node at (1/2,-1/2) {$h_1'$};
      \draw (0,0) to node [fill=white] {$g_1$} (1,0);
      \draw (0,-1) to node [fill=white] {$g_1'$} (1,-1);
      \fill[black] (0,0) circle[radius=0.02];
      \fill[black] (1,0) circle[radius=0.02];
      \fill[black] (0,-1) circle[radius=0.02];
      \fill[black] (1,-1) circle[radius=0.02];
	\end{scope}
	 \end{scope}
  \end{tikzpicture}
\end{center}
\caption{Diagram of ${\rm id}_{{\rm id}_{g_2}}\#_1 l_1$}
\label{d:hor1_of_l_1}
\end{figure}
\begin{figure}
\begin{center}
  \begin{tikzpicture}[scale=2]
    \begin{scope}[xshift=0.0cm]
      \draw (0,0) to node[sloped] {\tikz{\draw[->](0,0)--(0.1,0);}} (0,-1);
      \draw (1,0) to node[sloped] {\tikz{\draw[->](0,0)--(0.1,0);}} (1,-1);
      \draw (0,0) to node[sloped] {\tikz{\draw[-](0,0)--(0.1,0);}} (1,0);
      \draw (0,-1) to node  {\tikz{\draw[-](0,0)--(0.1,0);}} (1,-1);
      \node at (1/2,-1/2) {$h_2$};
      \draw (0,0) to node [fill=white] {$g_2$} (1,0);
      \draw (0,-1) to node[fill=white] {$g_2'$} (1,-1);
      \fill[black] (0,0) circle[radius=0.02];
      \fill[black] (1,0) circle[radius=0.02];
      \fill[black] (0,-1) circle[radius=0.02];
      \fill[black] (1,-1) circle[radius=0.02];
       \begin{scope}[xshift=1.0cm]
      \draw (0,0) to node[sloped] {\tikz{\draw[->](0,0)--(0.1,0);}} (0,-1);
      \draw (1,0) to node[sloped] {\tikz{\draw[->](0,0)--(0.1,0);}} (1,-1);
      \draw (0,0) to node[sloped] {\tikz{\draw[-](0,0)--(0.1,0);}} (1,0);
      \draw (0,-1) to node  {\tikz{\draw[-](0,0)--(0.1,0);}} (1,-1);
      \node at (1/2,-1/2) {$$};
      \draw (0,0) to node [fill=white] {$g_1$} (1,0);
      \draw (0,-1) to node [fill=white] {$g_1$} (1,-1);
      \fill[black] (0,0) circle[radius=0.02];
      \fill[black] (1,0) circle[radius=0.02];
      \fill[black] (0,-1) circle[radius=0.02];
      \fill[black] (1,-1) circle[radius=0.02]; 
       \draw[->] (1.1,-1/2) to (1.9,-1/2);  
      \node at (1.5,-1/2) [above]{$(l_2 \#_1{\rm id}_{{\rm id}_{g_1}})$};
	\end{scope}
    \begin{scope}[xshift=3.0cm]
      \draw (0,0) to node[sloped] {\tikz{\draw[->](0,0)--(0.1,0);}} (0,-1);
      \draw (1,0) to node[sloped] {\tikz{\draw[->](0,0)--(0.1,0);}} (1,-1);
      \draw (0,0) to node[sloped] {\tikz{\draw[-](0,0)--(0.1,0);}} (1,0);
      \draw (0,-1) to node  {\tikz{\draw[-](0,0)--(0.1,0);}} (1,-1);
      \node at (1/2,-1/2) {$h_2'$};
      \draw (0,0) to node [fill=white] {$g_2$} (1,0);
      \draw (0,-1) to node [fill=white] {$g_2'$}(1,-1);
      \fill[black] (0,0) circle[radius=0.02];
      \fill[black] (1,0) circle[radius=0.02];
      \fill[black] (0,-1) circle[radius=0.02];
      \fill[black] (1,-1) circle[radius=0.02];
	\end{scope}
	 \begin{scope}[xshift=4.0cm]
      \draw (0,0) to node[sloped] {\tikz{\draw[->](0,0)--(0.1,0);}} (0,-1);
      \draw (1,0) to node[sloped] {\tikz{\draw[->](0,0)--(0.1,0);}} (1,-1);
      \draw (0,0) to node[sloped] {\tikz{\draw[-](0,0)--(0.1,0);}} (1,0);
      \draw (0,-1) to node  {\tikz{\draw[-](0,0)--(0.1,0);}} (1,-1);
      \node at (1/2,-1/2) {$$};
      \draw (0,0) to node [fill=white] {$g_1$} (1,0);
      \draw (0,-1) to node [fill=white] {$g_1$} (1,-1);
      \fill[black] (0,0) circle[radius=0.02];
      \fill[black] (1,0) circle[radius=0.02];
      \fill[black] (0,-1) circle[radius=0.02];
      \fill[black] (1,-1) circle[radius=0.02];
	\end{scope}
	 \end{scope}
  \end{tikzpicture}
\end{center}
\caption{Diagram of $l_2 \#_1{\rm id}_{{\rm id}_{g_1}}$}
\label{d:hor1_of_l_2}
\end{figure}
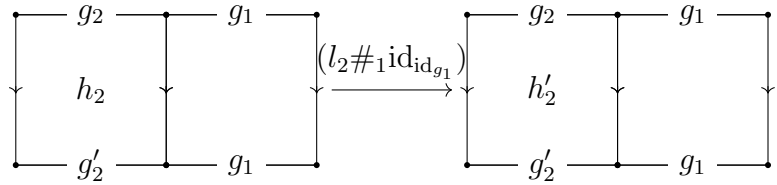
We can correspond ${\rm id}_{{\rm id}_{g_2}}\#_1 l_1$ as $\act{g_2}l_1$, and $l_2 \#_1{\rm id}_{{\rm id}_{g_1}}$ as $l_2$ in 3-crossed module.

 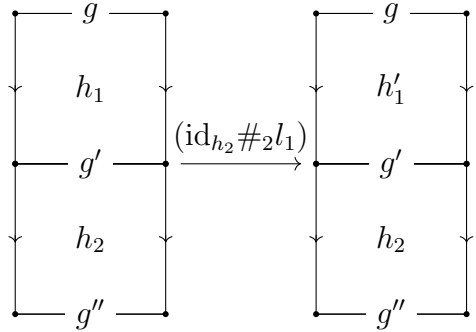
\begin{figure}[!]
 \begin{center}
  \begin{tikzpicture}[scale=2]
    \begin{scope}[xshift=0.0cm]
      \draw (0,0) to node[sloped] {\tikz{\draw[->](0,0)--(0.1,0);}} (0,-1);
      \draw (1,0) to node[sloped] {\tikz{\draw[->](0,0)--(0.1,0);}} (1,-1);
      \draw (0,0) to node[sloped] {\tikz{\draw[-](0,0)--(0.1,0);}} (1,0);
      \draw (0,-1) to node  {\tikz{\draw[-](0,0)--(0.1,0);}} (1,-1);
      \node at (1/2,-1/2) {$h_1$};
      \draw (0,0) to node [fill=white] {$g$} (1,0);
      \draw (0,-1) to node[fill=white] {$g'$} (1,-1);
      \fill[black] (0,0) circle[radius=0.02];
      \fill[black] (1,0) circle[radius=0.02];
      \fill[black] (0,-1) circle[radius=0.02];
      \fill[black] (1,-1) circle[radius=0.02];
       \draw[->] (1.1,-1) to (1.9,-1);  
      \node at (1.5,-1) [above]{$({\rm id}_{h_2}\#_2 l_1)$};
       \begin{scope}[yshift=-1.0cm]
      \draw (0,0) to node[sloped] {\tikz{\draw[->](0,0)--(0.1,0);}} (0,-1);
      \draw (1,0) to node[sloped] {\tikz{\draw[->](0,0)--(0.1,0);}} (1,-1);
      \draw (0,0) to node[sloped] {\tikz{\draw[-](0,0)--(0.1,0);}} (1,0);
      \draw (0,-1) to node  {\tikz{\draw[-](0,0)--(0.1,0);}} (1,-1);
      \node at (1/2,-1/2) {$h_2$};
      \draw (0,0) to node [fill=white] {$g'$} (1,0);
      \draw (0,-1) to node [fill=white] {$g''$} (1,-1);
      \fill[black] (0,0) circle[radius=0.02];
      \fill[black] (1,0) circle[radius=0.02];
      \fill[black] (0,-1) circle[radius=0.02];
      \fill[black] (1,-1) circle[radius=0.02]; 
	\end{scope}
    \begin{scope}[xshift=2.0cm]
      \draw (0,0) to node[sloped] {\tikz{\draw[->](0,0)--(0.1,0);}} (0,-1);
      \draw (1,0) to node[sloped] {\tikz{\draw[->](0,0)--(0.1,0);}} (1,-1);
      \draw (0,0) to node[sloped] {\tikz{\draw[-](0,0)--(0.1,0);}} (1,0);
      \draw (0,-1) to node  {\tikz{\draw[-](0,0)--(0.1,0);}} (1,-1);
      \node at (1/2,-1/2) {$h_1'$};
      \draw (0,0) to node [fill=white] {$g$} (1,0);
      \draw (0,-1) to node [fill=white] {$g'$}(1,-1);
      \fill[black] (0,0) circle[radius=0.02];
      \fill[black] (1,0) circle[radius=0.02];
      \fill[black] (0,-1) circle[radius=0.02];
      \fill[black] (1,-1) circle[radius=0.02];
	\end{scope}
	 \begin{scope}[xshift=2.0cm,yshift=-1.0cm]
      \draw (0,0) to node[sloped] {\tikz{\draw[->](0,0)--(0.1,0);}} (0,-1);
      \draw (1,0) to node[sloped] {\tikz{\draw[->](0,0)--(0.1,0);}} (1,-1);
      \draw (0,0) to node[sloped] {\tikz{\draw[-](0,0)--(0.1,0);}} (1,0);
      \draw (0,-1) to node  {\tikz{\draw[-](0,0)--(0.1,0);}} (1,-1);
      \node at (1/2,-1/2) {$h_2$};
      \draw (0,0) to node [fill=white] {$g'$} (1,0);
      \draw (0,-1) to node [fill=white] {$g''$} (1,-1);
      \fill[black] (0,0) circle[radius=0.02];
      \fill[black] (1,0) circle[radius=0.02];
      \fill[black] (0,-1) circle[radius=0.02];
      \fill[black] (1,-1) circle[radius=0.02];
	\end{scope}
	 \end{scope}
  \end{tikzpicture}
\end{center}
\caption{Diagram of ${\rm id}_{h_2}\#_2 l_1$}
\label{d:hor2_of_l_1}
\end{figure}

For each $g, g', g'' \in G$, $h_1,h_2,h_1',h_2'\in H$, and $l_1, l_2\in L$, such that
\begin{equation*}
    \begin{aligned}
h_1(g) &= g',\ &h_1'(g) = g',\\
h_2(g') &= g'',\ &h_2'(g') = g'',\\
l_1(h_1)&= h_1', \text{ and } &l_2(h_2) = h_2' , 
    \end{aligned}
\end{equation*}
 ${\rm id}_{h_2}\#_2 l_1$ and $l_2 \#_2{\rm id}_{h_1}$ can be determind in similar manner of the case of $H$ and it can be write 
 in Figures~\ref{d:hor2_of_l_1} and~\ref{d:hor2_of_l_2}, respectively.
We can correspond ${\rm id}_{h_2}\#_2 l_1$ to $\act{h_2}l_1$, and $l_2 \#_2{\rm id}_{h_1}$ to $l_2$ in the context of a 3-crossed module.

\begin{figure}[h]
 \begin{center}
  \begin{tikzpicture}[scale=2]
    \begin{scope}[xshift=0.0cm]
      \draw (0,0) to node[sloped] {\tikz{\draw[->](0,0)--(0.1,0);}} (0,-1);
      \draw (1,0) to node[sloped] {\tikz{\draw[->](0,0)--(0.1,0);}} (1,-1);
      \draw (0,0) to node[sloped] {\tikz{\draw[-](0,0)--(0.1,0);}} (1,0);
      \draw (0,-1) to node  {\tikz{\draw[-](0,0)--(0.1,0);}} (1,-1);
      \node at (1/2,-1/2) {$h_1$};
      \draw (0,0) to node [fill=white] {$g$} (1,0);
      \draw (0,-1) to node[fill=white] {$g'$} (1,-1);
      \fill[black] (0,0) circle[radius=0.02];
      \fill[black] (1,0) circle[radius=0.02];
      \fill[black] (0,-1) circle[radius=0.02];
      \fill[black] (1,-1) circle[radius=0.02];
       \draw[->] (1.1,-1) to (1.9,-1);  
      \node at (1.5,-1) [above]{$(l_2 \#_2{\rm id}_{h_1})$};
       \begin{scope}[yshift=-1.0cm]
      \draw (0,0) to node[sloped] {\tikz{\draw[->](0,0)--(0.1,0);}} (0,-1);
      \draw (1,0) to node[sloped] {\tikz{\draw[->](0,0)--(0.1,0);}} (1,-1);
      \draw (0,0) to node[sloped] {\tikz{\draw[-](0,0)--(0.1,0);}} (1,0);
      \draw (0,-1) to node  {\tikz{\draw[-](0,0)--(0.1,0);}} (1,-1);
      \node at (1/2,-1/2) {$h_2$};
      \draw (0,0) to node [fill=white] {$g'$} (1,0);
      \draw (0,-1) to node [fill=white] {$g''$} (1,-1);
      \fill[black] (0,0) circle[radius=0.02];
      \fill[black] (1,0) circle[radius=0.02];
      \fill[black] (0,-1) circle[radius=0.02];
      \fill[black] (1,-1) circle[radius=0.02]; 
	\end{scope}
    \begin{scope}[xshift=2.0cm]
      \draw (0,0) to node[sloped] {\tikz{\draw[->](0,0)--(0.1,0);}} (0,-1);
      \draw (1,0) to node[sloped] {\tikz{\draw[->](0,0)--(0.1,0);}} (1,-1);
      \draw (0,0) to node[sloped] {\tikz{\draw[-](0,0)--(0.1,0);}} (1,0);
      \draw (0,-1) to node  {\tikz{\draw[-](0,0)--(0.1,0);}} (1,-1);
      \node at (1/2,-1/2) {$h_1$};
      \draw (0,0) to node [fill=white] {$g$} (1,0);
      \draw (0,-1) to node [fill=white] {$g'$}(1,-1);
      \fill[black] (0,0) circle[radius=0.02];
      \fill[black] (1,0) circle[radius=0.02];
      \fill[black] (0,-1) circle[radius=0.02];
      \fill[black] (1,-1) circle[radius=0.02];
	\end{scope}
	 \begin{scope}[xshift=2.0cm,yshift=-1.0cm]
      \draw (0,0) to node[sloped] {\tikz{\draw[->](0,0)--(0.1,0);}} (0,-1);
      \draw (1,0) to node[sloped] {\tikz{\draw[->](0,0)--(0.1,0);}} (1,-1);
      \draw (0,0) to node[sloped] {\tikz{\draw[-](0,0)--(0.1,0);}} (1,0);
      \draw (0,-1) to node  {\tikz{\draw[-](0,0)--(0.1,0);}} (1,-1);
      \node at (1/2,-1/2) {$h_2'$};
      \draw (0,0) to node [fill=white] {$g'$} (1,0);
      \draw (0,-1) to node [fill=white] {$g''$} (1,-1);
      \fill[black] (0,0) circle[radius=0.02];
      \fill[black] (1,0) circle[radius=0.02];
      \fill[black] (0,-1) circle[radius=0.02];
      \fill[black] (1,-1) circle[radius=0.02];
	\end{scope}
	 \end{scope}
  \end{tikzpicture}
\end{center}
\caption{Diagram  of $l_2 \#_2{\rm id}_{h_1}$}
\label{d:hor2_of_l_2}
\end{figure}
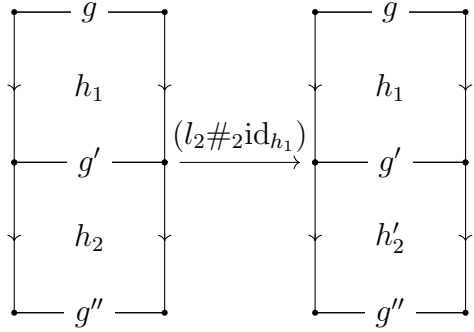

By using the Peiffer lifting, we obtain a map 
$
({\rm id}_{g'_2}\#_1 h_1) \#_2 (h_2\#_1 {\rm id}_{g_1})
$ to $
(h_2\#_1 {\rm id}_{g'_1}) \#_2$ $ ({\rm id}_{g_2}\#_1 h_1).
$
\begin{equation}
  \begin{split} 
\label{pfmap}
\{h_2, \act{g_2}h_1 \}(\act{\partial(h_2)g_2}h_1\cdot h_2) &= \partial(\{h_2, \act{g_2}h_1 \})\cdot \act{\partial(h_2)g_2}h_1\cdot h_2\\
&=h_2\cdot \act{g_2}h_1\cdot h_2^{-1}\cdot \act{\partial(h_2)g_2}h_1^{-1}\cdot  \act{\partial(h_2)g_2} h_1\cdot h_2\\
&= h_2\cdot \act{g_2}h_1
 \end{split}
\end{equation}

Equation~(\ref{pfmap}) shows that $\{h_2, \act{g_2}h_1 \}$ defines a morphism from $({\rm id}_{g'_2}\#_1 h_1) \#_2 (h_2\#_1 {\rm id}_{g_1})$ to $(h_2\#_1 {\rm id}_{g'_1}) \#_2 ({\rm id}_{g_2}\#_1 h_1)$. If we assume $g_1=g_2=e_g$, equation~(\ref{pfmap}) can be written as $\{h_2, h_1 \}(\act{\partial(h_2)}h_1\cdot h_2) = h_2\cdot h_1$. The map $\{h_2, \act{g_2}h_1 \}$ is illustrated in Figure~\ref{d:PF_lift}.
\begin{figure}
\begin{center}
  \begin{tikzpicture}[scale=2]
    \begin{scope}[xshift=0.0cm]
      \draw (0,0) to node[sloped] {\tikz{\draw[->](0,0)--(0.1,0);}} (0,-1);
      \draw (1,0) to node[sloped] {\tikz{\draw[->](0,0)--(0.1,0);}} (1,-1);
      \draw (0,0) to node[sloped] {\tikz{\draw[->](0,0)--(0.1,0);}} (1,0);
      \draw (0,-1) to node  {\tikz{\draw[-](0,0)--(0.1,0);}} (1,-1);
      \node at (1/2,-1/2) {$h_2$};
      \draw (0,0) to node [fill=white] {$g_2$} (1,0);
      \draw (0,-1) to node [fill=white] {$g_2'$} (1,-1);
      \fill[black] (0,0) circle[radius=0.02];
      \fill[black] (1,0) circle[radius=0.02];
      \fill[black] (0,-1) circle[radius=0.02];
      \fill[black] (1,-1) circle[radius=0.02];
    \end{scope}
    \begin{scope}[xshift=1.0cm]
      \draw (0,0) to node[sloped] {\tikz{\draw[->](0,0)--(0.1,0);}} (0,-1);
      \draw (1,0) to node[sloped] {\tikz{\draw[->](0,0)--(0.1,0);}} (1,-1);
      \draw (0,0) to node[sloped] {\tikz{\draw[->](0,0)--(0.1,0);}} (1,0);
      \draw (0,-1) to node  {\tikz{\draw[-](0,0)--(0.1,0);}} (1,-1);
      \node at (1/2,-1/2) {$$};
      \draw (0,0) to node [fill=white] {$g_1$} (1,0);
      \draw (0,-1) to node [fill=white] {$g_1$} (1,-1);
      \fill[black] (0,0) circle[radius=0.02];
      \fill[black] (1,0) circle[radius=0.02];
      \fill[black] (0,-1) circle[radius=0.02];
      \fill[black] (1,-1) circle[radius=0.02];
    \end{scope}
        \begin{scope}[xshift=0.0cm, yshift = -1.0cm]
      \draw (0,0) to node[sloped] {\tikz{\draw[->](0,0)--(0.1,0);}} (0,-1);
      \draw (1,0) to node[sloped] {\tikz{\draw[->](0,0)--(0.1,0);}} (1,-1);
      \draw (0,0) to node[sloped] {\tikz{\draw[->](0,0)--(0.1,0);}} (1,0);
      \draw (0,-1) to node  {\tikz{\draw[-](0,0)--(0.1,0);}} (1,-1);
      \node at (1/2,-1/2) {$$};
      \draw (0,0) to node [fill=white] {$g_2'$} (1,0);
      \draw (0,-1) to node [fill=white] {$g_2'$} (1,-1);
      \fill[black] (0,0) circle[radius=0.02];
      \fill[black] (1,0) circle[radius=0.02];
      \fill[black] (0,-1) circle[radius=0.02];
      \fill[black] (1,-1) circle[radius=0.02];
    \end{scope}
    \begin{scope}[xshift=1.0cm, yshift = -1.0cm]
      \draw (0,0) to node[sloped] {\tikz{\draw[->](0,0)--(0.1,0);}} (0,-1);
      \draw (1,0) to node[sloped] {\tikz{\draw[->](0,0)--(0.1,0);}} (1,-1);
      \draw (0,0) to node[sloped] {\tikz{\draw[->](0,0)--(0.1,0);}} (1,0);
      \draw (0,-1) to node  {\tikz{\draw[-](0,0)--(0.1,0);}} (1,-1);
      \node at (1/2,-1/2) {$h_1$};
      \draw (0,0) to node [fill=white] {$g_1$} (1,0);
      \draw (0,-1) to node [fill=white] {$g_1'$} (1,-1);
      \fill[black] (0,0) circle[radius=0.02];
      \fill[black] (1,0) circle[radius=0.02];
      \fill[black] (0,-1) circle[radius=0.02];
      \fill[black] (1,-1) circle[radius=0.02];
    \end{scope}
       \draw[->] (2.1,-1) to (3.4,-1);
      \node at (2.75,-1) [above]{$\{h_2, \act{g_2}h_1 \}$};
    \begin{scope}[xshift=3.5cm]
      \draw (0,0) to node[sloped] {\tikz{\draw[->](0,0)--(0.1,0);}} (0,-1);
      \draw (1,0) to node[sloped] {\tikz{\draw[->](0,0)--(0.1,0);}} (1,-1);
      \draw (0,0) to node[sloped] {\tikz{\draw[->](0,0)--(0.1,0);}} (1,0);
      \draw (0,-1) to node  {\tikz{\draw[-](0,0)--(0.1,0);}} (1,-1);
      \node at (1/2,-1/2) {$$};
      \draw (0,0) to node [fill=white] {$g_2$} (1,0);
      \draw (0,-1) to node [fill=white] {$g_2$} (1,-1);
      \fill[black] (0,0) circle[radius=0.02];
      \fill[black] (1,0) circle[radius=0.02];
      \fill[black] (0,-1) circle[radius=0.02];
      \fill[black] (1,-1) circle[radius=0.02];
    \end{scope}
    \begin{scope}[xshift=4.5cm]
      \draw (0,0) to node[sloped] {\tikz{\draw[->](0,0)--(0.1,0);}} (0,-1);
      \draw (1,0) to node[sloped] {\tikz{\draw[->](0,0)--(0.1,0);}} (1,-1);
      \draw (0,0) to node[sloped] {\tikz{\draw[->](0,0)--(0.1,0);}} (1,0);
      \draw (0,-1) to node  {\tikz{\draw[-](0,0)--(0.1,0);}} (1,-1);
      \node at (1/2,-1/2) {$h_1$};
      \draw (0,0) to node [fill=white] {$g_1$} (1,0);
      \draw (0,-1) to node [fill=white] {$g_1'$} (1,-1);
      \fill[black] (0,0) circle[radius=0.02];
      \fill[black] (1,0) circle[radius=0.02];
      \fill[black] (0,-1) circle[radius=0.02];
      \fill[black] (1,-1) circle[radius=0.02];
    \end{scope}
        \begin{scope}[xshift=3.5cm, yshift = -1.0cm]
      \draw (0,0) to node[sloped] {\tikz{\draw[->](0,0)--(0.1,0);}} (0,-1);
      \draw (1,0) to node[sloped] {\tikz{\draw[->](0,0)--(0.1,0);}} (1,-1);
      \draw (0,0) to node[sloped] {\tikz{\draw[->](0,0)--(0.1,0);}} (1,0);
      \draw (0,-1) to node  {\tikz{\draw[-](0,0)--(0.1,0);}} (1,-1);
      \node at (1/2,-1/2) {$h_2$};
      \draw (0,0) to node [fill=white] {$g_2$} (1,0);
      \draw (0,-1) to node [fill=white] {$g_2'$} (1,-1);
      \fill[black] (0,0) circle[radius=0.02];
      \fill[black] (1,0) circle[radius=0.02];
      \fill[black] (0,-1) circle[radius=0.02];
      \fill[black] (1,-1) circle[radius=0.02];
    \end{scope}
    \begin{scope}[xshift=4.5cm, yshift = -1.0cm]
      \draw (0,0) to node[sloped] {\tikz{\draw[->](0,0)--(0.1,0);}} (0,-1);
      \draw (1,0) to node[sloped] {\tikz{\draw[->](0,0)--(0.1,0);}} (1,-1);
      \draw (0,0) to node[sloped] {\tikz{\draw[->](0,0)--(0.1,0);}} (1,0);
      \draw (0,-1) to node  {\tikz{\draw[-](0,0)--(0.1,0);}} (1,-1);
      \node at (1/2,-1/2) {$$};
      \draw (0,0) to node [fill=white] {$g_1'$} (1,0);
      \draw (0,-1) to node [fill=white] {$g_1'$} (1,-1);
      \fill[black] (0,0) circle[radius=0.02];
      \fill[black] (1,0) circle[radius=0.02];
      \fill[black] (0,-1) circle[radius=0.02];
      \fill[black] (1,-1) circle[radius=0.02];
    \end{scope}
  \end{tikzpicture}
\end{center}
\caption{Diagram  of Peiffer lifting}
\label{d:PF_lift}
\end{figure}
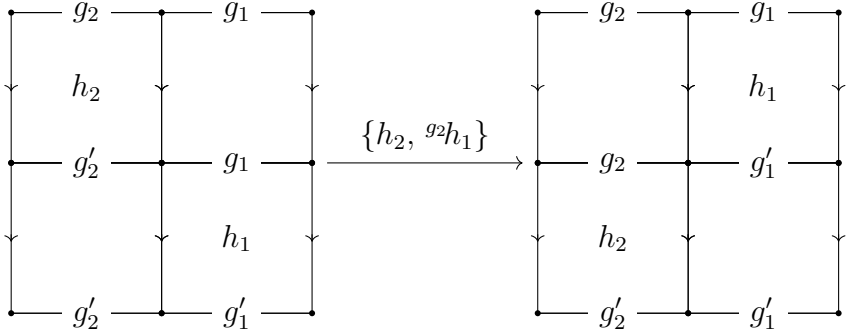

By referring to the diagram of the Peiffer lifting, we can understand  we can visually understand Properties~\ref{lpl} and ~\ref{rpl} of the 2-crossed module. Let $h_1, h_2, h_3 \in H$, and consider $h_3 \#_1 h_2 \# h_1 (e_g \#_1 e_g \#_1 e_g ) = \del(h_3) \#_1 \del(h_2) \#_1 \del(h_1)$.  Property~\ref{lpl} of the 2-crossed module can be represented in Figure~\ref{d:PF_lift_left}, and Property~\ref{rpl} can be represented in Figure~\ref{d:PF_lift_right}.
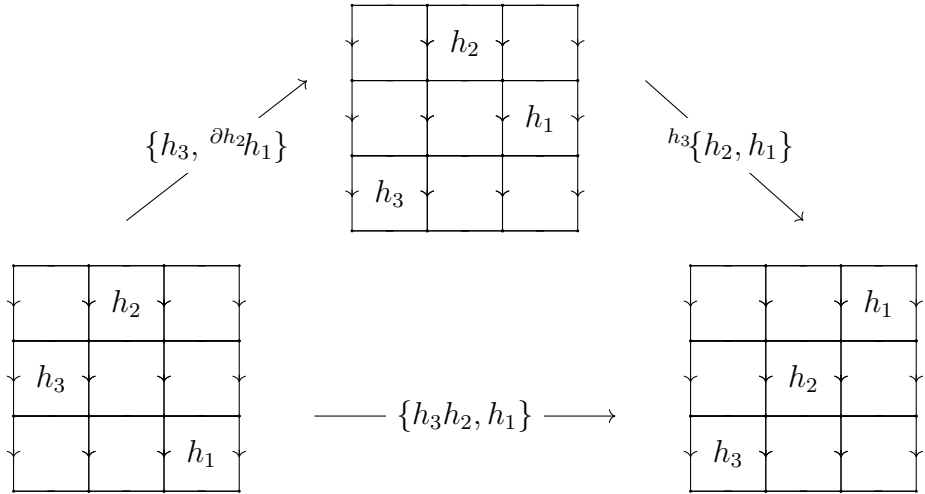
\begin{figure}
\begin{center}
  \begin{tikzpicture}[scale=1]
  	\begin{scope}[xshift=0.0cm]
	      \draw (0,0) to node[sloped] {\tikz{\draw[->](0,0)--(0.1,0);}} (0,-1);
	      \draw (1,0) to node[sloped] {\tikz{\draw[->](0,0)--(0.1,0);}} (1,-1);
	      \draw (0,0) to node[sloped] {\tikz{\draw[-](0,0)--(0.1,0);}} (1,0);
	      \draw (0,-1) to node  {\tikz{\draw[-](0,0)--(0.1,0);}} (1,-1);
	      \node at (1/2,-1/2) {$$};
	      \draw (0,0) to node {\tikz{\draw[-](0,0)--(0.1,0);}}(1,0);
	      \draw (0,-1) to node {\tikz{\draw[-](0,0)--(0.1,0);}} (1,-1);
	      \fill[black] (0,0) circle[radius=0.02];
	      \fill[black] (1,0) circle[radius=0.02];
	      \fill[black] (0,-1) circle[radius=0.02];
	      \fill[black] (1,-1) circle[radius=0.02];
	    \begin{scope}[xshift=1.0cm]
	      \draw (0,0) to node[sloped] {\tikz{\draw[->](0,0)--(0.1,0);}} (0,-1);
	      \draw (1,0) to node[sloped] {\tikz{\draw[->](0,0)--(0.1,0);}} (1,-1);
	      \draw (0,0) to node[sloped] {\tikz{\draw[-](0,0)--(0.1,0);}} (1,0);
	      \draw (0,-1) to node  {\tikz{\draw[-](0,0)--(0.1,0);}} (1,-1);
	      \node at (1/2,-1/2) {$h_2$};
	      \draw (0,0) to node {\tikz{\draw[-](0,0)--(0.1,0);}} (1,0);
	      \draw (0,-1) to node {\tikz{\draw[-](0,0)--(0.1,0);}} (1,-1);
	      \fill[black] (0,0) circle[radius=0.02];
	      \fill[black] (1,0) circle[radius=0.02];
	      \fill[black] (0,-1) circle[radius=0.02];
	      \fill[black] (1,-1) circle[radius=0.02];
	    \end{scope}
	    	 \begin{scope}[xshift=2.0cm]
	      \draw (0,0) to node[sloped] {\tikz{\draw[->](0,0)--(0.1,0);}} (0,-1);
	      \draw (1,0) to node[sloped] {\tikz{\draw[->](0,0)--(0.1,0);}} (1,-1);
	      \draw (0,0) to node[sloped] {\tikz{\draw[-](0,0)--(0.1,0);}} (1,0);
	      \draw (0,-1) to node  {\tikz{\draw[-](0,0)--(0.1,0);}} (1,-1);
	      \node at (1/2,-1/2) {$$};
	      \draw (0,0) to node {\tikz{\draw[-](0,0)--(0.1,0);}} (1,0);
	      \draw (0,-1) to node {\tikz{\draw[-](0,0)--(0.1,0);}} (1,-1);
	      \fill[black] (0,0) circle[radius=0.02];
	      \fill[black] (1,0) circle[radius=0.02];
	      \fill[black] (0,-1) circle[radius=0.02];
	      \fill[black] (1,-1) circle[radius=0.02];
	    \end{scope}
	        \begin{scope}[xshift=0.0cm, yshift = -1.0cm]
	      \draw (0,0) to node[sloped] {\tikz{\draw[->](0,0)--(0.1,0);}} (0,-1);
	      \draw (1,0) to node[sloped] {\tikz{\draw[->](0,0)--(0.1,0);}} (1,-1);
	      \draw (0,0) to node[sloped] {\tikz{\draw[-](0,0)--(0.1,0);}} (1,0);
	      \draw (0,-1) to node  {\tikz{\draw[-](0,0)--(0.1,0);}} (1,-1);
	      \node at (1/2,-1/2) {$h_3$};
	      \draw (0,0) to node {\tikz{\draw[-](0,0)--(0.1,0);}} (1,0);
	      \draw (0,-1) to node {\tikz{\draw[-](0,0)--(0.1,0);}} (1,-1);
	      \fill[black] (0,0) circle[radius=0.02];
	      \fill[black] (1,0) circle[radius=0.02];
	      \fill[black] (0,-1) circle[radius=0.02];
	      \fill[black] (1,-1) circle[radius=0.02];
	    \end{scope}
	    \begin{scope}[xshift=1.0cm, yshift = -1.0cm]
	      \draw (0,0) to node[sloped] {\tikz{\draw[->](0,0)--(0.1,0);}} (0,-1);
	      \draw (1,0) to node[sloped] {\tikz{\draw[->](0,0)--(0.1,0);}} (1,-1);
	      \draw (0,0) to node[sloped] {\tikz{\draw[-](0,0)--(0.1,0);}} (1,0);
	      \draw (0,-1) to node  {\tikz{\draw[-](0,0)--(0.1,0);}} (1,-1);
	      \node at (1/2,-1/2) {$$};
	      \draw (0,0) to node {\tikz{\draw[-](0,0)--(0.1,0);}} (1,0);
	      \draw (0,-1) to node {\tikz{\draw[-](0,0)--(0.1,0);}} (1,-1);
	      \fill[black] (0,0) circle[radius=0.02];
	      \fill[black] (1,0) circle[radius=0.02];
	      \fill[black] (0,-1) circle[radius=0.02];
	      \fill[black] (1,-1) circle[radius=0.02];
	    \end{scope}
	     \begin{scope}[xshift=2.0cm, yshift = -1.0cm]
	      \draw (0,0) to node[sloped] {\tikz{\draw[->](0,0)--(0.1,0);}} (0,-1);
	      \draw (1,0) to node[sloped] {\tikz{\draw[->](0,0)--(0.1,0);}} (1,-1);
	      \draw (0,0) to node[sloped] {\tikz{\draw[-](0,0)--(0.1,0);}} (1,0);
	      \draw (0,-1) to node  {\tikz{\draw[-](0,0)--(0.1,0);}} (1,-1);
	      \node at (1/2,-1/2) {$$};
	      \draw (0,0) to node {\tikz{\draw[-](0,0)--(0.1,0);}} (1,0);
	      \draw (0,-1) to node {\tikz{\draw[-](0,0)--(0.1,0);}} (1,-1);
	      \fill[black] (0,0) circle[radius=0.02];
	      \fill[black] (1,0) circle[radius=0.02];
	      \fill[black] (0,-1) circle[radius=0.02];
	      \fill[black] (1,-1) circle[radius=0.02];
	    \end{scope}
	     \begin{scope}[xshift=0.0cm, yshift = -2.0cm]
	      \draw (0,0) to node[sloped] {\tikz{\draw[->](0,0)--(0.1,0);}} (0,-1);
	      \draw (1,0) to node[sloped] {\tikz{\draw[->](0,0)--(0.1,0);}} (1,-1);
	      \draw (0,0) to node[sloped] {\tikz{\draw[-](0,0)--(0.1,0);}} (1,0);
	      \draw (0,-1) to node  {\tikz{\draw[-](0,0)--(0.1,0);}} (1,-1);
	      \node at (1/2,-1/2) {$$};
	      \draw (0,0) to node {\tikz{\draw[-](0,0)--(0.1,0);}} (1,0);
	      \draw (0,-1) to node {\tikz{\draw[-](0,0)--(0.1,0);}} (1,-1);
	      \fill[black] (0,0) circle[radius=0.02];
	      \fill[black] (1,0) circle[radius=0.02];
	      \fill[black] (0,-1) circle[radius=0.02];
	      \fill[black] (1,-1) circle[radius=0.02];
	    \end{scope}
	    \begin{scope}[xshift=1.0cm, yshift = -2.0cm]
	      \draw (0,0) to node[sloped] {\tikz{\draw[->](0,0)--(0.1,0);}} (0,-1);
	      \draw (1,0) to node[sloped] {\tikz{\draw[->](0,0)--(0.1,0);}} (1,-1);
	      \draw (0,0) to node[sloped] {\tikz{\draw[-](0,0)--(0.1,0);}} (1,0);
	      \draw (0,-1) to node  {\tikz{\draw[-](0,0)--(0.1,0);}} (1,-1);
	      \node at (1/2,-1/2) {$$};
	      \draw (0,0) to node {\tikz{\draw[-](0,0)--(0.1,0);}}(1,0);
	      \draw (0,-1) to node {\tikz{\draw[-](0,0)--(0.1,0);}} (1,-1);
	      \fill[black] (0,0) circle[radius=0.02];
	      \fill[black] (1,0) circle[radius=0.02];
	      \fill[black] (0,-1) circle[radius=0.02];
	      \fill[black] (1,-1) circle[radius=0.02];
	    \end{scope}
	     \begin{scope}[xshift=2.0cm, yshift = -2.0cm]
	      \draw (0,0) to node[sloped] {\tikz{\draw[->](0,0)--(0.1,0);}} (0,-1);
	      \draw (1,0) to node[sloped] {\tikz{\draw[->](0,0)--(0.1,0);}} (1,-1);
	      \draw (0,0) to node[sloped] {\tikz{\draw[-](0,0)--(0.1,0);}} (1,0);
	      \draw (0,-1) to node  {\tikz{\draw[-](0,0)--(0.1,0);}} (1,-1);
	      \node at (1/2,-1/2) {$h_1$};
	      \draw (0,0) to node {\tikz{\draw[-](0,0)--(0.1,0);}}(1,0);
	      \draw (0,-1) to node {\tikz{\draw[-](0,0)--(0.1,0);}} (1,-1);
	      \fill[black] (0,0) circle[radius=0.02];
	      \fill[black] (1,0) circle[radius=0.02];
	      \fill[black] (0,-1) circle[radius=0.02];
	      \fill[black] (1,-1) circle[radius=0.02];
	    \end{scope}
    \end{scope}
    \begin{scope}[xshift=4.5cm, yshift = 2*1.732cm]
	      \draw (0,0) to node[sloped] {\tikz{\draw[->](0,0)--(0.1,0);}} (0,-1);
	      \draw (1,0) to node[sloped] {\tikz{\draw[->](0,0)--(0.1,0);}} (1,-1);
	      \draw (0,0) to node[sloped] {\tikz{\draw[-](0,0)--(0.1,0);}} (1,0);
	      \draw (0,-1) to node  {\tikz{\draw[-](0,0)--(0.1,0);}} (1,-1);
	      \node at (1/2,-1/2) {$$};
	      \draw (0,0) to node {\tikz{\draw[-](0,0)--(0.1,0);}}(1,0);
	      \draw (0,-1) to node {\tikz{\draw[-](0,0)--(0.1,0);}} (1,-1);
	      \fill[black] (0,0) circle[radius=0.02];
	      \fill[black] (1,0) circle[radius=0.02];
	      \fill[black] (0,-1) circle[radius=0.02];
	      \fill[black] (1,-1) circle[radius=0.02];
	    \begin{scope}[xshift=1.0cm]
	      \draw (0,0) to node[sloped] {\tikz{\draw[->](0,0)--(0.1,0);}} (0,-1);
	      \draw (1,0) to node[sloped] {\tikz{\draw[->](0,0)--(0.1,0);}} (1,-1);
	      \draw (0,0) to node[sloped] {\tikz{\draw[-](0,0)--(0.1,0);}} (1,0);
	      \draw (0,-1) to node  {\tikz{\draw[-](0,0)--(0.1,0);}} (1,-1);
	      \node at (1/2,-1/2) {$h_2$};
	      \draw (0,0) to node {\tikz{\draw[-](0,0)--(0.1,0);}} (1,0);
	      \draw (0,-1) to node {\tikz{\draw[-](0,0)--(0.1,0);}} (1,-1);
	      \fill[black] (0,0) circle[radius=0.02];
	      \fill[black] (1,0) circle[radius=0.02];
	      \fill[black] (0,-1) circle[radius=0.02];
	      \fill[black] (1,-1) circle[radius=0.02];
	    \end{scope}
	    	 \begin{scope}[xshift=2.0cm]
	      \draw (0,0) to node[sloped] {\tikz{\draw[->](0,0)--(0.1,0);}} (0,-1);
	      \draw (1,0) to node[sloped] {\tikz{\draw[->](0,0)--(0.1,0);}} (1,-1);
	      \draw (0,0) to node[sloped] {\tikz{\draw[-](0,0)--(0.1,0);}} (1,0);
	      \draw (0,-1) to node  {\tikz{\draw[-](0,0)--(0.1,0);}} (1,-1);
	      \node at (1/2,-1/2) {$$};
	      \draw (0,0) to node {\tikz{\draw[-](0,0)--(0.1,0);}} (1,0);
	      \draw (0,-1) to node {\tikz{\draw[-](0,0)--(0.1,0);}} (1,-1);
	      \fill[black] (0,0) circle[radius=0.02];
	      \fill[black] (1,0) circle[radius=0.02];
	      \fill[black] (0,-1) circle[radius=0.02];
	      \fill[black] (1,-1) circle[radius=0.02];
	    \end{scope}
	        \begin{scope}[xshift=0.0cm, yshift = -1.0cm]
	      \draw (0,0) to node[sloped] {\tikz{\draw[->](0,0)--(0.1,0);}} (0,-1);
	      \draw (1,0) to node[sloped] {\tikz{\draw[->](0,0)--(0.1,0);}} (1,-1);
	      \draw (0,0) to node[sloped] {\tikz{\draw[-](0,0)--(0.1,0);}} (1,0);
	      \draw (0,-1) to node  {\tikz{\draw[-](0,0)--(0.1,0);}} (1,-1);
	      \node at (1/2,-1/2) {$$};
	      \draw (0,0) to node {\tikz{\draw[-](0,0)--(0.1,0);}} (1,0);
	      \draw (0,-1) to node {\tikz{\draw[-](0,0)--(0.1,0);}} (1,-1);
	      \fill[black] (0,0) circle[radius=0.02];
	      \fill[black] (1,0) circle[radius=0.02];
	      \fill[black] (0,-1) circle[radius=0.02];
	      \fill[black] (1,-1) circle[radius=0.02];
	    \end{scope}
	    \begin{scope}[xshift=1.0cm, yshift = -1.0cm]
	      \draw (0,0) to node[sloped] {\tikz{\draw[->](0,0)--(0.1,0);}} (0,-1);
	      \draw (1,0) to node[sloped] {\tikz{\draw[->](0,0)--(0.1,0);}} (1,-1);
	      \draw (0,0) to node[sloped] {\tikz{\draw[-](0,0)--(0.1,0);}} (1,0);
	      \draw (0,-1) to node  {\tikz{\draw[-](0,0)--(0.1,0);}} (1,-1);
	      \node at (1/2,-1/2) {$$};
	      \draw (0,0) to node {\tikz{\draw[-](0,0)--(0.1,0);}} (1,0);
	      \draw (0,-1) to node {\tikz{\draw[-](0,0)--(0.1,0);}} (1,-1);
	      \fill[black] (0,0) circle[radius=0.02];
	      \fill[black] (1,0) circle[radius=0.02];
	      \fill[black] (0,-1) circle[radius=0.02];
	      \fill[black] (1,-1) circle[radius=0.02];
	    \end{scope}
	     \begin{scope}[xshift=2.0cm, yshift = -1.0cm]
	      \draw (0,0) to node[sloped] {\tikz{\draw[->](0,0)--(0.1,0);}} (0,-1);
	      \draw (1,0) to node[sloped] {\tikz{\draw[->](0,0)--(0.1,0);}} (1,-1);
	      \draw (0,0) to node[sloped] {\tikz{\draw[-](0,0)--(0.1,0);}} (1,0);
	      \draw (0,-1) to node  {\tikz{\draw[-](0,0)--(0.1,0);}} (1,-1);
	      \node at (1/2,-1/2) {$h_1$};
	      \draw (0,0) to node {\tikz{\draw[-](0,0)--(0.1,0);}} (1,0);
	      \draw (0,-1) to node {\tikz{\draw[-](0,0)--(0.1,0);}} (1,-1);
	      \fill[black] (0,0) circle[radius=0.02];
	      \fill[black] (1,0) circle[radius=0.02];
	      \fill[black] (0,-1) circle[radius=0.02];
	      \fill[black] (1,-1) circle[radius=0.02];
	    \end{scope}
	     \begin{scope}[xshift=0.0cm, yshift = -2.0cm]
	      \draw (0,0) to node[sloped] {\tikz{\draw[->](0,0)--(0.1,0);}} (0,-1);
	      \draw (1,0) to node[sloped] {\tikz{\draw[->](0,0)--(0.1,0);}} (1,-1);
	      \draw (0,0) to node[sloped] {\tikz{\draw[-](0,0)--(0.1,0);}} (1,0);
	      \draw (0,-1) to node  {\tikz{\draw[-](0,0)--(0.1,0);}} (1,-1);
	      \node at (1/2,-1/2) {$h_3$};
	      \draw (0,0) to node {\tikz{\draw[-](0,0)--(0.1,0);}} (1,0);
	      \draw (0,-1) to node {\tikz{\draw[-](0,0)--(0.1,0);}} (1,-1);
	      \fill[black] (0,0) circle[radius=0.02];
	      \fill[black] (1,0) circle[radius=0.02];
	      \fill[black] (0,-1) circle[radius=0.02];
	      \fill[black] (1,-1) circle[radius=0.02];
	    \end{scope}
	    \begin{scope}[xshift=1.0cm, yshift = -2.0cm]
	      \draw (0,0) to node[sloped] {\tikz{\draw[->](0,0)--(0.1,0);}} (0,-1);
	      \draw (1,0) to node[sloped] {\tikz{\draw[->](0,0)--(0.1,0);}} (1,-1);
	      \draw (0,0) to node[sloped] {\tikz{\draw[-](0,0)--(0.1,0);}} (1,0);
	      \draw (0,-1) to node  {\tikz{\draw[-](0,0)--(0.1,0);}} (1,-1);
	      \node at (1/2,-1/2) {$$};
	      \draw (0,0) to node {\tikz{\draw[-](0,0)--(0.1,0);}}(1,0);
	      \draw (0,-1) to node {\tikz{\draw[-](0,0)--(0.1,0);}} (1,-1);
	      \fill[black] (0,0) circle[radius=0.02];
	      \fill[black] (1,0) circle[radius=0.02];
	      \fill[black] (0,-1) circle[radius=0.02];
	      \fill[black] (1,-1) circle[radius=0.02];
	    \end{scope}
	     \begin{scope}[xshift=2.0cm, yshift = -2.0cm]
	      \draw (0,0) to node[sloped] {\tikz{\draw[->](0,0)--(0.1,0);}} (0,-1);
	      \draw (1,0) to node[sloped] {\tikz{\draw[->](0,0)--(0.1,0);}} (1,-1);
	      \draw (0,0) to node[sloped] {\tikz{\draw[-](0,0)--(0.1,0);}} (1,0);
	      \draw (0,-1) to node  {\tikz{\draw[-](0,0)--(0.1,0);}} (1,-1);
	      \node at (1/2,-1/2) {$$};
	      \draw (0,0) to node {\tikz{\draw[-](0,0)--(0.1,0);}}(1,0);
	      \draw (0,-1) to node {\tikz{\draw[-](0,0)--(0.1,0);}} (1,-1);
	      \fill[black] (0,0) circle[radius=0.02];
	      \fill[black] (1,0) circle[radius=0.02];
	      \fill[black] (0,-1) circle[radius=0.02];
	      \fill[black] (1,-1) circle[radius=0.02];
	    \end{scope}
    \end{scope}
    \begin{scope}[xshift=9.0cm]
	      \draw (0,0) to node[sloped] {\tikz{\draw[->](0,0)--(0.1,0);}} (0,-1);
	      \draw (1,0) to node[sloped] {\tikz{\draw[->](0,0)--(0.1,0);}} (1,-1);
	      \draw (0,0) to node[sloped] {\tikz{\draw[-](0,0)--(0.1,0);}} (1,0);
	      \draw (0,-1) to node  {\tikz{\draw[-](0,0)--(0.1,0);}} (1,-1);
	      \node at (1/2,-1/2) {$$};
	      \draw (0,0) to node {\tikz{\draw[-](0,0)--(0.1,0);}}(1,0);
	      \draw (0,-1) to node {\tikz{\draw[-](0,0)--(0.1,0);}} (1,-1);
	      \fill[black] (0,0) circle[radius=0.02];
	      \fill[black] (1,0) circle[radius=0.02];
	      \fill[black] (0,-1) circle[radius=0.02];
	      \fill[black] (1,-1) circle[radius=0.02];
	    \begin{scope}[xshift=1.0cm]
	      \draw (0,0) to node[sloped] {\tikz{\draw[->](0,0)--(0.1,0);}} (0,-1);
	      \draw (1,0) to node[sloped] {\tikz{\draw[->](0,0)--(0.1,0);}} (1,-1);
	      \draw (0,0) to node[sloped] {\tikz{\draw[-](0,0)--(0.1,0);}} (1,0);
	      \draw (0,-1) to node  {\tikz{\draw[-](0,0)--(0.1,0);}} (1,-1);
	      \node at (1/2,-1/2) {$$};
	      \draw (0,0) to node {\tikz{\draw[-](0,0)--(0.1,0);}} (1,0);
	      \draw (0,-1) to node {\tikz{\draw[-](0,0)--(0.1,0);}} (1,-1);
	      \fill[black] (0,0) circle[radius=0.02];
	      \fill[black] (1,0) circle[radius=0.02];
	      \fill[black] (0,-1) circle[radius=0.02];
	      \fill[black] (1,-1) circle[radius=0.02];
	    \end{scope}
	    	 \begin{scope}[xshift=2.0cm]
	      \draw (0,0) to node[sloped] {\tikz{\draw[->](0,0)--(0.1,0);}} (0,-1);
	      \draw (1,0) to node[sloped] {\tikz{\draw[->](0,0)--(0.1,0);}} (1,-1);
	      \draw (0,0) to node[sloped] {\tikz{\draw[-](0,0)--(0.1,0);}} (1,0);
	      \draw (0,-1) to node  {\tikz{\draw[-](0,0)--(0.1,0);}} (1,-1);
	      \node at (1/2,-1/2) {$h_1$};
	      \draw (0,0) to node {\tikz{\draw[-](0,0)--(0.1,0);}} (1,0);
	      \draw (0,-1) to node {\tikz{\draw[-](0,0)--(0.1,0);}} (1,-1);
	      \fill[black] (0,0) circle[radius=0.02];
	      \fill[black] (1,0) circle[radius=0.02];
	      \fill[black] (0,-1) circle[radius=0.02];
	      \fill[black] (1,-1) circle[radius=0.02];
	    \end{scope}
	        \begin{scope}[xshift=0.0cm, yshift = -1.0cm]
	      \draw (0,0) to node[sloped] {\tikz{\draw[->](0,0)--(0.1,0);}} (0,-1);
	      \draw (1,0) to node[sloped] {\tikz{\draw[->](0,0)--(0.1,0);}} (1,-1);
	      \draw (0,0) to node[sloped] {\tikz{\draw[-](0,0)--(0.1,0);}} (1,0);
	      \draw (0,-1) to node  {\tikz{\draw[-](0,0)--(0.1,0);}} (1,-1);
	      \node at (1/2,-1/2) {$$};
	      \draw (0,0) to node {\tikz{\draw[-](0,0)--(0.1,0);}} (1,0);
	      \draw (0,-1) to node {\tikz{\draw[-](0,0)--(0.1,0);}} (1,-1);
	      \fill[black] (0,0) circle[radius=0.02];
	      \fill[black] (1,0) circle[radius=0.02];
	      \fill[black] (0,-1) circle[radius=0.02];
	      \fill[black] (1,-1) circle[radius=0.02];
	    \end{scope}
	    \begin{scope}[xshift=1.0cm, yshift = -1.0cm]
	      \draw (0,0) to node[sloped] {\tikz{\draw[->](0,0)--(0.1,0);}} (0,-1);
	      \draw (1,0) to node[sloped] {\tikz{\draw[->](0,0)--(0.1,0);}} (1,-1);
	      \draw (0,0) to node[sloped] {\tikz{\draw[-](0,0)--(0.1,0);}} (1,0);
	      \draw (0,-1) to node  {\tikz{\draw[-](0,0)--(0.1,0);}} (1,-1);
	      \node at (1/2,-1/2) {$h_2$};
	      \draw (0,0) to node {\tikz{\draw[-](0,0)--(0.1,0);}} (1,0);
	      \draw (0,-1) to node {\tikz{\draw[-](0,0)--(0.1,0);}} (1,-1);
	      \fill[black] (0,0) circle[radius=0.02];
	      \fill[black] (1,0) circle[radius=0.02];
	      \fill[black] (0,-1) circle[radius=0.02];
	      \fill[black] (1,-1) circle[radius=0.02];
	    \end{scope}
	     \begin{scope}[xshift=2.0cm, yshift = -1.0cm]
	      \draw (0,0) to node[sloped] {\tikz{\draw[->](0,0)--(0.1,0);}} (0,-1);
	      \draw (1,0) to node[sloped] {\tikz{\draw[->](0,0)--(0.1,0);}} (1,-1);
	      \draw (0,0) to node[sloped] {\tikz{\draw[-](0,0)--(0.1,0);}} (1,0);
	      \draw (0,-1) to node  {\tikz{\draw[-](0,0)--(0.1,0);}} (1,-1);
	      \node at (1/2,-1/2) {$$};
	      \draw (0,0) to node {\tikz{\draw[-](0,0)--(0.1,0);}} (1,0);
	      \draw (0,-1) to node {\tikz{\draw[-](0,0)--(0.1,0);}} (1,-1);
	      \fill[black] (0,0) circle[radius=0.02];
	      \fill[black] (1,0) circle[radius=0.02];
	      \fill[black] (0,-1) circle[radius=0.02];
	      \fill[black] (1,-1) circle[radius=0.02];
	    \end{scope}
	     \begin{scope}[xshift=0.0cm, yshift = -2.0cm]
	      \draw (0,0) to node[sloped] {\tikz{\draw[->](0,0)--(0.1,0);}} (0,-1);
	      \draw (1,0) to node[sloped] {\tikz{\draw[->](0,0)--(0.1,0);}} (1,-1);
	      \draw (0,0) to node[sloped] {\tikz{\draw[-](0,0)--(0.1,0);}} (1,0);
	      \draw (0,-1) to node  {\tikz{\draw[-](0,0)--(0.1,0);}} (1,-1);
	      \node at (1/2,-1/2) {$h_3$};
	      \draw (0,0) to node {\tikz{\draw[-](0,0)--(0.1,0);}} (1,0);
	      \draw (0,-1) to node {\tikz{\draw[-](0,0)--(0.1,0);}} (1,-1);
	      \fill[black] (0,0) circle[radius=0.02];
	      \fill[black] (1,0) circle[radius=0.02];
	      \fill[black] (0,-1) circle[radius=0.02];
	      \fill[black] (1,-1) circle[radius=0.02];
	    \end{scope}
	    \begin{scope}[xshift=1.0cm, yshift = -2.0cm]
	      \draw (0,0) to node[sloped] {\tikz{\draw[->](0,0)--(0.1,0);}} (0,-1);
	      \draw (1,0) to node[sloped] {\tikz{\draw[->](0,0)--(0.1,0);}} (1,-1);
	      \draw (0,0) to node[sloped] {\tikz{\draw[-](0,0)--(0.1,0);}} (1,0);
	      \draw (0,-1) to node  {\tikz{\draw[-](0,0)--(0.1,0);}} (1,-1);
	      \node at (1/2,-1/2) {$$};
	      \draw (0,0) to node {\tikz{\draw[-](0,0)--(0.1,0);}}(1,0);
	      \draw (0,-1) to node {\tikz{\draw[-](0,0)--(0.1,0);}} (1,-1);
	      \fill[black] (0,0) circle[radius=0.02];
	      \fill[black] (1,0) circle[radius=0.02];
	      \fill[black] (0,-1) circle[radius=0.02];
	      \fill[black] (1,-1) circle[radius=0.02];
	    \end{scope}
	     \begin{scope}[xshift=2.0cm, yshift = -2.0cm]
	      \draw (0,0) to node[sloped] {\tikz{\draw[->](0,0)--(0.1,0);}} (0,-1);
	      \draw (1,0) to node[sloped] {\tikz{\draw[->](0,0)--(0.1,0);}} (1,-1);
	      \draw (0,0) to node[sloped] {\tikz{\draw[-](0,0)--(0.1,0);}} (1,0);
	      \draw (0,-1) to node  {\tikz{\draw[-](0,0)--(0.1,0);}} (1,-1);
	      \node at (1/2,-1/2) {$$};
	      \draw (0,0) to node {\tikz{\draw[-](0,0)--(0.1,0);}}(1,0);
	      \draw (0,-1) to node {\tikz{\draw[-](0,0)--(0.1,0);}} (1,-1);
	      \fill[black] (0,0) circle[radius=0.02];
	      \fill[black] (1,0) circle[radius=0.02];
	      \fill[black] (0,-1) circle[radius=0.02];
	      \fill[black] (1,-1) circle[radius=0.02];
	    \end{scope}
    \end{scope}
     \draw[->] (4,-2) to (8,-2);
      \node at (6,-2) [fill=white]{$\{h_3 h_2, h_1\}$};
       \draw[->] (1.5 , 0.1+1/2) to (4.4-1/2,-1+ 2*1.732);
      \node at (2.7,1.6) [fill=white]{$\{h_3, \act{\del h_2}h_1 \}$};
       \draw[->](3+4.4+1,-1+ 2*1.732) to (10.5, 0.6);
      \node at (9.5,1.6)  [fill=white]{$\act{h_3}\{h_2, h_1\}$};
  \end{tikzpicture}
\end{center}
\caption{Diagram  of Property \ref{lpl} of 2-crossed module}
\label{d:PF_lift_left}
\end{figure}

\begin{figure}
\begin{center}
  \begin{tikzpicture}[scale=1]
  	\begin{scope}[xshift=0.0cm]
	      \draw (0,0) to node[sloped] {\tikz{\draw[->](0,0)--(0.1,0);}} (0,-1);
	      \draw (1,0) to node[sloped] {\tikz{\draw[->](0,0)--(0.1,0);}} (1,-1);
	      \draw (0,0) to node[sloped] {\tikz{\draw[-](0,0)--(0.1,0);}} (1,0);
	      \draw (0,-1) to node  {\tikz{\draw[-](0,0)--(0.1,0);}} (1,-1);
	      \node at (1/2,-1/2) {$h_3$};
	      \draw (0,0) to node {\tikz{\draw[-](0,0)--(0.1,0);}}(1,0);
	      \draw (0,-1) to node {\tikz{\draw[-](0,0)--(0.1,0);}} (1,-1);
	      \fill[black] (0,0) circle[radius=0.02];
	      \fill[black] (1,0) circle[radius=0.02];
	      \fill[black] (0,-1) circle[radius=0.02];
	      \fill[black] (1,-1) circle[radius=0.02];
	    \begin{scope}[xshift=1.0cm]
	      \draw (0,0) to node[sloped] {\tikz{\draw[->](0,0)--(0.1,0);}} (0,-1);
	      \draw (1,0) to node[sloped] {\tikz{\draw[->](0,0)--(0.1,0);}} (1,-1);
	      \draw (0,0) to node[sloped] {\tikz{\draw[-](0,0)--(0.1,0);}} (1,0);
	      \draw (0,-1) to node  {\tikz{\draw[-](0,0)--(0.1,0);}} (1,-1);
	      \node at (1/2,-1/2) {$$};
	      \draw (0,0) to node {\tikz{\draw[-](0,0)--(0.1,0);}} (1,0);
	      \draw (0,-1) to node {\tikz{\draw[-](0,0)--(0.1,0);}} (1,-1);
	      \fill[black] (0,0) circle[radius=0.02];
	      \fill[black] (1,0) circle[radius=0.02];
	      \fill[black] (0,-1) circle[radius=0.02];
	      \fill[black] (1,-1) circle[radius=0.02];
	    \end{scope}
	    	 \begin{scope}[xshift=2.0cm]
	      \draw (0,0) to node[sloped] {\tikz{\draw[->](0,0)--(0.1,0);}} (0,-1);
	      \draw (1,0) to node[sloped] {\tikz{\draw[->](0,0)--(0.1,0);}} (1,-1);
	      \draw (0,0) to node[sloped] {\tikz{\draw[-](0,0)--(0.1,0);}} (1,0);
	      \draw (0,-1) to node  {\tikz{\draw[-](0,0)--(0.1,0);}} (1,-1);
	      \node at (1/2,-1/2) {$$};
	      \draw (0,0) to node {\tikz{\draw[-](0,0)--(0.1,0);}} (1,0);
	      \draw (0,-1) to node {\tikz{\draw[-](0,0)--(0.1,0);}} (1,-1);
	      \fill[black] (0,0) circle[radius=0.02];
	      \fill[black] (1,0) circle[radius=0.02];
	      \fill[black] (0,-1) circle[radius=0.02];
	      \fill[black] (1,-1) circle[radius=0.02];
	    \end{scope}
	        \begin{scope}[xshift=0.0cm, yshift = -1.0cm]
	      \draw (0,0) to node[sloped] {\tikz{\draw[->](0,0)--(0.1,0);}} (0,-1);
	      \draw (1,0) to node[sloped] {\tikz{\draw[->](0,0)--(0.1,0);}} (1,-1);
	      \draw (0,0) to node[sloped] {\tikz{\draw[-](0,0)--(0.1,0);}} (1,0);
	      \draw (0,-1) to node  {\tikz{\draw[-](0,0)--(0.1,0);}} (1,-1);
	      \node at (1/2,-1/2) {$$};
	      \draw (0,0) to node {\tikz{\draw[-](0,0)--(0.1,0);}} (1,0);
	      \draw (0,-1) to node {\tikz{\draw[-](0,0)--(0.1,0);}} (1,-1);
	      \fill[black] (0,0) circle[radius=0.02];
	      \fill[black] (1,0) circle[radius=0.02];
	      \fill[black] (0,-1) circle[radius=0.02];
	      \fill[black] (1,-1) circle[radius=0.02];
	    \end{scope}
	    \begin{scope}[xshift=1.0cm, yshift = -1.0cm]
	      \draw (0,0) to node[sloped] {\tikz{\draw[->](0,0)--(0.1,0);}} (0,-1);
	      \draw (1,0) to node[sloped] {\tikz{\draw[->](0,0)--(0.1,0);}} (1,-1);
	      \draw (0,0) to node[sloped] {\tikz{\draw[-](0,0)--(0.1,0);}} (1,0);
	      \draw (0,-1) to node  {\tikz{\draw[-](0,0)--(0.1,0);}} (1,-1);
	      \node at (1/2,-1/2) {$$};
	      \draw (0,0) to node {\tikz{\draw[-](0,0)--(0.1,0);}} (1,0);
	      \draw (0,-1) to node {\tikz{\draw[-](0,0)--(0.1,0);}} (1,-1);
	      \fill[black] (0,0) circle[radius=0.02];
	      \fill[black] (1,0) circle[radius=0.02];
	      \fill[black] (0,-1) circle[radius=0.02];
	      \fill[black] (1,-1) circle[radius=0.02];
	    \end{scope}
	     \begin{scope}[xshift=2.0cm, yshift = -1.0cm]
	      \draw (0,0) to node[sloped] {\tikz{\draw[->](0,0)--(0.1,0);}} (0,-1);
	      \draw (1,0) to node[sloped] {\tikz{\draw[->](0,0)--(0.1,0);}} (1,-1);
	      \draw (0,0) to node[sloped] {\tikz{\draw[-](0,0)--(0.1,0);}} (1,0);
	      \draw (0,-1) to node  {\tikz{\draw[-](0,0)--(0.1,0);}} (1,-1);
	      \node at (1/2,-1/2) {$h_1$};
	      \draw (0,0) to node {\tikz{\draw[-](0,0)--(0.1,0);}} (1,0);
	      \draw (0,-1) to node {\tikz{\draw[-](0,0)--(0.1,0);}} (1,-1);
	      \fill[black] (0,0) circle[radius=0.02];
	      \fill[black] (1,0) circle[radius=0.02];
	      \fill[black] (0,-1) circle[radius=0.02];
	      \fill[black] (1,-1) circle[radius=0.02];
	    \end{scope}
	     \begin{scope}[xshift=0.0cm, yshift = -2.0cm]
	      \draw (0,0) to node[sloped] {\tikz{\draw[->](0,0)--(0.1,0);}} (0,-1);
	      \draw (1,0) to node[sloped] {\tikz{\draw[->](0,0)--(0.1,0);}} (1,-1);
	      \draw (0,0) to node[sloped] {\tikz{\draw[-](0,0)--(0.1,0);}} (1,0);
	      \draw (0,-1) to node  {\tikz{\draw[-](0,0)--(0.1,0);}} (1,-1);
	      \node at (1/2,-1/2) {$$};
	      \draw (0,0) to node {\tikz{\draw[-](0,0)--(0.1,0);}} (1,0);
	      \draw (0,-1) to node {\tikz{\draw[-](0,0)--(0.1,0);}} (1,-1);
	      \fill[black] (0,0) circle[radius=0.02];
	      \fill[black] (1,0) circle[radius=0.02];
	      \fill[black] (0,-1) circle[radius=0.02];
	      \fill[black] (1,-1) circle[radius=0.02];
	    \end{scope}
	    \begin{scope}[xshift=1.0cm, yshift = -2.0cm]
	      \draw (0,0) to node[sloped] {\tikz{\draw[->](0,0)--(0.1,0);}} (0,-1);
	      \draw (1,0) to node[sloped] {\tikz{\draw[->](0,0)--(0.1,0);}} (1,-1);
	      \draw (0,0) to node[sloped] {\tikz{\draw[-](0,0)--(0.1,0);}} (1,0);
	      \draw (0,-1) to node  {\tikz{\draw[-](0,0)--(0.1,0);}} (1,-1);
	      \node at (1/2,-1/2) {$h_2$};
	      \draw (0,0) to node {\tikz{\draw[-](0,0)--(0.1,0);}}(1,0);
	      \draw (0,-1) to node {\tikz{\draw[-](0,0)--(0.1,0);}} (1,-1);
	      \fill[black] (0,0) circle[radius=0.02];
	      \fill[black] (1,0) circle[radius=0.02];
	      \fill[black] (0,-1) circle[radius=0.02];
	      \fill[black] (1,-1) circle[radius=0.02];
	    \end{scope}
	     \begin{scope}[xshift=2.0cm, yshift = -2.0cm]
	      \draw (0,0) to node[sloped] {\tikz{\draw[->](0,0)--(0.1,0);}} (0,-1);
	      \draw (1,0) to node[sloped] {\tikz{\draw[->](0,0)--(0.1,0);}} (1,-1);
	      \draw (0,0) to node[sloped] {\tikz{\draw[-](0,0)--(0.1,0);}} (1,0);
	      \draw (0,-1) to node  {\tikz{\draw[-](0,0)--(0.1,0);}} (1,-1);
	      \node at (1/2,-1/2) {$$};
	      \draw (0,0) to node {\tikz{\draw[-](0,0)--(0.1,0);}}(1,0);
	      \draw (0,-1) to node {\tikz{\draw[-](0,0)--(0.1,0);}} (1,-1);
	      \fill[black] (0,0) circle[radius=0.02];
	      \fill[black] (1,0) circle[radius=0.02];
	      \fill[black] (0,-1) circle[radius=0.02];
	      \fill[black] (1,-1) circle[radius=0.02];
	    \end{scope}
    \end{scope}
    \begin{scope}[xshift=4.5cm, yshift = 2*1.732cm]
	      \draw (0,0) to node[sloped] {\tikz{\draw[->](0,0)--(0.1,0);}} (0,-1);
	      \draw (1,0) to node[sloped] {\tikz{\draw[->](0,0)--(0.1,0);}} (1,-1);
	      \draw (0,0) to node[sloped] {\tikz{\draw[-](0,0)--(0.1,0);}} (1,0);
	      \draw (0,-1) to node  {\tikz{\draw[-](0,0)--(0.1,0);}} (1,-1);
	      \node at (1/2,-1/2) {$$};
	      \draw (0,0) to node {\tikz{\draw[-](0,0)--(0.1,0);}}(1,0);
	      \draw (0,-1) to node {\tikz{\draw[-](0,0)--(0.1,0);}} (1,-1);
	      \fill[black] (0,0) circle[radius=0.02];
	      \fill[black] (1,0) circle[radius=0.02];
	      \fill[black] (0,-1) circle[radius=0.02];
	      \fill[black] (1,-1) circle[radius=0.02];
	    \begin{scope}[xshift=1.0cm]
	      \draw (0,0) to node[sloped] {\tikz{\draw[->](0,0)--(0.1,0);}} (0,-1);
	      \draw (1,0) to node[sloped] {\tikz{\draw[->](0,0)--(0.1,0);}} (1,-1);
	      \draw (0,0) to node[sloped] {\tikz{\draw[-](0,0)--(0.1,0);}} (1,0);
	      \draw (0,-1) to node  {\tikz{\draw[-](0,0)--(0.1,0);}} (1,-1);
	      \node at (1/2,-1/2) {$$};
	      \draw (0,0) to node {\tikz{\draw[-](0,0)--(0.1,0);}} (1,0);
	      \draw (0,-1) to node {\tikz{\draw[-](0,0)--(0.1,0);}} (1,-1);
	      \fill[black] (0,0) circle[radius=0.02];
	      \fill[black] (1,0) circle[radius=0.02];
	      \fill[black] (0,-1) circle[radius=0.02];
	      \fill[black] (1,-1) circle[radius=0.02];
	    \end{scope}
	    	 \begin{scope}[xshift=2.0cm]
	      \draw (0,0) to node[sloped] {\tikz{\draw[->](0,0)--(0.1,0);}} (0,-1);
	      \draw (1,0) to node[sloped] {\tikz{\draw[->](0,0)--(0.1,0);}} (1,-1);
	      \draw (0,0) to node[sloped] {\tikz{\draw[-](0,0)--(0.1,0);}} (1,0);
	      \draw (0,-1) to node  {\tikz{\draw[-](0,0)--(0.1,0);}} (1,-1);
	      \node at (1/2,-1/2) {$h_1$};
	      \draw (0,0) to node {\tikz{\draw[-](0,0)--(0.1,0);}} (1,0);
	      \draw (0,-1) to node {\tikz{\draw[-](0,0)--(0.1,0);}} (1,-1);
	      \fill[black] (0,0) circle[radius=0.02];
	      \fill[black] (1,0) circle[radius=0.02];
	      \fill[black] (0,-1) circle[radius=0.02];
	      \fill[black] (1,-1) circle[radius=0.02];
	    \end{scope}
	        \begin{scope}[xshift=0.0cm, yshift = -1.0cm]
	      \draw (0,0) to node[sloped] {\tikz{\draw[->](0,0)--(0.1,0);}} (0,-1);
	      \draw (1,0) to node[sloped] {\tikz{\draw[->](0,0)--(0.1,0);}} (1,-1);
	      \draw (0,0) to node[sloped] {\tikz{\draw[-](0,0)--(0.1,0);}} (1,0);
	      \draw (0,-1) to node  {\tikz{\draw[-](0,0)--(0.1,0);}} (1,-1);
	      \node at (1/2,-1/2) {$h_3$};
	      \draw (0,0) to node {\tikz{\draw[-](0,0)--(0.1,0);}} (1,0);
	      \draw (0,-1) to node {\tikz{\draw[-](0,0)--(0.1,0);}} (1,-1);
	      \fill[black] (0,0) circle[radius=0.02];
	      \fill[black] (1,0) circle[radius=0.02];
	      \fill[black] (0,-1) circle[radius=0.02];
	      \fill[black] (1,-1) circle[radius=0.02];
	    \end{scope}
	    \begin{scope}[xshift=1.0cm, yshift = -1.0cm]
	      \draw (0,0) to node[sloped] {\tikz{\draw[->](0,0)--(0.1,0);}} (0,-1);
	      \draw (1,0) to node[sloped] {\tikz{\draw[->](0,0)--(0.1,0);}} (1,-1);
	      \draw (0,0) to node[sloped] {\tikz{\draw[-](0,0)--(0.1,0);}} (1,0);
	      \draw (0,-1) to node  {\tikz{\draw[-](0,0)--(0.1,0);}} (1,-1);
	      \node at (1/2,-1/2) {$$};
	      \draw (0,0) to node {\tikz{\draw[-](0,0)--(0.1,0);}} (1,0);
	      \draw (0,-1) to node {\tikz{\draw[-](0,0)--(0.1,0);}} (1,-1);
	      \fill[black] (0,0) circle[radius=0.02];
	      \fill[black] (1,0) circle[radius=0.02];
	      \fill[black] (0,-1) circle[radius=0.02];
	      \fill[black] (1,-1) circle[radius=0.02];
	    \end{scope}
	     \begin{scope}[xshift=2.0cm, yshift = -1.0cm]
	      \draw (0,0) to node[sloped] {\tikz{\draw[->](0,0)--(0.1,0);}} (0,-1);
	      \draw (1,0) to node[sloped] {\tikz{\draw[->](0,0)--(0.1,0);}} (1,-1);
	      \draw (0,0) to node[sloped] {\tikz{\draw[-](0,0)--(0.1,0);}} (1,0);
	      \draw (0,-1) to node  {\tikz{\draw[-](0,0)--(0.1,0);}} (1,-1);
	      \node at (1/2,-1/2) {$$};
	      \draw (0,0) to node {\tikz{\draw[-](0,0)--(0.1,0);}} (1,0);
	      \draw (0,-1) to node {\tikz{\draw[-](0,0)--(0.1,0);}} (1,-1);
	      \fill[black] (0,0) circle[radius=0.02];
	      \fill[black] (1,0) circle[radius=0.02];
	      \fill[black] (0,-1) circle[radius=0.02];
	      \fill[black] (1,-1) circle[radius=0.02];
	    \end{scope}
	     \begin{scope}[xshift=0.0cm, yshift = -2.0cm]
	      \draw (0,0) to node[sloped] {\tikz{\draw[->](0,0)--(0.1,0);}} (0,-1);
	      \draw (1,0) to node[sloped] {\tikz{\draw[->](0,0)--(0.1,0);}} (1,-1);
	      \draw (0,0) to node[sloped] {\tikz{\draw[-](0,0)--(0.1,0);}} (1,0);
	      \draw (0,-1) to node  {\tikz{\draw[-](0,0)--(0.1,0);}} (1,-1);
	      \node at (1/2,-1/2) {$$};
	      \draw (0,0) to node {\tikz{\draw[-](0,0)--(0.1,0);}} (1,0);
	      \draw (0,-1) to node {\tikz{\draw[-](0,0)--(0.1,0);}} (1,-1);
	      \fill[black] (0,0) circle[radius=0.02];
	      \fill[black] (1,0) circle[radius=0.02];
	      \fill[black] (0,-1) circle[radius=0.02];
	      \fill[black] (1,-1) circle[radius=0.02];
	    \end{scope}
	    \begin{scope}[xshift=1.0cm, yshift = -2.0cm]
	      \draw (0,0) to node[sloped] {\tikz{\draw[->](0,0)--(0.1,0);}} (0,-1);
	      \draw (1,0) to node[sloped] {\tikz{\draw[->](0,0)--(0.1,0);}} (1,-1);
	      \draw (0,0) to node[sloped] {\tikz{\draw[-](0,0)--(0.1,0);}} (1,0);
	      \draw (0,-1) to node  {\tikz{\draw[-](0,0)--(0.1,0);}} (1,-1);
	      \node at (1/2,-1/2) {$h_2$};
	      \draw (0,0) to node {\tikz{\draw[-](0,0)--(0.1,0);}}(1,0);
	      \draw (0,-1) to node {\tikz{\draw[-](0,0)--(0.1,0);}} (1,-1);
	      \fill[black] (0,0) circle[radius=0.02];
	      \fill[black] (1,0) circle[radius=0.02];
	      \fill[black] (0,-1) circle[radius=0.02];
	      \fill[black] (1,-1) circle[radius=0.02];
	    \end{scope}
	     \begin{scope}[xshift=2.0cm, yshift = -2.0cm]
	      \draw (0,0) to node[sloped] {\tikz{\draw[->](0,0)--(0.1,0);}} (0,-1);
	      \draw (1,0) to node[sloped] {\tikz{\draw[->](0,0)--(0.1,0);}} (1,-1);
	      \draw (0,0) to node[sloped] {\tikz{\draw[-](0,0)--(0.1,0);}} (1,0);
	      \draw (0,-1) to node  {\tikz{\draw[-](0,0)--(0.1,0);}} (1,-1);
	      \node at (1/2,-1/2) {$$};
	      \draw (0,0) to node {\tikz{\draw[-](0,0)--(0.1,0);}}(1,0);
	      \draw (0,-1) to node {\tikz{\draw[-](0,0)--(0.1,0);}} (1,-1);
	      \fill[black] (0,0) circle[radius=0.02];
	      \fill[black] (1,0) circle[radius=0.02];
	      \fill[black] (0,-1) circle[radius=0.02];
	      \fill[black] (1,-1) circle[radius=0.02];
	    \end{scope}
    \end{scope}
    \begin{scope}[xshift=9.0cm]
	      \draw (0,0) to node[sloped] {\tikz{\draw[->](0,0)--(0.1,0);}} (0,-1);
	      \draw (1,0) to node[sloped] {\tikz{\draw[->](0,0)--(0.1,0);}} (1,-1);
	      \draw (0,0) to node[sloped] {\tikz{\draw[-](0,0)--(0.1,0);}} (1,0);
	      \draw (0,-1) to node  {\tikz{\draw[-](0,0)--(0.1,0);}} (1,-1);
	      \node at (1/2,-1/2) {$$};
	      \draw (0,0) to node {\tikz{\draw[-](0,0)--(0.1,0);}}(1,0);
	      \draw (0,-1) to node {\tikz{\draw[-](0,0)--(0.1,0);}} (1,-1);
	      \fill[black] (0,0) circle[radius=0.02];
	      \fill[black] (1,0) circle[radius=0.02];
	      \fill[black] (0,-1) circle[radius=0.02];
	      \fill[black] (1,-1) circle[radius=0.02];
	    \begin{scope}[xshift=1.0cm]
	      \draw (0,0) to node[sloped] {\tikz{\draw[->](0,0)--(0.1,0);}} (0,-1);
	      \draw (1,0) to node[sloped] {\tikz{\draw[->](0,0)--(0.1,0);}} (1,-1);
	      \draw (0,0) to node[sloped] {\tikz{\draw[-](0,0)--(0.1,0);}} (1,0);
	      \draw (0,-1) to node  {\tikz{\draw[-](0,0)--(0.1,0);}} (1,-1);
	      \node at (1/2,-1/2) {$$};
	      \draw (0,0) to node {\tikz{\draw[-](0,0)--(0.1,0);}} (1,0);
	      \draw (0,-1) to node {\tikz{\draw[-](0,0)--(0.1,0);}} (1,-1);
	      \fill[black] (0,0) circle[radius=0.02];
	      \fill[black] (1,0) circle[radius=0.02];
	      \fill[black] (0,-1) circle[radius=0.02];
	      \fill[black] (1,-1) circle[radius=0.02];
	    \end{scope}
	    	 \begin{scope}[xshift=2.0cm]
	      \draw (0,0) to node[sloped] {\tikz{\draw[->](0,0)--(0.1,0);}} (0,-1);
	      \draw (1,0) to node[sloped] {\tikz{\draw[->](0,0)--(0.1,0);}} (1,-1);
	      \draw (0,0) to node[sloped] {\tikz{\draw[-](0,0)--(0.1,0);}} (1,0);
	      \draw (0,-1) to node  {\tikz{\draw[-](0,0)--(0.1,0);}} (1,-1);
	      \node at (1/2,-1/2) {$h_1$};
	      \draw (0,0) to node {\tikz{\draw[-](0,0)--(0.1,0);}} (1,0);
	      \draw (0,-1) to node {\tikz{\draw[-](0,0)--(0.1,0);}} (1,-1);
	      \fill[black] (0,0) circle[radius=0.02];
	      \fill[black] (1,0) circle[radius=0.02];
	      \fill[black] (0,-1) circle[radius=0.02];
	      \fill[black] (1,-1) circle[radius=0.02];
	    \end{scope}
	        \begin{scope}[xshift=0.0cm, yshift = -1.0cm]
	      \draw (0,0) to node[sloped] {\tikz{\draw[->](0,0)--(0.1,0);}} (0,-1);
	      \draw (1,0) to node[sloped] {\tikz{\draw[->](0,0)--(0.1,0);}} (1,-1);
	      \draw (0,0) to node[sloped] {\tikz{\draw[-](0,0)--(0.1,0);}} (1,0);
	      \draw (0,-1) to node  {\tikz{\draw[-](0,0)--(0.1,0);}} (1,-1);
	      \node at (1/2,-1/2) {$$};
	      \draw (0,0) to node {\tikz{\draw[-](0,0)--(0.1,0);}} (1,0);
	      \draw (0,-1) to node {\tikz{\draw[-](0,0)--(0.1,0);}} (1,-1);
	      \fill[black] (0,0) circle[radius=0.02];
	      \fill[black] (1,0) circle[radius=0.02];
	      \fill[black] (0,-1) circle[radius=0.02];
	      \fill[black] (1,-1) circle[radius=0.02];
	    \end{scope}
	    \begin{scope}[xshift=1.0cm, yshift = -1.0cm]
	      \draw (0,0) to node[sloped] {\tikz{\draw[->](0,0)--(0.1,0);}} (0,-1);
	      \draw (1,0) to node[sloped] {\tikz{\draw[->](0,0)--(0.1,0);}} (1,-1);
	      \draw (0,0) to node[sloped] {\tikz{\draw[-](0,0)--(0.1,0);}} (1,0);
	      \draw (0,-1) to node  {\tikz{\draw[-](0,0)--(0.1,0);}} (1,-1);
	      \node at (1/2,-1/2) {$h_2$};
	      \draw (0,0) to node {\tikz{\draw[-](0,0)--(0.1,0);}} (1,0);
	      \draw (0,-1) to node {\tikz{\draw[-](0,0)--(0.1,0);}} (1,-1);
	      \fill[black] (0,0) circle[radius=0.02];
	      \fill[black] (1,0) circle[radius=0.02];
	      \fill[black] (0,-1) circle[radius=0.02];
	      \fill[black] (1,-1) circle[radius=0.02];
	    \end{scope}
	     \begin{scope}[xshift=2.0cm, yshift = -1.0cm]
	      \draw (0,0) to node[sloped] {\tikz{\draw[->](0,0)--(0.1,0);}} (0,-1);
	      \draw (1,0) to node[sloped] {\tikz{\draw[->](0,0)--(0.1,0);}} (1,-1);
	      \draw (0,0) to node[sloped] {\tikz{\draw[-](0,0)--(0.1,0);}} (1,0);
	      \draw (0,-1) to node  {\tikz{\draw[-](0,0)--(0.1,0);}} (1,-1);
	      \node at (1/2,-1/2) {$$};
	      \draw (0,0) to node {\tikz{\draw[-](0,0)--(0.1,0);}} (1,0);
	      \draw (0,-1) to node {\tikz{\draw[-](0,0)--(0.1,0);}} (1,-1);
	      \fill[black] (0,0) circle[radius=0.02];
	      \fill[black] (1,0) circle[radius=0.02];
	      \fill[black] (0,-1) circle[radius=0.02];
	      \fill[black] (1,-1) circle[radius=0.02];
	    \end{scope}
	     \begin{scope}[xshift=0.0cm, yshift = -2.0cm]
	      \draw (0,0) to node[sloped] {\tikz{\draw[->](0,0)--(0.1,0);}} (0,-1);
	      \draw (1,0) to node[sloped] {\tikz{\draw[->](0,0)--(0.1,0);}} (1,-1);
	      \draw (0,0) to node[sloped] {\tikz{\draw[-](0,0)--(0.1,0);}} (1,0);
	      \draw (0,-1) to node  {\tikz{\draw[-](0,0)--(0.1,0);}} (1,-1);
	      \node at (1/2,-1/2) {$h_3$};
	      \draw (0,0) to node {\tikz{\draw[-](0,0)--(0.1,0);}} (1,0);
	      \draw (0,-1) to node {\tikz{\draw[-](0,0)--(0.1,0);}} (1,-1);
	      \fill[black] (0,0) circle[radius=0.02];
	      \fill[black] (1,0) circle[radius=0.02];
	      \fill[black] (0,-1) circle[radius=0.02];
	      \fill[black] (1,-1) circle[radius=0.02];
	    \end{scope}
	    \begin{scope}[xshift=1.0cm, yshift = -2.0cm]
	      \draw (0,0) to node[sloped] {\tikz{\draw[->](0,0)--(0.1,0);}} (0,-1);
	      \draw (1,0) to node[sloped] {\tikz{\draw[->](0,0)--(0.1,0);}} (1,-1);
	      \draw (0,0) to node[sloped] {\tikz{\draw[-](0,0)--(0.1,0);}} (1,0);
	      \draw (0,-1) to node  {\tikz{\draw[-](0,0)--(0.1,0);}} (1,-1);
	      \node at (1/2,-1/2) {$$};
	      \draw (0,0) to node {\tikz{\draw[-](0,0)--(0.1,0);}}(1,0);
	      \draw (0,-1) to node {\tikz{\draw[-](0,0)--(0.1,0);}} (1,-1);
	      \fill[black] (0,0) circle[radius=0.02];
	      \fill[black] (1,0) circle[radius=0.02];
	      \fill[black] (0,-1) circle[radius=0.02];
	      \fill[black] (1,-1) circle[radius=0.02];
	    \end{scope}
	     \begin{scope}[xshift=2.0cm, yshift = -2.0cm]
	      \draw (0,0) to node[sloped] {\tikz{\draw[->](0,0)--(0.1,0);}} (0,-1);
	      \draw (1,0) to node[sloped] {\tikz{\draw[->](0,0)--(0.1,0);}} (1,-1);
	      \draw (0,0) to node[sloped] {\tikz{\draw[-](0,0)--(0.1,0);}} (1,0);
	      \draw (0,-1) to node  {\tikz{\draw[-](0,0)--(0.1,0);}} (1,-1);
	      \node at (1/2,-1/2) {$$};
	      \draw (0,0) to node {\tikz{\draw[-](0,0)--(0.1,0);}}(1,0);
	      \draw (0,-1) to node {\tikz{\draw[-](0,0)--(0.1,0);}} (1,-1);
	      \fill[black] (0,0) circle[radius=0.02];
	      \fill[black] (1,0) circle[radius=0.02];
	      \fill[black] (0,-1) circle[radius=0.02];
	      \fill[black] (1,-1) circle[radius=0.02];
	    \end{scope}
    \end{scope}
     \draw[->] (4,-2) to (8,-2);
      \node at (6,-2) [fill=white]{$\{h_3 ,h_2 h_1\}$};
       \draw[->] (1.5 , 0.1+1/2) to (4.4-1/2,-1+ 2*1.732);
      \node at (2.7,1.6) [fill=white]{$\act{(\act{\del h_3}h_2)}\{h_3, h_1 \}$};
       \draw[->](3+4.4+1,-1+ 2*1.732) to (10.5, 0.6);
      \node at (9.5,1.6)  [fill=white]{$\{h_3, h_2\}$};
  \end{tikzpicture}
\end{center}
\caption{Figure of property \ref{rpl} of 2-crossed module}
\label{d:PF_lift_right}
\end{figure}
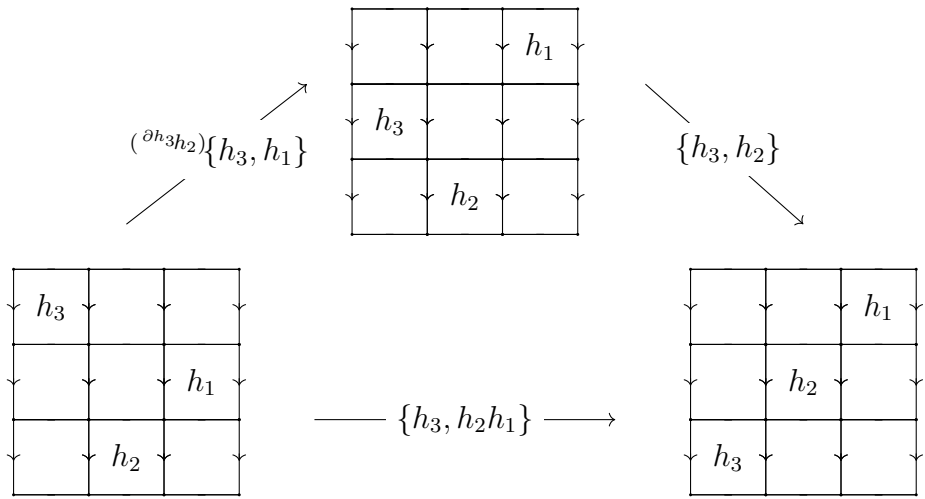

Property~\ref{pfeq} of 2-crossed module, i.e. $\act{g}\{h_2, h_1\} = \{\act{g}h_2, \act{g} h_1\}$ for $g\in G$ and $h_1,h_2\in H$ guarantees that the diagram is well-defined. Let $h_1,h_2,h_3\in H$. We now consider Figure~\ref{d:PF_lift_wd}.
\begin{figure}
\begin{center}
  \begin{tikzpicture}[scale=1]
  	\begin{scope}[xshift=0.0cm]
	      \draw (0,0) to node[sloped] {\tikz{\draw[->](0,0)--(0.1,0);}} (0,-1);
	      \draw (1,0) to node[sloped] {\tikz{\draw[->](0,0)--(0.1,0);}} (1,-1);
	      \draw (0,0) to node[sloped] {\tikz{\draw[-](0,0)--(0.1,0);}} (1,0);
	      \draw (0,-1) to node  {\tikz{\draw[-](0,0)--(0.1,0);}} (1,-1);
	      \node at (1/2,-1/2) {$h_3$};
	      \draw (0,0) to node {\tikz{\draw[-](0,0)--(0.1,0);}}(1,0);
	      \draw (0,-1) to node {\tikz{\draw[-](0,0)--(0.1,0);}} (1,-1);
	      \fill[black] (0,0) circle[radius=0.02];
	      \fill[black] (1,0) circle[radius=0.02];
	      \fill[black] (0,-1) circle[radius=0.02];
	      \fill[black] (1,-1) circle[radius=0.02];
	    \begin{scope}[xshift=1.0cm]
	      \draw (0,0) to node[sloped] {\tikz{\draw[->](0,0)--(0.1,0);}} (0,-1);
	      \draw (1,0) to node[sloped] {\tikz{\draw[->](0,0)--(0.1,0);}} (1,-1);
	      \draw (0,0) to node[sloped] {\tikz{\draw[-](0,0)--(0.1,0);}} (1,0);
	      \draw (0,-1) to node  {\tikz{\draw[-](0,0)--(0.1,0);}} (1,-1);
	      \node at (1/2,-1/2) {$$};
	      \draw (0,0) to node {\tikz{\draw[-](0,0)--(0.1,0);}} (1,0);
	      \draw (0,-1) to node {\tikz{\draw[-](0,0)--(0.1,0);}} (1,-1);
	      \fill[black] (0,0) circle[radius=0.02];
	      \fill[black] (1,0) circle[radius=0.02];
	      \fill[black] (0,-1) circle[radius=0.02];
	      \fill[black] (1,-1) circle[radius=0.02];
	    \end{scope}
	    	 \begin{scope}[xshift=2.0cm]
	      \draw (0,0) to node[sloped] {\tikz{\draw[->](0,0)--(0.1,0);}} (0,-1);
	      \draw (1,0) to node[sloped] {\tikz{\draw[->](0,0)--(0.1,0);}} (1,-1);
	      \draw (0,0) to node[sloped] {\tikz{\draw[-](0,0)--(0.1,0);}} (1,0);
	      \draw (0,-1) to node  {\tikz{\draw[-](0,0)--(0.1,0);}} (1,-1);
	      \node at (1/2,-1/2) {$$};
	      \draw (0,0) to node {\tikz{\draw[-](0,0)--(0.1,0);}} (1,0);
	      \draw (0,-1) to node {\tikz{\draw[-](0,0)--(0.1,0);}} (1,-1);
	      \fill[black] (0,0) circle[radius=0.02];
	      \fill[black] (1,0) circle[radius=0.02];
	      \fill[black] (0,-1) circle[radius=0.02];
	      \fill[black] (1,-1) circle[radius=0.02];
	    \end{scope}
	        \begin{scope}[xshift=0.0cm, yshift = -1.0cm]
	      \draw (0,0) to node[sloped] {\tikz{\draw[->](0,0)--(0.1,0);}} (0,-1);
	      \draw (1,0) to node[sloped] {\tikz{\draw[->](0,0)--(0.1,0);}} (1,-1);
	      \draw (0,0) to node[sloped] {\tikz{\draw[-](0,0)--(0.1,0);}} (1,0);
	      \draw (0,-1) to node  {\tikz{\draw[-](0,0)--(0.1,0);}} (1,-1);
	      \node at (1/2,-1/2) {$$};
	      \draw (0,0) to node {\tikz{\draw[-](0,0)--(0.1,0);}} (1,0);
	      \draw (0,-1) to node {\tikz{\draw[-](0,0)--(0.1,0);}} (1,-1);
	      \fill[black] (0,0) circle[radius=0.02];
	      \fill[black] (1,0) circle[radius=0.02];
	      \fill[black] (0,-1) circle[radius=0.02];
	      \fill[black] (1,-1) circle[radius=0.02];
	    \end{scope}
	    \begin{scope}[xshift=1.0cm, yshift = -1.0cm]
	      \draw (0,0) to node[sloped] {\tikz{\draw[->](0,0)--(0.1,0);}} (0,-1);
	      \draw (1,0) to node[sloped] {\tikz{\draw[->](0,0)--(0.1,0);}} (1,-1);
	      \draw (0,0) to node[sloped] {\tikz{\draw[-](0,0)--(0.1,0);}} (1,0);
	      \draw (0,-1) to node  {\tikz{\draw[-](0,0)--(0.1,0);}} (1,-1);
	      \node at (1/2,-1/2) {$h_2$};
	      \draw (0,0) to node {\tikz{\draw[-](0,0)--(0.1,0);}} (1,0);
	      \draw (0,-1) to node {\tikz{\draw[-](0,0)--(0.1,0);}} (1,-1);
	      \fill[black] (0,0) circle[radius=0.02];
	      \fill[black] (1,0) circle[radius=0.02];
	      \fill[black] (0,-1) circle[radius=0.02];
	      \fill[black] (1,-1) circle[radius=0.02];
	    \end{scope}
	     \begin{scope}[xshift=2.0cm, yshift = -1.0cm]
	      \draw (0,0) to node[sloped] {\tikz{\draw[->](0,0)--(0.1,0);}} (0,-1);
	      \draw (1,0) to node[sloped] {\tikz{\draw[->](0,0)--(0.1,0);}} (1,-1);
	      \draw (0,0) to node[sloped] {\tikz{\draw[-](0,0)--(0.1,0);}} (1,0);
	      \draw (0,-1) to node  {\tikz{\draw[-](0,0)--(0.1,0);}} (1,-1);
	      \node at (1/2,-1/2) {$$};
	      \draw (0,0) to node {\tikz{\draw[-](0,0)--(0.1,0);}} (1,0);
	      \draw (0,-1) to node {\tikz{\draw[-](0,0)--(0.1,0);}} (1,-1);
	      \fill[black] (0,0) circle[radius=0.02];
	      \fill[black] (1,0) circle[radius=0.02];
	      \fill[black] (0,-1) circle[radius=0.02];
	      \fill[black] (1,-1) circle[radius=0.02];
	    \end{scope}
	     \begin{scope}[xshift=0.0cm, yshift = -2.0cm]
	      \draw (0,0) to node[sloped] {\tikz{\draw[->](0,0)--(0.1,0);}} (0,-1);
	      \draw (1,0) to node[sloped] {\tikz{\draw[->](0,0)--(0.1,0);}} (1,-1);
	      \draw (0,0) to node[sloped] {\tikz{\draw[-](0,0)--(0.1,0);}} (1,0);
	      \draw (0,-1) to node  {\tikz{\draw[-](0,0)--(0.1,0);}} (1,-1);
	      \node at (1/2,-1/2) {$$};
	      \draw (0,0) to node {\tikz{\draw[-](0,0)--(0.1,0);}} (1,0);
	      \draw (0,-1) to node {\tikz{\draw[-](0,0)--(0.1,0);}} (1,-1);
	      \fill[black] (0,0) circle[radius=0.02];
	      \fill[black] (1,0) circle[radius=0.02];
	      \fill[black] (0,-1) circle[radius=0.02];
	      \fill[black] (1,-1) circle[radius=0.02];
	    \end{scope}
	    \begin{scope}[xshift=1.0cm, yshift = -2.0cm]
	      \draw (0,0) to node[sloped] {\tikz{\draw[->](0,0)--(0.1,0);}} (0,-1);
	      \draw (1,0) to node[sloped] {\tikz{\draw[->](0,0)--(0.1,0);}} (1,-1);
	      \draw (0,0) to node[sloped] {\tikz{\draw[-](0,0)--(0.1,0);}} (1,0);
	      \draw (0,-1) to node  {\tikz{\draw[-](0,0)--(0.1,0);}} (1,-1);
	      \node at (1/2,-1/2) {$$};
	      \draw (0,0) to node {\tikz{\draw[-](0,0)--(0.1,0);}}(1,0);
	      \draw (0,-1) to node {\tikz{\draw[-](0,0)--(0.1,0);}} (1,-1);
	      \fill[black] (0,0) circle[radius=0.02];
	      \fill[black] (1,0) circle[radius=0.02];
	      \fill[black] (0,-1) circle[radius=0.02];
	      \fill[black] (1,-1) circle[radius=0.02];
	    \end{scope}
	     \begin{scope}[xshift=2.0cm, yshift = -2.0cm]
	      \draw (0,0) to node[sloped] {\tikz{\draw[->](0,0)--(0.1,0);}} (0,-1);
	      \draw (1,0) to node[sloped] {\tikz{\draw[->](0,0)--(0.1,0);}} (1,-1);
	      \draw (0,0) to node[sloped] {\tikz{\draw[-](0,0)--(0.1,0);}} (1,0);
	      \draw (0,-1) to node  {\tikz{\draw[-](0,0)--(0.1,0);}} (1,-1);
	      \node at (1/2,-1/2) {$h_1$};
	      \draw (0,0) to node {\tikz{\draw[-](0,0)--(0.1,0);}}(1,0);
	      \draw (0,-1) to node {\tikz{\draw[-](0,0)--(0.1,0);}} (1,-1);
	      \fill[black] (0,0) circle[radius=0.02];
	      \fill[black] (1,0) circle[radius=0.02];
	      \fill[black] (0,-1) circle[radius=0.02];
	      \fill[black] (1,-1) circle[radius=0.02];
	    \end{scope}
    \end{scope}
    \begin{scope}[xshift=6cm]
	      \draw (0,0) to node[sloped] {\tikz{\draw[->](0,0)--(0.1,0);}} (0,-1);
	      \draw (1,0) to node[sloped] {\tikz{\draw[->](0,0)--(0.1,0);}} (1,-1);
	      \draw (0,0) to node[sloped] {\tikz{\draw[-](0,0)--(0.1,0);}} (1,0);
	      \draw (0,-1) to node  {\tikz{\draw[-](0,0)--(0.1,0);}} (1,-1);
	      \node at (1/2,-1/2) {$h_3$};
	      \draw (0,0) to node {\tikz{\draw[-](0,0)--(0.1,0);}}(1,0);
	      \draw (0,-1) to node {\tikz{\draw[-](0,0)--(0.1,0);}} (1,-1);
	      \fill[black] (0,0) circle[radius=0.02];
	      \fill[black] (1,0) circle[radius=0.02];
	      \fill[black] (0,-1) circle[radius=0.02];
	      \fill[black] (1,-1) circle[radius=0.02];
	    \begin{scope}[xshift=1.0cm]
	      \draw (0,0) to node[sloped] {\tikz{\draw[->](0,0)--(0.1,0);}} (0,-1);
	      \draw (1,0) to node[sloped] {\tikz{\draw[->](0,0)--(0.1,0);}} (1,-1);
	      \draw (0,0) to node[sloped] {\tikz{\draw[-](0,0)--(0.1,0);}} (1,0);
	      \draw (0,-1) to node  {\tikz{\draw[-](0,0)--(0.1,0);}} (1,-1);
	      \node at (1/2,-1/2) {$$};
	      \draw (0,0) to node {\tikz{\draw[-](0,0)--(0.1,0);}} (1,0);
	      \draw (0,-1) to node {\tikz{\draw[-](0,0)--(0.1,0);}} (1,-1);
	      \fill[black] (0,0) circle[radius=0.02];
	      \fill[black] (1,0) circle[radius=0.02];
	      \fill[black] (0,-1) circle[radius=0.02];
	      \fill[black] (1,-1) circle[radius=0.02];
	    \end{scope}
	    	 \begin{scope}[xshift=2.0cm]
	      \draw (0,0) to node[sloped] {\tikz{\draw[->](0,0)--(0.1,0);}} (0,-1);
	      \draw (1,0) to node[sloped] {\tikz{\draw[->](0,0)--(0.1,0);}} (1,-1);
	      \draw (0,0) to node[sloped] {\tikz{\draw[-](0,0)--(0.1,0);}} (1,0);
	      \draw (0,-1) to node  {\tikz{\draw[-](0,0)--(0.1,0);}} (1,-1);
	      \node at (1/2,-1/2) {$$};
	      \draw (0,0) to node {\tikz{\draw[-](0,0)--(0.1,0);}} (1,0);
	      \draw (0,-1) to node {\tikz{\draw[-](0,0)--(0.1,0);}} (1,-1);
	      \fill[black] (0,0) circle[radius=0.02];
	      \fill[black] (1,0) circle[radius=0.02];
	      \fill[black] (0,-1) circle[radius=0.02];
	      \fill[black] (1,-1) circle[radius=0.02];
	    \end{scope}
	        \begin{scope}[xshift=0.0cm, yshift = -1.0cm]
	      \draw (0,0) to node[sloped] {\tikz{\draw[->](0,0)--(0.1,0);}} (0,-1);
	      \draw (1,0) to node[sloped] {\tikz{\draw[->](0,0)--(0.1,0);}} (1,-1);
	      \draw (0,0) to node[sloped] {\tikz{\draw[-](0,0)--(0.1,0);}} (1,0);
	      \draw (0,-1) to node  {\tikz{\draw[-](0,0)--(0.1,0);}} (1,-1);
	      \node at (1/2,-1/2) {$$};
	      \draw (0,0) to node {\tikz{\draw[-](0,0)--(0.1,0);}} (1,0);
	      \draw (0,-1) to node {\tikz{\draw[-](0,0)--(0.1,0);}} (1,-1);
	      \fill[black] (0,0) circle[radius=0.02];
	      \fill[black] (1,0) circle[radius=0.02];
	      \fill[black] (0,-1) circle[radius=0.02];
	      \fill[black] (1,-1) circle[radius=0.02];
	    \end{scope}
	    \begin{scope}[xshift=1.0cm, yshift = -1.0cm]
	      \draw (0,0) to node[sloped] {\tikz{\draw[->](0,0)--(0.1,0);}} (0,-1);
	      \draw (1,0) to node[sloped] {\tikz{\draw[->](0,0)--(0.1,0);}} (1,-1);
	      \draw (0,0) to node[sloped] {\tikz{\draw[-](0,0)--(0.1,0);}} (1,0);
	      \draw (0,-1) to node  {\tikz{\draw[-](0,0)--(0.1,0);}} (1,-1);
	      \node at (1/2,-1/2) {$$};
	      \draw (0,0) to node {\tikz{\draw[-](0,0)--(0.1,0);}} (1,0);
	      \draw (0,-1) to node {\tikz{\draw[-](0,0)--(0.1,0);}} (1,-1);
	      \fill[black] (0,0) circle[radius=0.02];
	      \fill[black] (1,0) circle[radius=0.02];
	      \fill[black] (0,-1) circle[radius=0.02];
	      \fill[black] (1,-1) circle[radius=0.02];
	    \end{scope}
	     \begin{scope}[xshift=2.0cm, yshift = -1.0cm]
	      \draw (0,0) to node[sloped] {\tikz{\draw[->](0,0)--(0.1,0);}} (0,-1);
	      \draw (1,0) to node[sloped] {\tikz{\draw[->](0,0)--(0.1,0);}} (1,-1);
	      \draw (0,0) to node[sloped] {\tikz{\draw[-](0,0)--(0.1,0);}} (1,0);
	      \draw (0,-1) to node  {\tikz{\draw[-](0,0)--(0.1,0);}} (1,-1);
	      \node at (1/2,-1/2) {$h_1$};
	      \draw (0,0) to node {\tikz{\draw[-](0,0)--(0.1,0);}} (1,0);
	      \draw (0,-1) to node {\tikz{\draw[-](0,0)--(0.1,0);}} (1,-1);
	      \fill[black] (0,0) circle[radius=0.02];
	      \fill[black] (1,0) circle[radius=0.02];
	      \fill[black] (0,-1) circle[radius=0.02];
	      \fill[black] (1,-1) circle[radius=0.02];
	    \end{scope}
	     \begin{scope}[xshift=0.0cm, yshift = -2.0cm]
	      \draw (0,0) to node[sloped] {\tikz{\draw[->](0,0)--(0.1,0);}} (0,-1);
	      \draw (1,0) to node[sloped] {\tikz{\draw[->](0,0)--(0.1,0);}} (1,-1);
	      \draw (0,0) to node[sloped] {\tikz{\draw[-](0,0)--(0.1,0);}} (1,0);
	      \draw (0,-1) to node  {\tikz{\draw[-](0,0)--(0.1,0);}} (1,-1);
	      \node at (1/2,-1/2) {$$};
	      \draw (0,0) to node {\tikz{\draw[-](0,0)--(0.1,0);}} (1,0);
	      \draw (0,-1) to node {\tikz{\draw[-](0,0)--(0.1,0);}} (1,-1);
	      \fill[black] (0,0) circle[radius=0.02];
	      \fill[black] (1,0) circle[radius=0.02];
	      \fill[black] (0,-1) circle[radius=0.02];
	      \fill[black] (1,-1) circle[radius=0.02];
	    \end{scope}
	    \begin{scope}[xshift=1.0cm, yshift = -2.0cm]
	      \draw (0,0) to node[sloped] {\tikz{\draw[->](0,0)--(0.1,0);}} (0,-1);
	      \draw (1,0) to node[sloped] {\tikz{\draw[->](0,0)--(0.1,0);}} (1,-1);
	      \draw (0,0) to node[sloped] {\tikz{\draw[-](0,0)--(0.1,0);}} (1,0);
	      \draw (0,-1) to node  {\tikz{\draw[-](0,0)--(0.1,0);}} (1,-1);
	      \node at (1/2,-1/2) {$h_2$};
	      \draw (0,0) to node {\tikz{\draw[-](0,0)--(0.1,0);}}(1,0);
	      \draw (0,-1) to node {\tikz{\draw[-](0,0)--(0.1,0);}} (1,-1);
	      \fill[black] (0,0) circle[radius=0.02];
	      \fill[black] (1,0) circle[radius=0.02];
	      \fill[black] (0,-1) circle[radius=0.02];
	      \fill[black] (1,-1) circle[radius=0.02];
	    \end{scope}
	     \begin{scope}[xshift=2.0cm, yshift = -2.0cm]
	      \draw (0,0) to node[sloped] {\tikz{\draw[->](0,0)--(0.1,0);}} (0,-1);
	      \draw (1,0) to node[sloped] {\tikz{\draw[->](0,0)--(0.1,0);}} (1,-1);
	      \draw (0,0) to node[sloped] {\tikz{\draw[-](0,0)--(0.1,0);}} (1,0);
	      \draw (0,-1) to node  {\tikz{\draw[-](0,0)--(0.1,0);}} (1,-1);
	      \node at (1/2,-1/2) {$$};
	      \draw (0,0) to node {\tikz{\draw[-](0,0)--(0.1,0);}}(1,0);
	      \draw (0,-1) to node {\tikz{\draw[-](0,0)--(0.1,0);}} (1,-1);
	      \fill[black] (0,0) circle[radius=0.02];
	      \fill[black] (1,0) circle[radius=0.02];
	      \fill[black] (0,-1) circle[radius=0.02];
	      \fill[black] (1,-1) circle[radius=0.02];
	    \end{scope}
    \end{scope}
     \draw[->] (3.5,-1.5) to (5.8,-1.5);
      \node at (4.7,-1.5) [fill=white]{$l$};
  \end{tikzpicture}
\end{center}
\caption{Figure of property \ref{pfeq} of 2-crossed module}
\label{d:PF_lift_wd}
\end{figure}
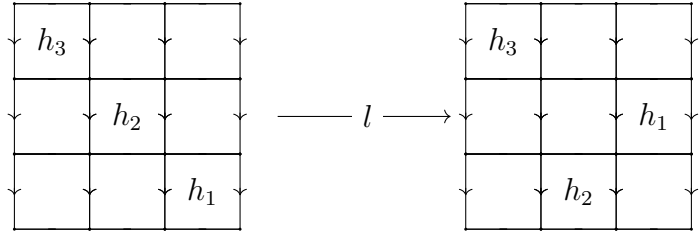
From Figure~\ref{d:PF_lift_wd}, we can derive two different formulas. One is 
\[
l = {\rm id}_{{\rm id}_{\del(h_3)}} \#_1 \{h_2,h_1\} = \act{\del h_3} \{h_2,h_1\}.
\]
 The other is 
 \[
 l =\{{\rm id}_{\del(h_3)}\#_1 h_2, {\rm id}_{\del(h_3)} \#_1h_1 \} =\{\act{\del h_3}h_2,\act{\del h_3}h_1 \} .
 \]
  By Property~\ref{pfeq} of 2-crossed module, these two formulas are identical.

We can now interpret the two new operators in the 3-crossed module — the left-homannian and right-homannian — as illustrated in the diagrams. The left-homannian is a twisted version of Property~\ref{lpl} of 2-crossed module, and it is represented in Figure~\ref{d:lhmn}.
\begin{figure}
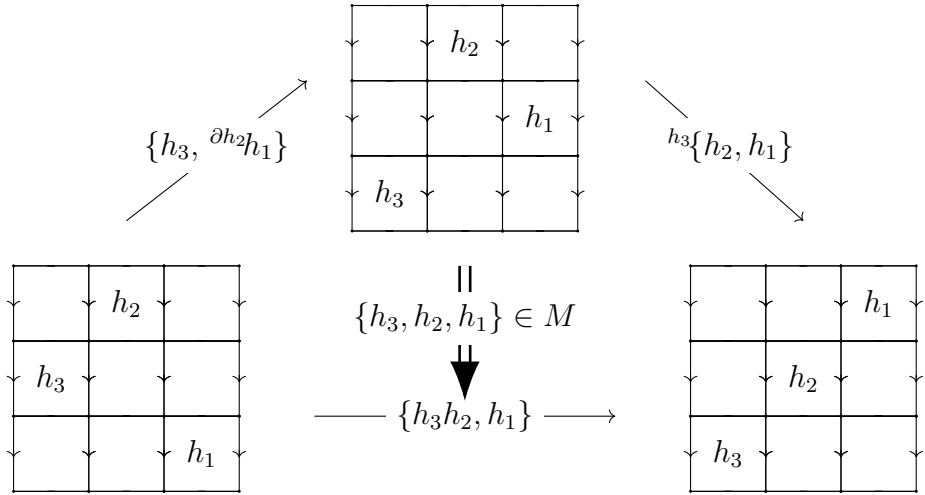

\begin{center}

\end{center}
\caption{Figure of left-homanian}
\label{d:lhmn}
\end{figure}

Similarly, the right-homannian is a twisted version of Property~\ref{rpl} of a 2-crossed module, as shown in Figure~\ref{d:rhmn}.

\begin{figure}
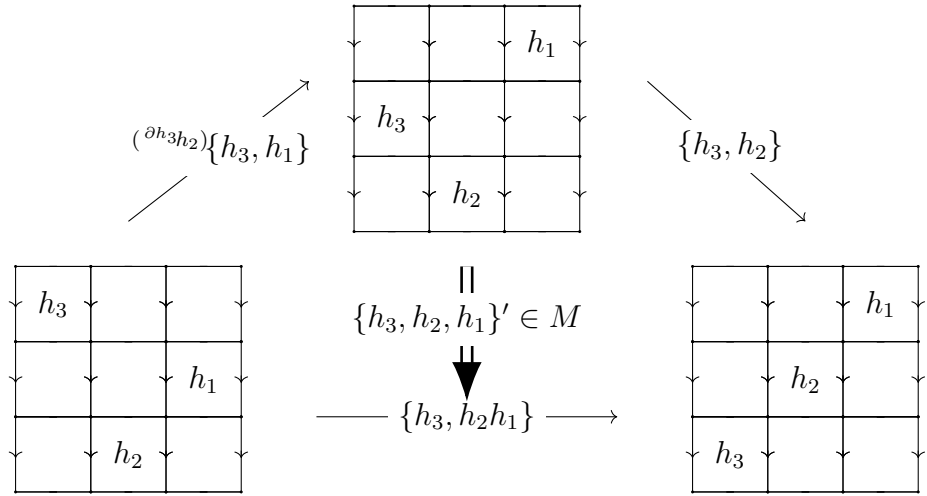

\begin{center}
  %
\end{center}
\caption{Digram of right-homanian}
\label{d:rhmn}
\end{figure}

For an element $m\in M$, just as for elements of $H$ and $L$, we can regard the map 
\[
m : L \longrightarrow L, \qquad l \longmapsto \partial(m)\,l .
\]
 This is illustrated in Figure~\ref{d:fig_m}.
\begin{figure}
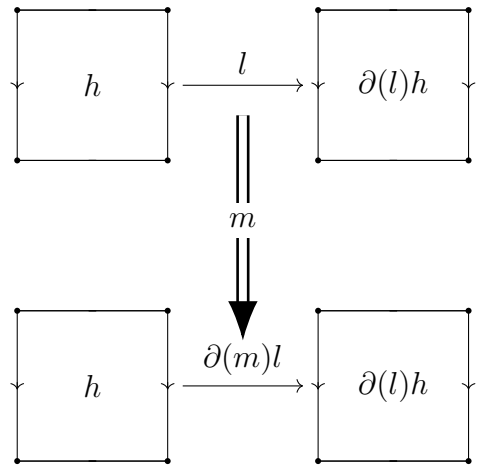

\begin{center}
  %
\end{center}
\caption{Diagram of $M$}
\label{d:fig_m}
\end{figure}

\begin{figure}
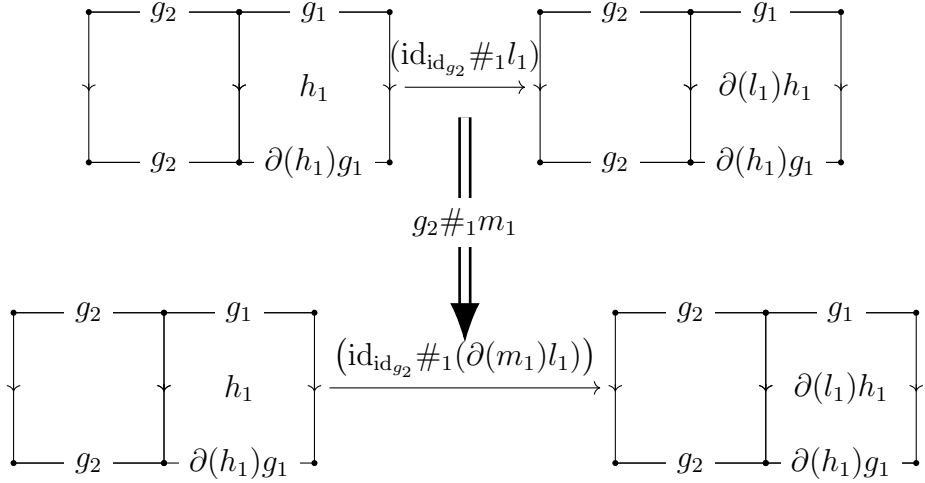

\begin{center}
  %
\end{center}
\caption{Diagram of $g_2 \#_1 m_1$}
\label{d:g2_m1}
\end{figure}

For elements of $M$, there are four types of compositions: $\#_1, \#_2, \#_3$, and $\#_4$. Here, $\#_4$ denotes the vertical composition in $M$, while $\#_1, \#_2$ ,and  $\#_3$ denote horizontal compositions, defined in a similar manner as before.
For each $g_2, g_1 \in G$, $h_2, h_1\in H$, $l_2, l_1 \in L$, and $m_1, m_2 \in M$, we define
\[
g_2 \#_1 m_1  \coloneqq {\rm id}_{{\rm id}_{{\rm id}_{g_2}}} \#_1 m_1, \qquad m_2 \#_1 g_1  \coloneqq m_2 \#_1  {\rm id}_{{\rm id}_{{\rm id}_{g_1}}}.
\]
 These are represented in Figures~\ref{d:g2_m1} and~\ref{d:m2_g1}, respectively. 

\begin{figure}[!]
\begin{center}
  \begin{tikzpicture}[scale=2]
    \begin{scope}[xshift=0.0cm]
      \draw[->] (2.1,-1/2) to (2.9,-1/2);
      \node at (2.5,-1/2) [above]{$( l_2 \#_1  {\rm id}_{{\rm id}_{g_1}} )$};
      \draw (0,0) to node[sloped] {\tikz{\draw[->](0,0)--(0.1,0);}} (0,-1);
      \draw (1,0) to node[sloped] {\tikz{\draw[->](0,0)--(0.1,0);}} (1,-1);
      \draw (0,0) to node[sloped] {\tikz{\draw[-](0,0)--(0.1,0);}} (1,0);
      \draw (0,-1) to node  {\tikz{\draw[-](0,0)--(0.1,0);}} (1,-1);
      \node at (1/2,-1/2) {$h_2$};
      \draw (0,0) to node [fill=white]{$g_2$} (1,0);
      \draw (0,-1) to node [fill=white]{$\del(h_2)g_2$} (1,-1);
      \fill[black] (0,0) circle[radius=0.02];
      \fill[black] (1,0) circle[radius=0.02];
      \fill[black] (0,-1) circle[radius=0.02];
      \fill[black] (1,-1) circle[radius=0.02];
      \begin{scope}[xshift=1.0cm]
      \draw (0,0) to node[sloped] {\tikz{\draw[->](0,0)--(0.1,0);}} (0,-1);
      \draw (1,0) to node[sloped] {\tikz{\draw[->](0,0)--(0.1,0);}} (1,-1);
      \draw (0,0) to node[sloped] {\tikz{\draw[-](0,0)--(0.1,0);}} (1,0);
      \node at (1/2,-1/2) {$$};
      \draw (0,0) to node [fill=white]{$g_1$} (1,0);
      \draw (0,-1) to node[fill=white]{$g_1$} (1,-1);
      \fill[black] (0,0) circle[radius=0.02];
      \fill[black] (1,0) circle[radius=0.02];
      \fill[black] (0,-1) circle[radius=0.02];
      \fill[black] (1,-1) circle[radius=0.02];
	\end{scope}
    \begin{scope}[xshift=3.0cm]
      \draw (0,0) to node[sloped] {\tikz{\draw[->](0,0)--(0.1,0);}} (0,-1);
      \draw (1,0) to node[sloped] {\tikz{\draw[->](0,0)--(0.1,0);}} (1,-1);
      \draw (0,0) to node[sloped] {\tikz{\draw[-](0,0)--(0.1,0);}} (1,0);
      \draw (0,-1) to node  {\tikz{\draw[-](0,0)--(0.1,0);}} (1,-1);
      \node at (1/2,-1/2) {$\del(l_2)h_2$};
	 \draw (0,0) to node [fill=white]{$g_2$} (1,0);
      \draw (0,-1) to node [fill=white]{$\del(h_2)g_2$} (1,-1);
      \fill[black] (0,0) circle[radius=0.02];
      \fill[black] (1,0) circle[radius=0.02];
      \fill[black] (0,-1) circle[radius=0.02];
      \fill[black] (1,-1) circle[radius=0.02];
	\end{scope}
	\begin{scope}[xshift=4.0cm]
      \draw (0,0) to node[sloped] {\tikz{\draw[->](0,0)--(0.1,0);}} (0,-1);
      \draw (1,0) to node[sloped] {\tikz{\draw[->](0,0)--(0.1,0);}} (1,-1);
      \draw (0,0) to node[sloped] {\tikz{\draw[-](0,0)--(0.1,0);}} (1,0);
      \draw (0,-1) to node  {\tikz{\draw[-](0,0)--(0.1,0);}} (1,-1);
      \node at (1/2,-1/2) {$$};
      \draw (0,0) to node [fill=white]{$g_1$} (1,0);
      \draw (0,-1) to node[fill=white]{$g_1$} (1,-1);
      \fill[black] (0,0) circle[radius=0.02];
      \fill[black] (1,0) circle[radius=0.02];
      \fill[black] (0,-1) circle[radius=0.02];
      \fill[black] (1,-1) circle[radius=0.02];
	\end{scope}
	  \end{scope}
	 \begin{scope}[xshift=0cm, yshift=-2.0cm]
	 \draw[->] (1.6,-1/2) to (3.4,-1/2);
      \node at (2.5,-1/2) [above]{$\left( (\del(m_2)l_2) \#_1  {\rm id}_{{\rm id}_{g_1}}\right)$};
	 \begin{scope}[xshift=-0.5cm]
      \draw (0,0) to node[sloped] {\tikz{\draw[->](0,0)--(0.1,0);}} (0,-1);
      \draw (1,0) to node[sloped] {\tikz{\draw[->](0,0)--(0.1,0);}} (1,-1);
      \draw (0,0) to node[sloped] {\tikz{\draw[-](0,0)--(0.1,0);}} (1,0);
      \draw (0,-1) to node  {\tikz{\draw[-](0,0)--(0.1,0);}} (1,-1);
      \node at (1/2,-1/2) {$h_2$};
      \draw (0,0) to node [fill=white]{$g_2$} (1,0);
      \draw (0,-1) to node [fill=white]{$\del(h_2)g_2$} (1,-1);
      \fill[black] (0,0) circle[radius=0.02];
      \fill[black] (1,0) circle[radius=0.02];
      \fill[black] (0,-1) circle[radius=0.02];
      \fill[black] (1,-1) circle[radius=0.02];
      \end{scope}
      \begin{scope}[xshift=0.5cm]
      \draw (0,0) to node[sloped] {\tikz{\draw[->](0,0)--(0.1,0);}} (0,-1);
      \draw (1,0) to node[sloped] {\tikz{\draw[->](0,0)--(0.1,0);}} (1,-1);
      \draw (0,0) to node[sloped] {\tikz{\draw[-](0,0)--(0.1,0);}} (1,0);
      \node at (1/2,-1/2) {$$};
      \draw (0,0) to node [fill=white]{$g_1$} (1,0);
      \draw (0,-1) to node[fill=white]{$g_1$} (1,-1);
      \fill[black] (0,0) circle[radius=0.02];
      \fill[black] (1,0) circle[radius=0.02];
      \fill[black] (0,-1) circle[radius=0.02];
      \fill[black] (1,-1) circle[radius=0.02];
	\end{scope}
    \begin{scope}[xshift=3.5cm]
      \draw (0,0) to node[sloped] {\tikz{\draw[->](0,0)--(0.1,0);}} (0,-1);
      \draw (1,0) to node[sloped] {\tikz{\draw[->](0,0)--(0.1,0);}} (1,-1);
      \draw (0,0) to node[sloped] {\tikz{\draw[-](0,0)--(0.1,0);}} (1,0);
      \draw (0,-1) to node  {\tikz{\draw[-](0,0)--(0.1,0);}} (1,-1);
      \node at (1/2,-1/2) {$\del(l_2)h_2$};
      \draw (0,0) to node [fill=white]{$g_2$} (1,0);
      \draw (0,-1) to node [fill=white]{$\del(h_2)g_2$} (1,-1);
      \fill[black] (0,0) circle[radius=0.02];
      \fill[black] (1,0) circle[radius=0.02];
      \fill[black] (0,-1) circle[radius=0.02];
      \fill[black] (1,-1) circle[radius=0.02];
	\end{scope}
	\begin{scope}[xshift=4.5cm]
      \draw (0,0) to node[sloped] {\tikz{\draw[->](0,0)--(0.1,0);}} (0,-1);
      \draw (1,0) to node[sloped] {\tikz{\draw[->](0,0)--(0.1,0);}} (1,-1);
      \draw (0,0) to node[sloped] {\tikz{\draw[-](0,0)--(0.1,0);}} (1,0);
      \draw (0,-1) to node  {\tikz{\draw[-](0,0)--(0.1,0);}} (1,-1);
      \node at (1/2,-1/2) {$$};
      \draw (0,0) to node [fill=white]{$g_1$} (1,0);
      \draw (0,-1) to node[fill=white]{$g_1$} (1,-1);
      \fill[black] (0,0) circle[radius=0.02];
      \fill[black] (1,0) circle[radius=0.02];
      \fill[black] (0,-1) circle[radius=0.02];
      \fill[black] (1,-1) circle[radius=0.02];
	\end{scope}
	  \end{scope}
	  \draw [line width=1pt, double distance=3pt, arrows = {-Latex[length=0pt 3 0]}] (2.5,-0.7) -- (2.5,-2.2);
      \node at (2.5, -1.4)  [fill=white]{$m_2\#_1 g_1$};
  \end{tikzpicture}
\end{center}
\caption{Diagram of $m_2 \#_1 g_1$}
\label{d:m2_g1}
\end{figure}
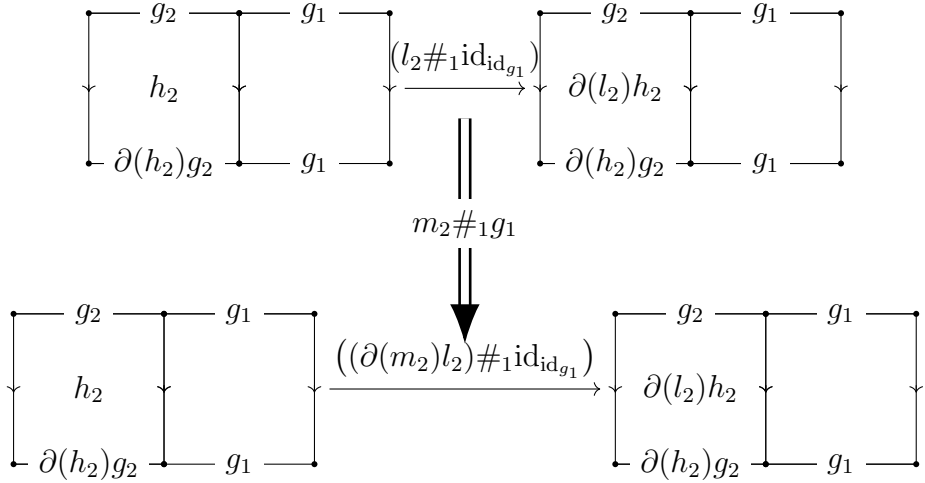

For each $h_2, h_1\in H$, $l_2, l_1 \in L$, and $m_1, m_2 \in M$, we define 
\[
h_2 \#_2 m_1 \coloneqq {\rm id}_{{\rm id}_{h_2}} \#_2 m_1, \qquad m_2 \#_2 h_1 \coloneqq m_1 \#_2 {\rm id}_{{\rm id}_{h_1}},
\]
 which are represented in Figures~\ref{d:h2_m1} and~\ref{d:m2_h1}, respectively.

\begin{figure}
\begin{center}
  \begin{tikzpicture}[scale=2]
    \begin{scope}[xshift=0.0cm]
      \draw[->] (1.1,-1) to (2.4,-1);
      \node at (1.75,-1) [above]{$(   {\rm id}_{h_2}  \#_1 l_1 )$};
      \draw (0,0) to node[sloped] {\tikz{\draw[->](0,0)--(0.1,0);}} (0,-1);
      \draw (1,0) to node[sloped] {\tikz{\draw[->](0,0)--(0.1,0);}} (1,-1);
      \draw (0,0) to node[sloped] {\tikz{\draw[-](0,0)--(0.1,0);}} (1,0);
      \draw (0,-1) to node  {\tikz{\draw[-](0,0)--(0.1,0);}} (1,-1);
      \node at (1/2,-1/2) {$h_1$};
      \draw (0,0) to node {\tikz{\draw[-](0,0)--(0.1,0);}} (1,0);
      \draw (0,-1) to node {\tikz{\draw[-](0,0)--(0.1,0);}}(1,-1);
      \fill[black] (0,0) circle[radius=0.02];
      \fill[black] (1,0) circle[radius=0.02];
      \fill[black] (0,-1) circle[radius=0.02];
      \fill[black] (1,-1) circle[radius=0.02];
      \begin{scope}[yshift=-1.0cm]
      \draw (0,0) to node[sloped] {\tikz{\draw[->](0,0)--(0.1,0);}} (0,-1);
      \draw (1,0) to node[sloped] {\tikz{\draw[->](0,0)--(0.1,0);}} (1,-1);
      \draw (0,0) to node[sloped] {\tikz{\draw[-](0,0)--(0.1,0);}} (1,0);
      \node at (1/2,-1/2) {$h_2$};
      \draw (0,0) to node {\tikz{\draw[-](0,0)--(0.1,0);}} (1,0);
      \draw (0,-1) to node{\tikz{\draw[-](0,0)--(0.1,0);}} (1,-1);
      \fill[black] (0,0) circle[radius=0.02];
      \fill[black] (1,0) circle[radius=0.02];
      \fill[black] (0,-1) circle[radius=0.02];
      \fill[black] (1,-1) circle[radius=0.02];
	\end{scope}
    \begin{scope}[xshift=2.5cm]
      \draw (0,0) to node[sloped] {\tikz{\draw[->](0,0)--(0.1,0);}} (0,-1);
      \draw (1,0) to node[sloped] {\tikz{\draw[->](0,0)--(0.1,0);}} (1,-1);
      \draw (0,0) to node[sloped] {\tikz{\draw[-](0,0)--(0.1,0);}} (1,0);
      \draw (0,-1) to node  {\tikz{\draw[-](0,0)--(0.1,0);}} (1,-1);
      \node at (1/2,-1/2) {$\del(l_1)h_1$};
	 \draw (0,0) to node {\tikz{\draw[-](0,0)--(0.1,0);}} (1,0);
      \draw (0,-1) to node {\tikz{\draw[-](0,0)--(0.1,0);}} (1,-1);
      \fill[black] (0,0) circle[radius=0.02];
      \fill[black] (1,0) circle[radius=0.02];
      \fill[black] (0,-1) circle[radius=0.02];
      \fill[black] (1,-1) circle[radius=0.02];
	\end{scope}
	\begin{scope}[xshift=2.5cm, yshift=-1.0cm]
      \draw (0,0) to node[sloped] {\tikz{\draw[->](0,0)--(0.1,0);}} (0,-1);
      \draw (1,0) to node[sloped] {\tikz{\draw[->](0,0)--(0.1,0);}} (1,-1);
      \draw (0,0) to node[sloped] {\tikz{\draw[-](0,0)--(0.1,0);}} (1,0);
      \draw (0,-1) to node  {\tikz{\draw[-](0,0)--(0.1,0);}} (1,-1);
      \node at (1/2,-1/2) {$h_2$};
      \draw (0,0) to node {\tikz{\draw[-](0,0)--(0.1,0);}}(1,0);
      \draw (0,-1) to node{\tikz{\draw[-](0,0)--(0.1,0);}} (1,-1);
      \fill[black] (0,0) circle[radius=0.02];
      \fill[black] (1,0) circle[radius=0.02];
      \fill[black] (0,-1) circle[radius=0.02];
      \fill[black] (1,-1) circle[radius=0.02];
	\end{scope}
	  \end{scope}
	 \begin{scope}[xshift=0cm, yshift=-2.2cm]
	 \draw[->] (1.1,-1) to (2.4,-1);
      \node at (1.75,-1) [above]{$\left({\rm id}_{h_2}  \#_1\del(m_1) l_1\right)$};
	 \begin{scope}[xshift=-0.3cm]
      \draw (0,0) to node[sloped] {\tikz{\draw[->](0,0)--(0.1,0);}} (0,-1);
      \draw (1,0) to node[sloped] {\tikz{\draw[->](0,0)--(0.1,0);}} (1,-1);
      \draw (0,0) to node[sloped] {\tikz{\draw[-](0,0)--(0.1,0);}} (1,0);
      \draw (0,-1) to node  {\tikz{\draw[-](0,0)--(0.1,0);}} (1,-1);
      \node at (1/2,-1/2) {$h_1$};
      \draw (0,0) to node {\tikz{\draw[-](0,0)--(0.1,0);}} (1,0);
      \draw (0,-1) to node {\tikz{\draw[-](0,0)--(0.1,0);}} (1,-1);
      \fill[black] (0,0) circle[radius=0.02];
      \fill[black] (1,0) circle[radius=0.02];
      \fill[black] (0,-1) circle[radius=0.02];
      \fill[black] (1,-1) circle[radius=0.02];
      \end{scope}
      \begin{scope}[xshift=-0.3cm,yshift=-1.0cm]
      \draw (0,0) to node[sloped] {\tikz{\draw[->](0,0)--(0.1,0);}} (0,-1);
      \draw (1,0) to node[sloped] {\tikz{\draw[->](0,0)--(0.1,0);}} (1,-1);
      \draw (0,0) to node[sloped] {\tikz{\draw[-](0,0)--(0.1,0);}} (1,0);
      \node at (1/2,-1/2) {$ h_2$};
      \draw (0,0) to node {\tikz{\draw[-](0,0)--(0.1,0);}} (1,0);
      \draw (0,-1) to node{\tikz{\draw[-](0,0)--(0.1,0);}} (1,-1);
      \fill[black] (0,0) circle[radius=0.02];
      \fill[black] (1,0) circle[radius=0.02];
      \fill[black] (0,-1) circle[radius=0.02];
      \fill[black] (1,-1) circle[radius=0.02];
	\end{scope}
    \begin{scope}[xshift=2.8cm]
      \draw (0,0) to node[sloped] {\tikz{\draw[->](0,0)--(0.1,0);}} (0,-1);
      \draw (1,0) to node[sloped] {\tikz{\draw[->](0,0)--(0.1,0);}} (1,-1);
      \draw (0,0) to node[sloped] {\tikz{\draw[-](0,0)--(0.1,0);}} (1,0);
      \draw (0,-1) to node  {\tikz{\draw[-](0,0)--(0.1,0);}} (1,-1);
      \node at (1/2,-1/2) {$\del(l_1)h_1$};
      \draw (0,0) to node {\tikz{\draw[-](0,0)--(0.1,0);}} (1,0);
      \draw (0,-1) to node {\tikz{\draw[-](0,0)--(0.1,0);}} (1,-1);
      \fill[black] (0,0) circle[radius=0.02];
      \fill[black] (1,0) circle[radius=0.02];
      \fill[black] (0,-1) circle[radius=0.02];
      \fill[black] (1,-1) circle[radius=0.02];
	\end{scope}
	\begin{scope}[xshift=2.8cm,yshift=-1.0cm]
      \draw (0,0) to node[sloped] {\tikz{\draw[->](0,0)--(0.1,0);}} (0,-1);
      \draw (1,0) to node[sloped] {\tikz{\draw[->](0,0)--(0.1,0);}} (1,-1);
      \draw (0,0) to node[sloped] {\tikz{\draw[-](0,0)--(0.1,0);}} (1,0);
      \draw (0,-1) to node  {\tikz{\draw[-](0,0)--(0.1,0);}} (1,-1);
      \node at (1/2,-1/2) {$h_2$};
      \draw (0,0) to node {\tikz{\draw[-](0,0)--(0.1,0);}} (1,0);
      \draw (0,-1) to node{\tikz{\draw[-](0,0)--(0.1,0);}} (1,-1);
      \fill[black] (0,0) circle[radius=0.02];
      \fill[black] (1,0) circle[radius=0.02];
      \fill[black] (0,-1) circle[radius=0.02];
      \fill[black] (1,-1) circle[radius=0.02];
	\end{scope}
	  \end{scope}
	  \draw [line width=1pt, double distance=3pt, arrows = {-Latex[length=0pt 3 0]}] (1.75,-1.1) -- (1.75,-2.8);
      \node at (1.75, -1.9)  [fill=white]{$ h_2 \#_2 m_1$};
  \end{tikzpicture}
\end{center}
\caption{Diagram of $h_2 \#_2 m_1$}
\label{d:h2_m1}
\end{figure}
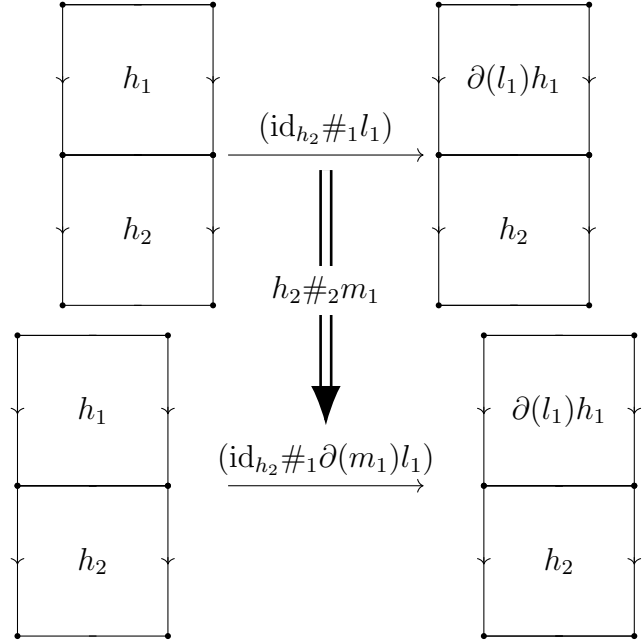

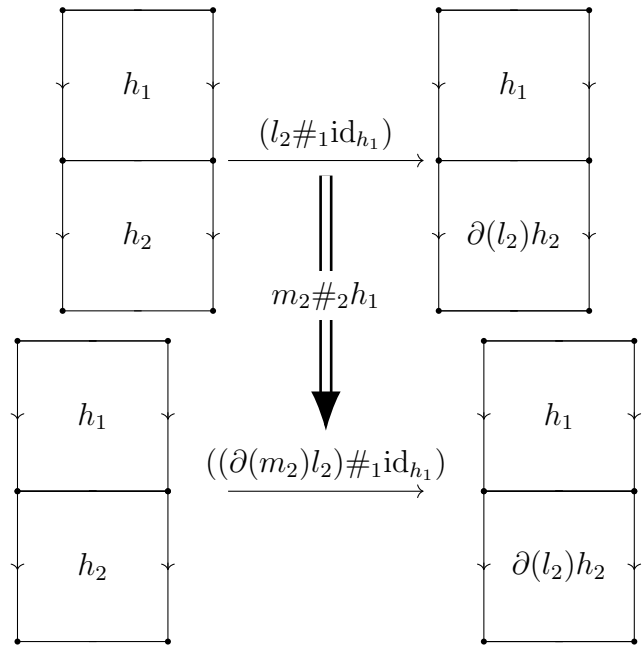
\begin{figure}
\begin{center}
  \begin{tikzpicture}[scale=2]
    \begin{scope}[xshift=0.0cm]
      \draw[->] (1.1,-1) to (2.4,-1);
      \node at (1.75,-1) [above]{$( l_2 \#_1  {\rm id}_{h_1} )$};
      \draw (0,0) to node[sloped] {\tikz{\draw[->](0,0)--(0.1,0);}} (0,-1);
      \draw (1,0) to node[sloped] {\tikz{\draw[->](0,0)--(0.1,0);}} (1,-1);
      \draw (0,0) to node[sloped] {\tikz{\draw[-](0,0)--(0.1,0);}} (1,0);
      \draw (0,-1) to node  {\tikz{\draw[-](0,0)--(0.1,0);}} (1,-1);
      \node at (1/2,-1/2) {$h_1$};
      \draw (0,0) to node {\tikz{\draw[-](0,0)--(0.1,0);}} (1,0);
      \draw (0,-1) to node {\tikz{\draw[-](0,0)--(0.1,0);}}(1,-1);
      \fill[black] (0,0) circle[radius=0.02];
      \fill[black] (1,0) circle[radius=0.02];
      \fill[black] (0,-1) circle[radius=0.02];
      \fill[black] (1,-1) circle[radius=0.02];
      \begin{scope}[yshift=-1.0cm]
      \draw (0,0) to node[sloped] {\tikz{\draw[->](0,0)--(0.1,0);}} (0,-1);
      \draw (1,0) to node[sloped] {\tikz{\draw[->](0,0)--(0.1,0);}} (1,-1);
      \draw (0,0) to node[sloped] {\tikz{\draw[-](0,0)--(0.1,0);}} (1,0);
      \node at (1/2,-1/2) {$h_2$};
      \draw (0,0) to node {\tikz{\draw[-](0,0)--(0.1,0);}} (1,0);
      \draw (0,-1) to node{\tikz{\draw[-](0,0)--(0.1,0);}} (1,-1);
      \fill[black] (0,0) circle[radius=0.02];
      \fill[black] (1,0) circle[radius=0.02];
      \fill[black] (0,-1) circle[radius=0.02];
      \fill[black] (1,-1) circle[radius=0.02];
	\end{scope}
    \begin{scope}[xshift=2.5cm]
      \draw (0,0) to node[sloped] {\tikz{\draw[->](0,0)--(0.1,0);}} (0,-1);
      \draw (1,0) to node[sloped] {\tikz{\draw[->](0,0)--(0.1,0);}} (1,-1);
      \draw (0,0) to node[sloped] {\tikz{\draw[-](0,0)--(0.1,0);}} (1,0);
      \draw (0,-1) to node  {\tikz{\draw[-](0,0)--(0.1,0);}} (1,-1);
      \node at (1/2,-1/2) {$h_1$};
	 \draw (0,0) to node {\tikz{\draw[-](0,0)--(0.1,0);}} (1,0);
      \draw (0,-1) to node {\tikz{\draw[-](0,0)--(0.1,0);}} (1,-1);
      \fill[black] (0,0) circle[radius=0.02];
      \fill[black] (1,0) circle[radius=0.02];
      \fill[black] (0,-1) circle[radius=0.02];
      \fill[black] (1,-1) circle[radius=0.02];
	\end{scope}
	\begin{scope}[xshift=2.5cm, yshift=-1.0cm]
      \draw (0,0) to node[sloped] {\tikz{\draw[->](0,0)--(0.1,0);}} (0,-1);
      \draw (1,0) to node[sloped] {\tikz{\draw[->](0,0)--(0.1,0);}} (1,-1);
      \draw (0,0) to node[sloped] {\tikz{\draw[-](0,0)--(0.1,0);}} (1,0);
      \draw (0,-1) to node  {\tikz{\draw[-](0,0)--(0.1,0);}} (1,-1);
      \node at (1/2,-1/2) {$\del(l_2)h_2$};
      \draw (0,0) to node {\tikz{\draw[-](0,0)--(0.1,0);}}(1,0);
      \draw (0,-1) to node{\tikz{\draw[-](0,0)--(0.1,0);}} (1,-1);
      \fill[black] (0,0) circle[radius=0.02];
      \fill[black] (1,0) circle[radius=0.02];
      \fill[black] (0,-1) circle[radius=0.02];
      \fill[black] (1,-1) circle[radius=0.02];
	\end{scope}
	  \end{scope}
	 \begin{scope}[xshift=0cm, yshift=-2.2cm]
	 \draw[->] (1.1,-1) to (2.4,-1);
      \node at (1.75,-1) [above]{$\left( (\del(m_2)l_2) \#_1  {\rm id}_{h_1}\right)$};
	 \begin{scope}[xshift=-0.3cm]
      \draw (0,0) to node[sloped] {\tikz{\draw[->](0,0)--(0.1,0);}} (0,-1);
      \draw (1,0) to node[sloped] {\tikz{\draw[->](0,0)--(0.1,0);}} (1,-1);
      \draw (0,0) to node[sloped] {\tikz{\draw[-](0,0)--(0.1,0);}} (1,0);
      \draw (0,-1) to node  {\tikz{\draw[-](0,0)--(0.1,0);}} (1,-1);
      \node at (1/2,-1/2) {$h_1$};
      \draw (0,0) to node {\tikz{\draw[-](0,0)--(0.1,0);}} (1,0);
      \draw (0,-1) to node {\tikz{\draw[-](0,0)--(0.1,0);}} (1,-1);
      \fill[black] (0,0) circle[radius=0.02];
      \fill[black] (1,0) circle[radius=0.02];
      \fill[black] (0,-1) circle[radius=0.02];
      \fill[black] (1,-1) circle[radius=0.02];
      \end{scope}
      \begin{scope}[xshift=-0.3cm,yshift=-1.0cm]
      \draw (0,0) to node[sloped] {\tikz{\draw[->](0,0)--(0.1,0);}} (0,-1);
      \draw (1,0) to node[sloped] {\tikz{\draw[->](0,0)--(0.1,0);}} (1,-1);
      \draw (0,0) to node[sloped] {\tikz{\draw[-](0,0)--(0.1,0);}} (1,0);
      \node at (1/2,-1/2) {$ h_2$};
      \draw (0,0) to node {\tikz{\draw[-](0,0)--(0.1,0);}} (1,0);
      \draw (0,-1) to node{\tikz{\draw[-](0,0)--(0.1,0);}} (1,-1);
      \fill[black] (0,0) circle[radius=0.02];
      \fill[black] (1,0) circle[radius=0.02];
      \fill[black] (0,-1) circle[radius=0.02];
      \fill[black] (1,-1) circle[radius=0.02];
	\end{scope}
    \begin{scope}[xshift=2.8cm]
      \draw (0,0) to node[sloped] {\tikz{\draw[->](0,0)--(0.1,0);}} (0,-1);
      \draw (1,0) to node[sloped] {\tikz{\draw[->](0,0)--(0.1,0);}} (1,-1);
      \draw (0,0) to node[sloped] {\tikz{\draw[-](0,0)--(0.1,0);}} (1,0);
      \draw (0,-1) to node  {\tikz{\draw[-](0,0)--(0.1,0);}} (1,-1);
      \node at (1/2,-1/2) {$h_1$};
      \draw (0,0) to node {\tikz{\draw[-](0,0)--(0.1,0);}} (1,0);
      \draw (0,-1) to node {\tikz{\draw[-](0,0)--(0.1,0);}} (1,-1);
      \fill[black] (0,0) circle[radius=0.02];
      \fill[black] (1,0) circle[radius=0.02];
      \fill[black] (0,-1) circle[radius=0.02];
      \fill[black] (1,-1) circle[radius=0.02];
	\end{scope}
	\begin{scope}[xshift=2.8cm,yshift=-1.0cm]
      \draw (0,0) to node[sloped] {\tikz{\draw[->](0,0)--(0.1,0);}} (0,-1);
      \draw (1,0) to node[sloped] {\tikz{\draw[->](0,0)--(0.1,0);}} (1,-1);
      \draw (0,0) to node[sloped] {\tikz{\draw[-](0,0)--(0.1,0);}} (1,0);
      \draw (0,-1) to node  {\tikz{\draw[-](0,0)--(0.1,0);}} (1,-1);
      \node at (1/2,-1/2) {$\del(l_2)h_2$};
      \draw (0,0) to node {\tikz{\draw[-](0,0)--(0.1,0);}} (1,0);
      \draw (0,-1) to node{\tikz{\draw[-](0,0)--(0.1,0);}} (1,-1);
      \fill[black] (0,0) circle[radius=0.02];
      \fill[black] (1,0) circle[radius=0.02];
      \fill[black] (0,-1) circle[radius=0.02];
      \fill[black] (1,-1) circle[radius=0.02];
	\end{scope}
	  \end{scope}
	  \draw [line width=1pt, double distance=3pt, arrows = {-Latex[length=0pt 3 0]}] (1.75,-1.1) -- (1.75,-2.8);
      \node at (1.75, -1.9)  [fill=white]{$m_2\#_2 h_1$};
  \end{tikzpicture}
\end{center}
\caption{Diagram of $m_2 \#_2 h_1$}
\label{d:m2_h1}
\end{figure}

For each $h\in H$, $l_2, l_1 \in L$, and $m_1, m_2 \in M$, we define
\[
l_2 \#_3 m_1 \coloneqq {\rm id}_{l_2} \#_3 m_1, \qquad m_2 \#_2 l_1 \coloneqq  m_2 \#_2 {\rm id}_{l_1},
\]
 which are represented in Figures~\ref{d:l2_m1} and~\ref{d:m2_l1}, respectively.

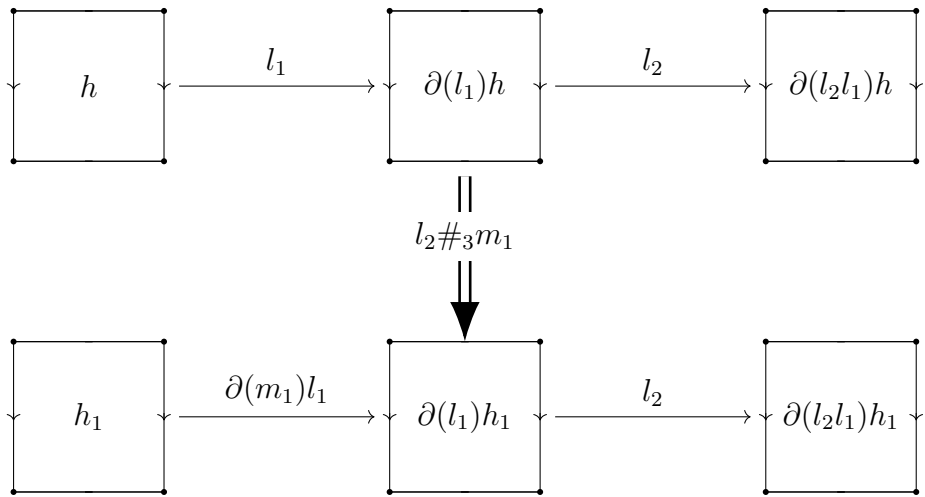
\begin{figure}
\begin{center}
  \begin{tikzpicture}[scale=2]
    \begin{scope}[xshift=0.0cm]
      \draw[->] (1.1,-1/2) to (2.4,-1/2);
      \node at (1.75,-1/2) [above]{$l_1$};
      \draw (0,0) to node[sloped] {\tikz{\draw[->](0,0)--(0.1,0);}} (0,-1);
      \draw (1,0) to node[sloped] {\tikz{\draw[->](0,0)--(0.1,0);}} (1,-1);
      \draw (0,0) to node[sloped] {\tikz{\draw[-](0,0)--(0.1,0);}} (1,0);
      \draw (0,-1) to node  {\tikz{\draw[-](0,0)--(0.1,0);}} (1,-1);
      \node at (1/2,-1/2) {$h$};
      \draw (0,0) to node {\tikz{\draw[-](0,0)--(0.1,0);}} (1,0);
      \draw (0,-1) to node {\tikz{\draw[-](0,0)--(0.1,0);}}(1,-1);
      \fill[black] (0,0) circle[radius=0.02];
      \fill[black] (1,0) circle[radius=0.02];
      \fill[black] (0,-1) circle[radius=0.02];
      \fill[black] (1,-1) circle[radius=0.02];   
	  \end{scope}
	  \begin{scope}[xshift=2.5cm]
      \draw[->] (1.1,-1/2) to (2.4,-1/2);
      \node at (1.75,-1/2) [above]{$l_2$};
      \draw (0,0) to node[sloped] {\tikz{\draw[->](0,0)--(0.1,0);}} (0,-1);
      \draw (1,0) to node[sloped] {\tikz{\draw[->](0,0)--(0.1,0);}} (1,-1);
      \draw (0,0) to node[sloped] {\tikz{\draw[-](0,0)--(0.1,0);}} (1,0);
      \draw (0,-1) to node  {\tikz{\draw[-](0,0)--(0.1,0);}} (1,-1);
      \node at (1/2,-1/2) {$\del(l_1)h$};
      \draw (0,0) to node {\tikz{\draw[-](0,0)--(0.1,0);}} (1,0);
      \draw (0,-1) to node {\tikz{\draw[-](0,0)--(0.1,0);}}(1,-1);
      \fill[black] (0,0) circle[radius=0.02];
      \fill[black] (1,0) circle[radius=0.02];
      \fill[black] (0,-1) circle[radius=0.02];
      \fill[black] (1,-1) circle[radius=0.02];   
	  \end{scope}
	    \begin{scope}[xshift=5.0cm]
      \draw (0,0) to node[sloped] {\tikz{\draw[->](0,0)--(0.1,0);}} (0,-1);
      \draw (1,0) to node[sloped] {\tikz{\draw[->](0,0)--(0.1,0);}} (1,-1);
      \draw (0,0) to node[sloped] {\tikz{\draw[-](0,0)--(0.1,0);}} (1,0);
      \draw (0,-1) to node  {\tikz{\draw[-](0,0)--(0.1,0);}} (1,-1);
      \node at (1/2,-1/2) {$\del(l_2 l_1)h$};
      \draw (0,0) to node {\tikz{\draw[-](0,0)--(0.1,0);}} (1,0);
      \draw (0,-1) to node {\tikz{\draw[-](0,0)--(0.1,0);}}(1,-1);
      \fill[black] (0,0) circle[radius=0.02];
      \fill[black] (1,0) circle[radius=0.02];
      \fill[black] (0,-1) circle[radius=0.02];
      \fill[black] (1,-1) circle[radius=0.02];   
	  \end{scope}
	 \begin{scope}[xshift=0cm, yshift=-2.2cm]
	 \begin{scope}[xshift=0.0cm]
      \draw[->] (1.1,-1/2) to (2.4,-1/2);
      \node at (1.75,-1/2) [above]{$\del(m_1)l_1$};
      \draw (0,0) to node[sloped] {\tikz{\draw[->](0,0)--(0.1,0);}} (0,-1);
      \draw (1,0) to node[sloped] {\tikz{\draw[->](0,0)--(0.1,0);}} (1,-1);
      \draw (0,0) to node[sloped] {\tikz{\draw[-](0,0)--(0.1,0);}} (1,0);
      \draw (0,-1) to node  {\tikz{\draw[-](0,0)--(0.1,0);}} (1,-1);
      \node at (1/2,-1/2) {$h_1$};
      \draw (0,0) to node {\tikz{\draw[-](0,0)--(0.1,0);}} (1,0);
      \draw (0,-1) to node {\tikz{\draw[-](0,0)--(0.1,0);}}(1,-1);
      \fill[black] (0,0) circle[radius=0.02];
      \fill[black] (1,0) circle[radius=0.02];
      \fill[black] (0,-1) circle[radius=0.02];
      \fill[black] (1,-1) circle[radius=0.02];   
	  \end{scope}
	  \begin{scope}[xshift=2.5cm]
      \draw[->] (1.1,-1/2) to (2.4,-1/2);
      \node at (1.75,-1/2) [above]{$l_2$};
      \draw (0,0) to node[sloped] {\tikz{\draw[->](0,0)--(0.1,0);}} (0,-1);
      \draw (1,0) to node[sloped] {\tikz{\draw[->](0,0)--(0.1,0);}} (1,-1);
      \draw (0,0) to node[sloped] {\tikz{\draw[-](0,0)--(0.1,0);}} (1,0);
      \draw (0,-1) to node  {\tikz{\draw[-](0,0)--(0.1,0);}} (1,-1);
      \node at (1/2,-1/2) {$\del(l_1)h_1$};
      \draw (0,0) to node {\tikz{\draw[-](0,0)--(0.1,0);}} (1,0);
      \draw (0,-1) to node {\tikz{\draw[-](0,0)--(0.1,0);}}(1,-1);
      \fill[black] (0,0) circle[radius=0.02];
      \fill[black] (1,0) circle[radius=0.02];
      \fill[black] (0,-1) circle[radius=0.02];
      \fill[black] (1,-1) circle[radius=0.02];   
	  \end{scope}
	    \begin{scope}[xshift=5.0cm]
      \draw (0,0) to node[sloped] {\tikz{\draw[->](0,0)--(0.1,0);}} (0,-1);
      \draw (1,0) to node[sloped] {\tikz{\draw[->](0,0)--(0.1,0);}} (1,-1);
      \draw (0,0) to node[sloped] {\tikz{\draw[-](0,0)--(0.1,0);}} (1,0);
      \draw (0,-1) to node  {\tikz{\draw[-](0,0)--(0.1,0);}} (1,-1);
      \node at (1/2,-1/2) {$\del(l_2 l_1)h_1$};
      \draw (0,0) to node {\tikz{\draw[-](0,0)--(0.1,0);}} (1,0);
      \draw (0,-1) to node {\tikz{\draw[-](0,0)--(0.1,0);}}(1,-1);
      \fill[black] (0,0) circle[radius=0.02];
      \fill[black] (1,0) circle[radius=0.02];
      \fill[black] (0,-1) circle[radius=0.02];
      \fill[black] (1,-1) circle[radius=0.02];   
	  \end{scope}
	  \end{scope}
	  \draw [line width=1pt, double distance=3pt, arrows = {-Latex[length=0pt 3 0]}] (3.0,-1.1) -- (3.0,-2.2);
      \node at (3.0, -1.5)  [fill=white]{${l_2} \#_3 m_1$};
  \end{tikzpicture}
\end{center}
\caption{Diagram of $l_2 \#_3 m_1$}
\label{d:l2_m1}
\end{figure}

\begin{figure}
\begin{center}
  \begin{tikzpicture}[scale=2]
    \begin{scope}[xshift=0.0cm]
      \draw[->] (1.1,-1/2) to (2.4,-1/2);
      \node at (1.75,-1/2) [above]{$l_1$};
      \draw (0,0) to node[sloped] {\tikz{\draw[->](0,0)--(0.1,0);}} (0,-1);
      \draw (1,0) to node[sloped] {\tikz{\draw[->](0,0)--(0.1,0);}} (1,-1);
      \draw (0,0) to node[sloped] {\tikz{\draw[-](0,0)--(0.1,0);}} (1,0);
      \draw (0,-1) to node  {\tikz{\draw[-](0,0)--(0.1,0);}} (1,-1);
      \node at (1/2,-1/2) {$h$};
      \draw (0,0) to node {\tikz{\draw[-](0,0)--(0.1,0);}} (1,0);
      \draw (0,-1) to node {\tikz{\draw[-](0,0)--(0.1,0);}}(1,-1);
      \fill[black] (0,0) circle[radius=0.02];
      \fill[black] (1,0) circle[radius=0.02];
      \fill[black] (0,-1) circle[radius=0.02];
      \fill[black] (1,-1) circle[radius=0.02];   
	  \end{scope}
	  \begin{scope}[xshift=2.5cm]
      \draw[->] (1.1,-1/2) to (2.4,-1/2);
      \node at (1.75,-1/2) [above]{$l_2$};
      \draw (0,0) to node[sloped] {\tikz{\draw[->](0,0)--(0.1,0);}} (0,-1);
      \draw (1,0) to node[sloped] {\tikz{\draw[->](0,0)--(0.1,0);}} (1,-1);
      \draw (0,0) to node[sloped] {\tikz{\draw[-](0,0)--(0.1,0);}} (1,0);
      \draw (0,-1) to node  {\tikz{\draw[-](0,0)--(0.1,0);}} (1,-1);
      \node at (1/2,-1/2) {$\del(l_1)h$};
      \draw (0,0) to node {\tikz{\draw[-](0,0)--(0.1,0);}} (1,0);
      \draw (0,-1) to node {\tikz{\draw[-](0,0)--(0.1,0);}}(1,-1);
      \fill[black] (0,0) circle[radius=0.02];
      \fill[black] (1,0) circle[radius=0.02];
      \fill[black] (0,-1) circle[radius=0.02];
      \fill[black] (1,-1) circle[radius=0.02];   
	  \end{scope}
	    \begin{scope}[xshift=5.0cm]
      \draw (0,0) to node[sloped] {\tikz{\draw[->](0,0)--(0.1,0);}} (0,-1);
      \draw (1,0) to node[sloped] {\tikz{\draw[->](0,0)--(0.1,0);}} (1,-1);
      \draw (0,0) to node[sloped] {\tikz{\draw[-](0,0)--(0.1,0);}} (1,0);
      \draw (0,-1) to node  {\tikz{\draw[-](0,0)--(0.1,0);}} (1,-1);
      \node at (1/2,-1/2) {$\del(l_2 l_1)h$};
      \draw (0,0) to node {\tikz{\draw[-](0,0)--(0.1,0);}} (1,0);
      \draw (0,-1) to node {\tikz{\draw[-](0,0)--(0.1,0);}}(1,-1);
      \fill[black] (0,0) circle[radius=0.02];
      \fill[black] (1,0) circle[radius=0.02];
      \fill[black] (0,-1) circle[radius=0.02];
      \fill[black] (1,-1) circle[radius=0.02];   
	  \end{scope}
	 \begin{scope}[xshift=0cm, yshift=-2.2cm]
	 \begin{scope}[xshift=0.0cm]
      \draw[->] (1.1,-1/2) to (2.4,-1/2);
      \node at (1.75,-1/2) [above]{$l_1$};
      \draw (0,0) to node[sloped] {\tikz{\draw[->](0,0)--(0.1,0);}} (0,-1);
      \draw (1,0) to node[sloped] {\tikz{\draw[->](0,0)--(0.1,0);}} (1,-1);
      \draw (0,0) to node[sloped] {\tikz{\draw[-](0,0)--(0.1,0);}} (1,0);
      \draw (0,-1) to node  {\tikz{\draw[-](0,0)--(0.1,0);}} (1,-1);
      \node at (1/2,-1/2) {$h_1$};
      \draw (0,0) to node {\tikz{\draw[-](0,0)--(0.1,0);}} (1,0);
      \draw (0,-1) to node {\tikz{\draw[-](0,0)--(0.1,0);}}(1,-1);
      \fill[black] (0,0) circle[radius=0.02];
      \fill[black] (1,0) circle[radius=0.02];
      \fill[black] (0,-1) circle[radius=0.02];
      \fill[black] (1,-1) circle[radius=0.02];   
	  \end{scope}
	  \begin{scope}[xshift=2.5cm]
      \draw[->] (1.1,-1/2) to (2.4,-1/2);
      \node at (1.75,-1/2) [above]{$\del(m_2)l_2$};
      \draw (0,0) to node[sloped] {\tikz{\draw[->](0,0)--(0.1,0);}} (0,-1);
      \draw (1,0) to node[sloped] {\tikz{\draw[->](0,0)--(0.1,0);}} (1,-1);
      \draw (0,0) to node[sloped] {\tikz{\draw[-](0,0)--(0.1,0);}} (1,0);
      \draw (0,-1) to node  {\tikz{\draw[-](0,0)--(0.1,0);}} (1,-1);
      \node at (1/2,-1/2) {$\del(l_1)h_1$};
      \draw (0,0) to node {\tikz{\draw[-](0,0)--(0.1,0);}} (1,0);
      \draw (0,-1) to node {\tikz{\draw[-](0,0)--(0.1,0);}}(1,-1);
      \fill[black] (0,0) circle[radius=0.02];
      \fill[black] (1,0) circle[radius=0.02];
      \fill[black] (0,-1) circle[radius=0.02];
      \fill[black] (1,-1) circle[radius=0.02];   
	  \end{scope}
	    \begin{scope}[xshift=5.0cm]
      \draw (0,0) to node[sloped] {\tikz{\draw[->](0,0)--(0.1,0);}} (0,-1);
      \draw (1,0) to node[sloped] {\tikz{\draw[->](0,0)--(0.1,0);}} (1,-1);
      \draw (0,0) to node[sloped] {\tikz{\draw[-](0,0)--(0.1,0);}} (1,0);
      \draw (0,-1) to node  {\tikz{\draw[-](0,0)--(0.1,0);}} (1,-1);
      \node at (1/2,-1/2) {$\del(l_2 l_1)h_1$};
      \draw (0,0) to node {\tikz{\draw[-](0,0)--(0.1,0);}} (1,0);
      \draw (0,-1) to node {\tikz{\draw[-](0,0)--(0.1,0);}}(1,-1);
      \fill[black] (0,0) circle[radius=0.02];
      \fill[black] (1,0) circle[radius=0.02];
      \fill[black] (0,-1) circle[radius=0.02];
      \fill[black] (1,-1) circle[radius=0.02];   
	  \end{scope}
	  \end{scope}
	  \draw [line width=1pt, double distance=3pt, arrows = {-Latex[length=0pt 3 0]}] (3.0,-1.1) -- (3.0,-2.2);
      \node at (3.0, -1.5)  [fill=white]{$ m_2 \#_3 {l_1} $};
  \end{tikzpicture}
\end{center}
\caption{Diagram of $m_2 \#_3 l_1$}
\label{d:m2_l1}
\end{figure}
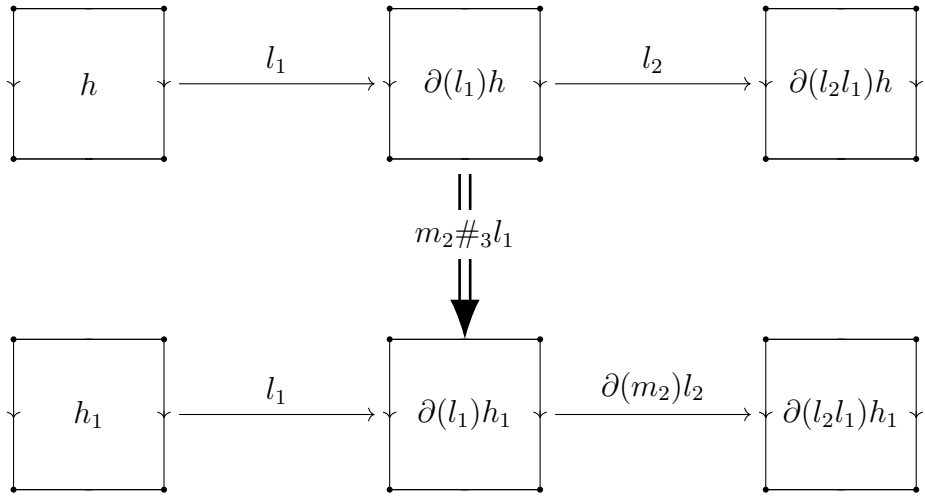

The LL-Peiffer lifting is a one-dimensional higher analogue of the ordinary Peiffer lifting. For each $l_1,l_2 \in L$, we have the following relation:
\begin{align*}
\del \{l_2 , l_1 \}_{LL} \act{\del l_2}l_1 l_2 = l_2 l_1. \\
\end{align*}
Using 
\[
\act{\del l_2}l_1 l_2 = \left({\rm id}_{\del l_2} \#_2 l_1\right) \#_3 \left(l_2 \#_2 {\rm id}_{e_H}\right), \qquad l_2 l_1 = \left(l_2 \#_2 {\rm id}_{\del l_1}\right) \#_3 \left({\rm id}_{e_H} \#_2 l_1\right),
\]
we can describe $\{l_2 , l_1 \}_{LL}$ as shown in Figure~\ref{d:LL_lift}.

\begin{figure}
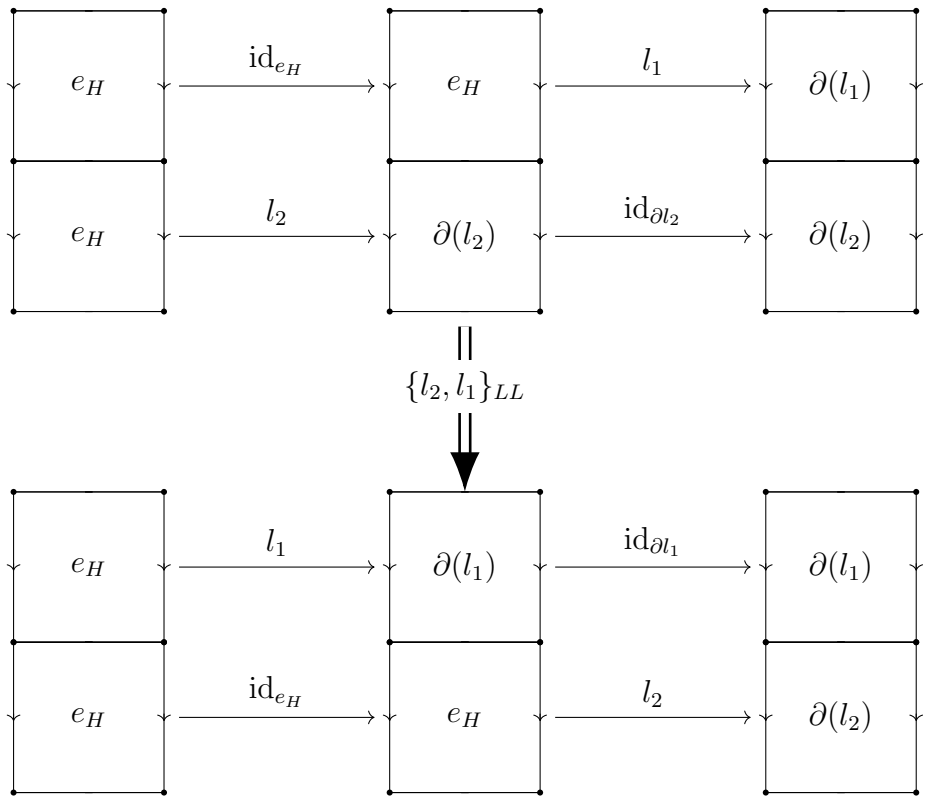

\begin{center}

\end{center}
\caption{Diagram of LL-Peiﬀer lifting}
\label{d:LL_lift}
\end{figure}

The meanings of Properties~\ref{lhmnr}, \ref{rhmnr}, \ref{lrhmnr} and~\ref{yanBt} of a 3-crossed module can be understood diagrammatically.  In particular, Property~\ref{lhmnr} can be represented as in Figure~\ref{d:rel_of_lhmn}.

\begin{figure}
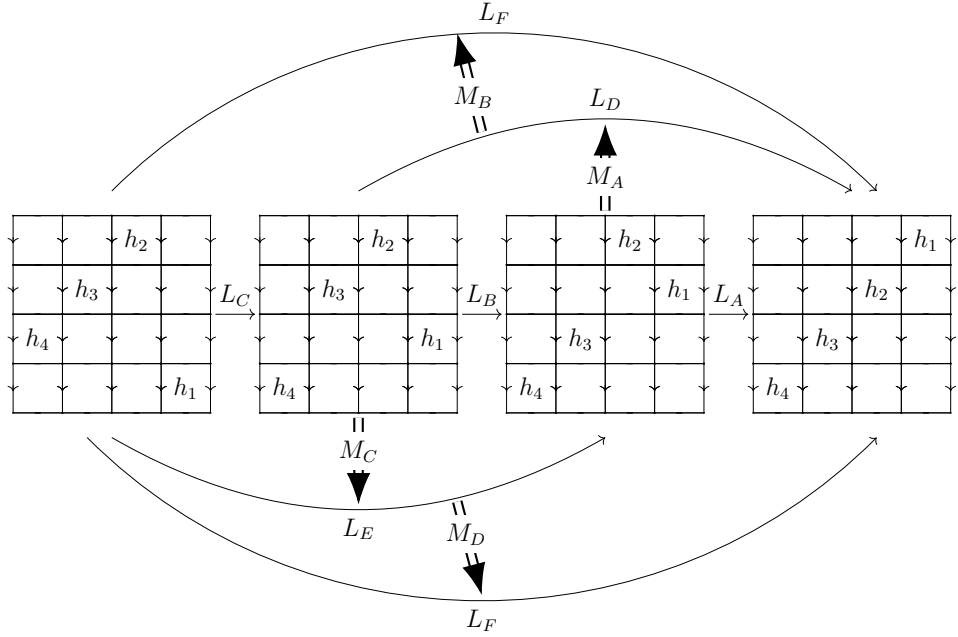

\begin{center}
\scalebox{0.8}{
  %
  }
\end{center}
\caption{Property \ref{lhmnr} of 3-crossed module}
\label{d:rel_of_lhmn}
\end{figure}

Each label form $L_A$ to $L_G$ and from $M_A$ to $M_D$ can be written as follows:
\begin{equation*}
    \begin{aligned}
L_A &= \act{h_4h_3}\{h_2, h_1\}, \ 
L_B = \act{h_4}\{h_3,\act{\del h_2} h_1 \} \\
L_C &= \{h_4, \act{\del(h_3h_2)}h_1 \}, \ 
L_D=\act{h_4}\{h_3h_2, h_1\} \\
L_E &=\{h_4h_3,h_1\}, \ 
L_F =\{h_4h_3h_2,h_1\}\\
M_A &= \act{h_4}\{h_3, h_2, h_1\}, \ 
M_B =\{h_4, h_3 h_2, h_1\}\\
M_C &= \{h_4h_3,\act{\del h_2}h_1\}, \ 
M_D = \{h_4h_3, h_2, h_1\}.
\end{aligned}
\end{equation*}

Hence, Property~\ref{lhmnr} can be expressed as
\[
M_D \act{L_A}M_C = M_B M_A
\]
which means that the two possible compositions corresponding to the left-Homanian are identical.

Property~\ref{rhmnr} of a 3-crossed module can be obtained in a similar manner.  This property can be represented by Figure~\ref{d:rel_of_rhmn}.

\begin{figure}[htbp]
\begin{center}
\scalebox{0.8}{

  }
\end{center}
\caption{Proprety of \ref{rhmnr} of 3-crossed module}
\label{d:rel_of_rhmn}
\end{figure}

In Figure~\ref{d:rel_of_rhmn}, each label form $L_A$ to $L_G$ and from $M_A$ to $M_D$ can be written as in equation~(\ref{eq:drhmnr}).

\begin{equation}
\label{eq:drhmnr}
    \begin{aligned}
L_A &= \act{}\{h_4, h_3\}, \ 
L_B = \act{(\act{\del h_4}h_3)}\{h_4, h_2\} \\
L_C &= \act{\left(\act{\del h_4}(h_3h_2)\right)}\{h_4,h_1\}, \ 
L_D=\act{}\{h_4,h_3, h_2\} \\
L_E &=\act{(\act{\del h_4}h_3)}\{h_4,h_2h_1\}, \ 
L_F =\{h_4,h_3h_2h_1\}\\
M_A &= \act{}\{h_4, h_3, h_2\}', \ 
M_B =\{h_4, h_3 h_2, h_1\}'\\
M_C &= \act{(\act{\del h_4}h_3)}\{h_4,h_2,h_1\}', \ 
M_D = \{h_4,h_3, h_2h_1\}'
\end{aligned}
\end{equation}

This shows that Property~\ref{rhmnr} of a 3-crossed module can be expressed as 
\[
M_B M_A = M_D \act{L_A} M_C.
\]
By examining Figure~\ref{d:rel_of_rhmn}, we see that this property states that the two ways of applying the right-Homanian coincide.

Property~\ref{lrhmnr} of a 3-crossed module can be represented by Figure~\ref{dlrhmnr}.
\begin{figure}[htbp]
\begin{center}
\scalebox{0.7}{

  }
\end{center}
\caption{Proprety of \ref{lrhmnr} of 3-crossed module}
\label{dlrhmnr}
\end{figure}

Each label from $L_A$ to $L_K$ and from $M_A$ to $M_G$ in Figure~\ref{dlrhmnr} corresponds to an element of $L$ or $M$. Here, we explicitly describe only those labels that are used in Property~\ref{lrhmnr}. Some of the labels in Figure~\ref{dlrhmnr} can be written as in equation~(\ref{eq:dlrhmnr}).
\begin{equation}
\label{eq:dlrhmnr}
    \begin{aligned}
L_A &= \act{h_4}\{h_3, h_2\}, \ 
L_G = \act{h_4}\{h_3, h_2h_1\} \\
L_K &= \act{}\{h_4h_3,h_2\} \\ 
M_A &= \act{}\{h_4, \act{\del h_3}h_2, \act{\del h_3}h_1\}', \ 
M_B =\act{h_4}\{h_3, h_2, h_1\}'\\
M_C &= \act{}\{h_4,h_3.h_2h_1\}, \ 
M_D = \{ \{h_4,\act{\del h_3}h_2\}, \act{\act{\del(h_4h_3)}h_2h_4}\{h_3,h_1\}\}_{LL}\\
M_E &=\act{\act{\del(h_4h_3)}h_2}\{h_4,h_3,h_1\}, \ 
M_F =\{h_4, h_3,h_2\}\\
M_G&=\{h_4h_3, h_2,h_1\}
\end{aligned}
\end{equation}
Figure~\ref{dlrhmnr} shows that Property~\ref{lrhmnr} of a 3-crossed module can be expressed as 
\[
M_C \act{L_G}M_A M_B = M_G \act{L_K}M_E M_F \act{L_A}M_D.
\]

Property~\ref{yanBt} of a 3-crossed module can be represented by Figure~\ref{d:yanBt}.

\begin{figure}[htbp]
\begin{center}
\scalebox{0.8}{

  }
\end{center}
\caption{Diagram of proprety of \ref{yanBt} of 3-crossed module}
\label{d:yanBt}
\end{figure}

Each label from $L_A$ to $L_D$ and from $M_A$ to $M_F$ in Figure~\ref{d:yanBt} corresponds to an element of $L$ or $M$. Again, we explicitly describe only those labels that appear in Property~\ref{yanBt}. Some of the labels in Figure~\ref{d:yanBt} can be written as in equation~(\ref{eq:yanBt}).
\begin{equation}
\label{eq:yanBt}
    \begin{aligned}
L_A &= \{h_3,h_2\}, \ 
L_D = \act{h_3}\{h_2,h_1\} \\
M_A &= \{h_3,h_2,h_1\}' , \\
M_B &=\act{(\act{h_3}\{h_2,h_1\})}\{h_3,\partial\{h_2,h_1\}^{-1},h_2h_1\}'\{h_3,\{h_2,h_1\}\}' \{\act{\partial h_3}\{h_2,h_1\},\act{\partial\act{\partial h_3}\{h_2,h_1\}^{-1}}\{h_3,h_2h_1\}\} \\
M_C &= \{h_3,\act{\partial h_2}h_1,h_2\}'^{-1}, \ 
M_D = \{\act{\partial h_3}h_2,h_3,h_1\}\\
M_E &=\act{\{h_3h_2,h_1\}}\{\act{\partial(h_3h_2)}h_1,\{h_3,h_2\}\}^{-1}\\
&\ \ \ \ \ \ \times\{\{h_3,h_2\},\act{\partial\{h_3,h_2\}^{-1}}\{h_3h_2,h_1\}\}^{-1}\act{\{h_3,h_2\}}\{\partial\{h_3,h_2\}^{-1},h_3h_2,h_1\}^{-1}, \\ 
M_F &=\{h_3,h_2,h_1\}^{-1}
\end{aligned}
\end{equation}
Figure~\ref{d:yanBt} shows that property \ref{yanBt} of 3-crossed module can be write as $\act{L_D}M_C M_B M_A = M_F M_E \act{L_A}M_D$.

\section{Cocycle condition}
\label{cocycle}

In this section, we introduce the meaning of formulas~(\ref{r2-simplex m})-(\ref{r5-simplex m}). Throughout this section, we use the lattice-type representation.

First, we illustrate formula~(\ref{r2-simplex m}) by a diagram. For clarity, we consider the specific case $0<1<2<3<4<5$ instead of the general indices $i<j<k<m<p<q$.
 The corresponding diagram is shown in Figure~\ref{d:h012}.
 \begin{figure}
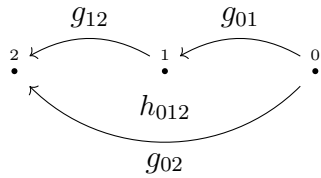

\begin{center}

\end{center}
\caption{Figure of $h_{012}$}
\label{d:h012}
\end{figure}

By using Figure~\ref{d:h012}, we can understand formulas~(\ref{r3-simplex m})-(\ref{r5-simplex m}) diagrammatically.
Formula~(\ref{r3-simplex m}) can be described as shown in Figure~\ref{d:l0123}.
\begin{figure}
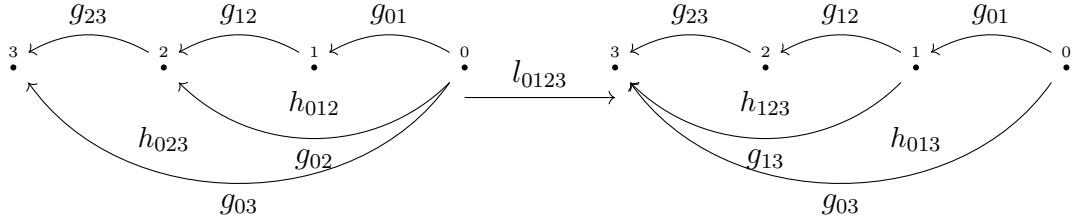

\begin{center}
  %
\end{center}
\caption{Diagram of $l_{0123}$}
\label{d:l0123}
\end{figure}

We can interpret $l_{0123}$ as an operator that shifts the inner arch within the outer arch. Formula~\ref{r4-simplex m} can then be represented as in Figure~\ref{d:m01234}.
\begin{figure}[!]
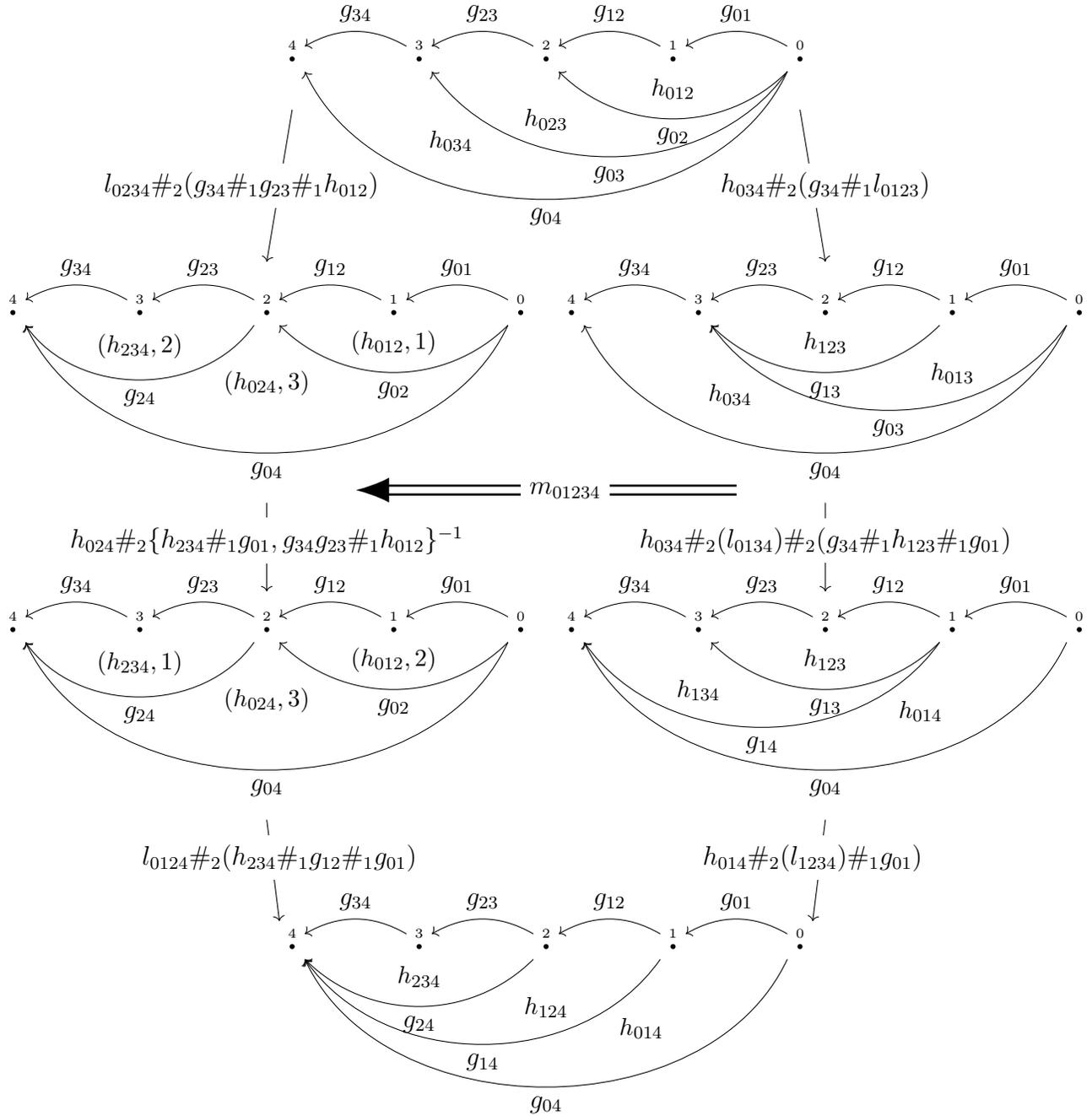

\begin{center}
  %
\end{center}
\caption{Diagram of $m_{01234}$}
\label{d:m01234}
\end{figure}

Formula~(\ref{r5-simplex m:eq}) can also be represented diagrammatically. Figure~\ref{d:5simp-lhs} illustrates the left-hand side of~(\ref{r5-simplex m:eq}) while Figure~\ref{d:5simp-rhs} illustrates the right-hand side.
\begin{figure}[!]
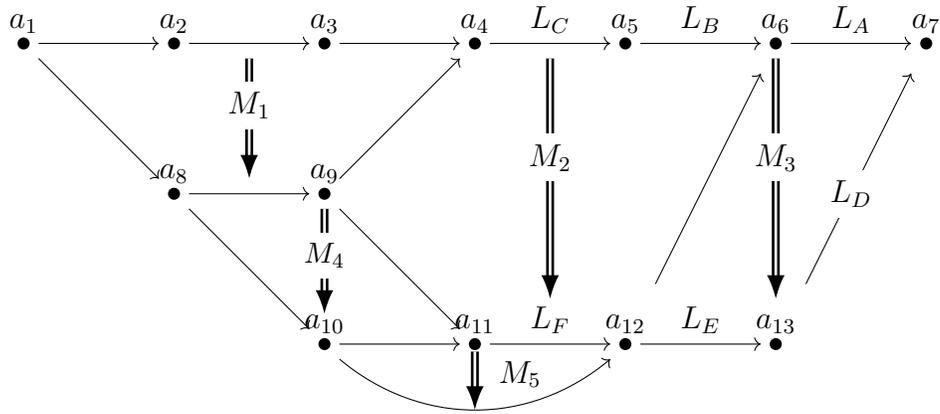

\begin{center}
  %
\end{center}
\caption{Diagram of left hand side of (\ref{r5-simplex m:eq}) }
\label{d:5simp-lhs}
  \end{figure}

\begin{figure}[!]
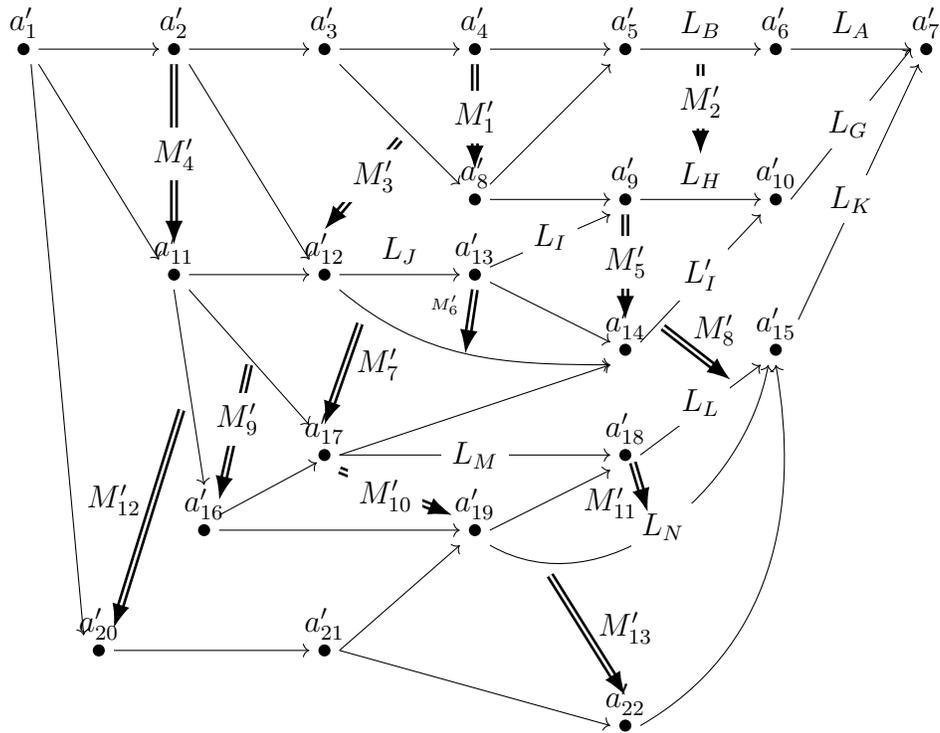

\begin{center}
  %
\end{center}
\caption{Diagram of right hand side of (\ref{r5-simplex m:eq}) }
\label{d:5simp-rhs}
  \end{figure}

In the following, we use vertex labels instead of the lattice-type ones used in the previous figures. List~\ref{VandD} shows the correspondence between the vertices and the lattice-type diagrams appearing in the figures.
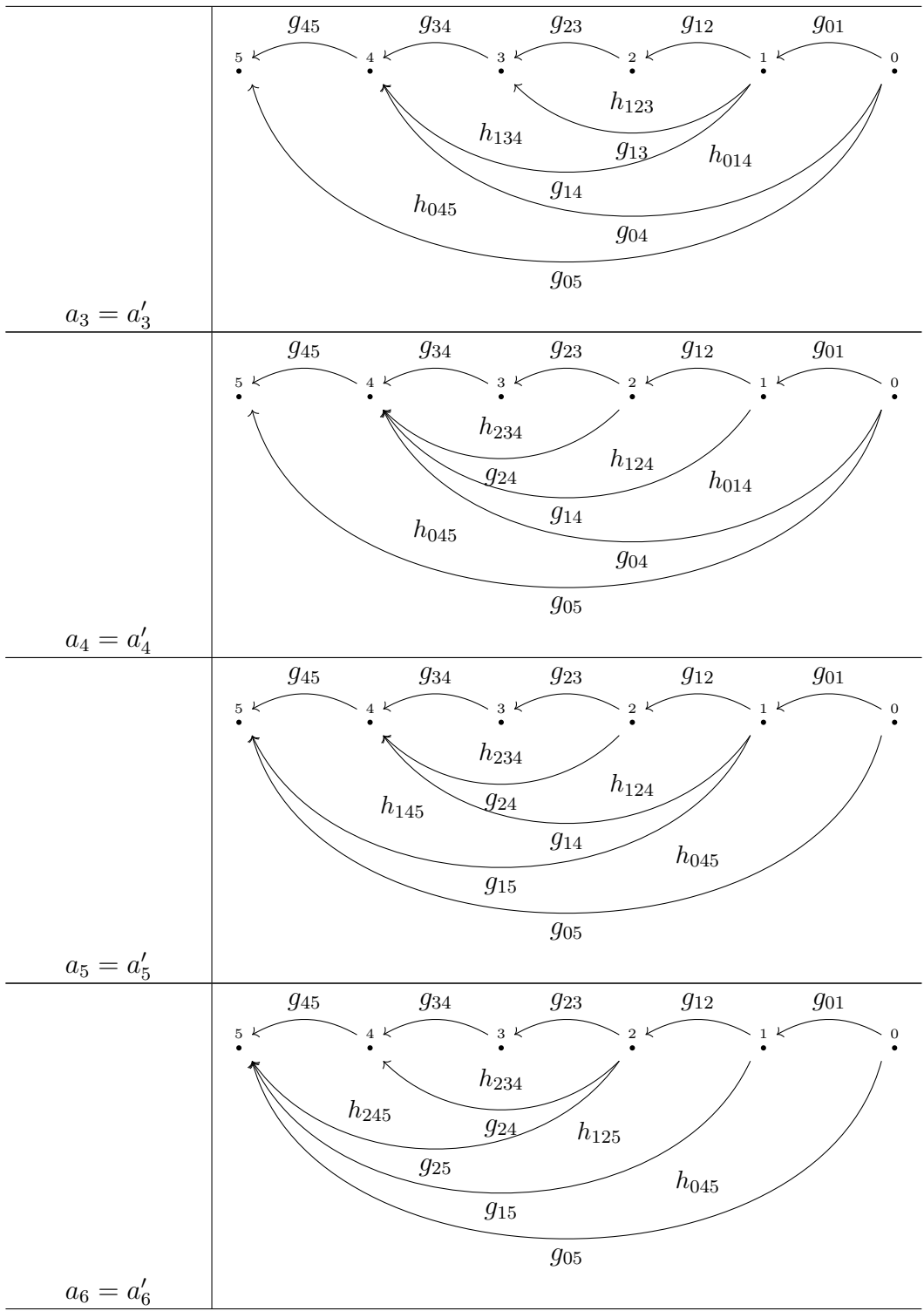
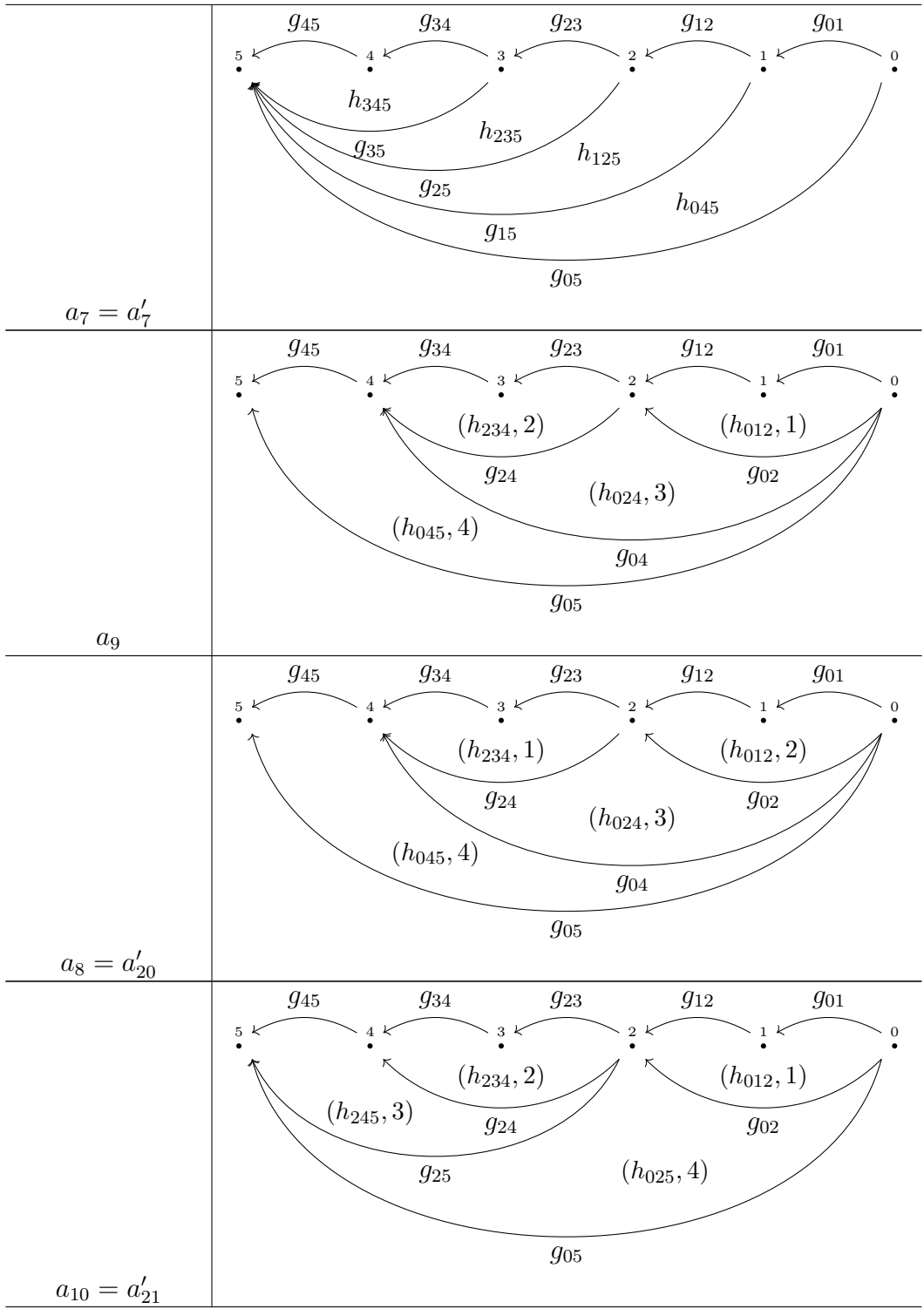
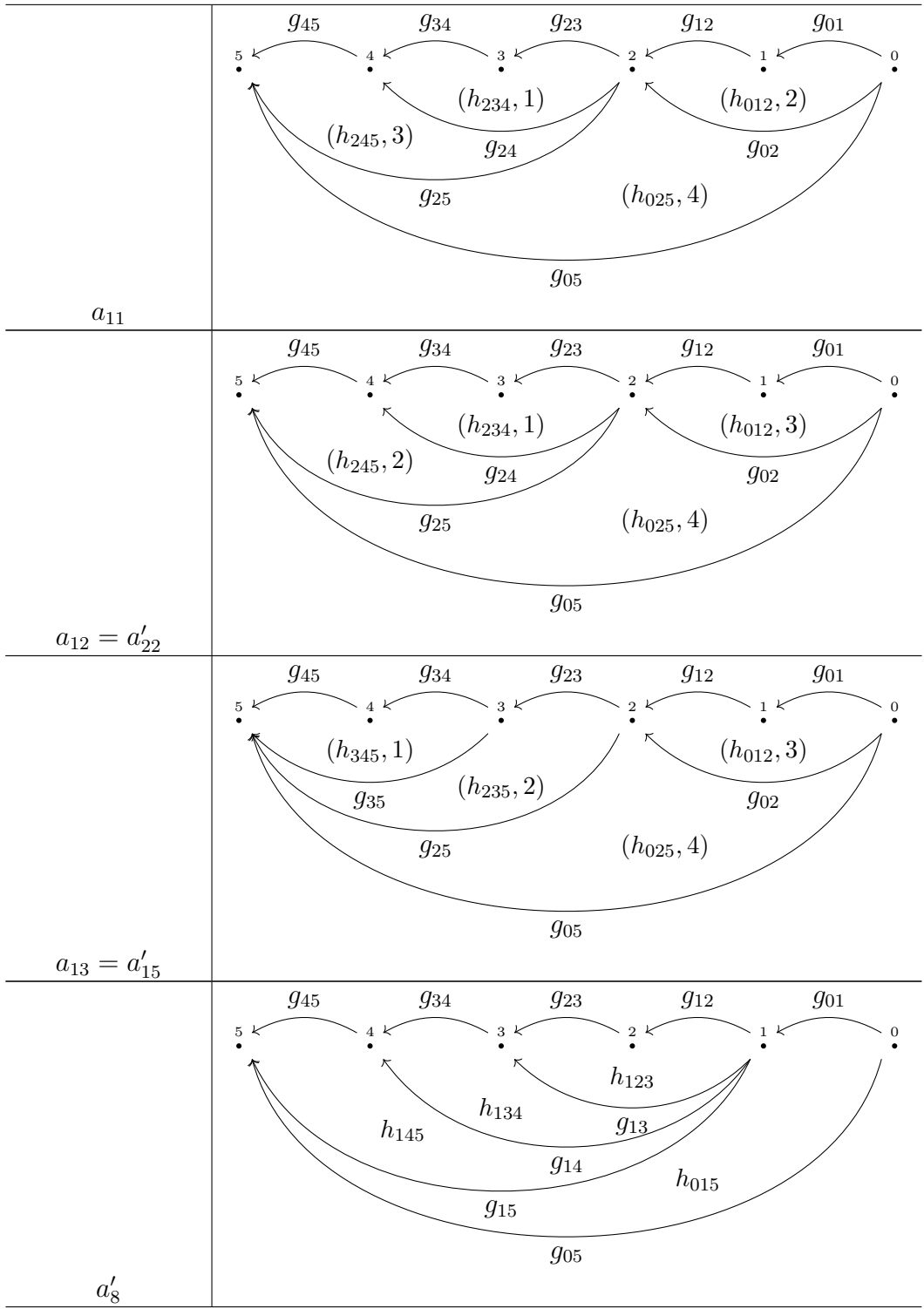
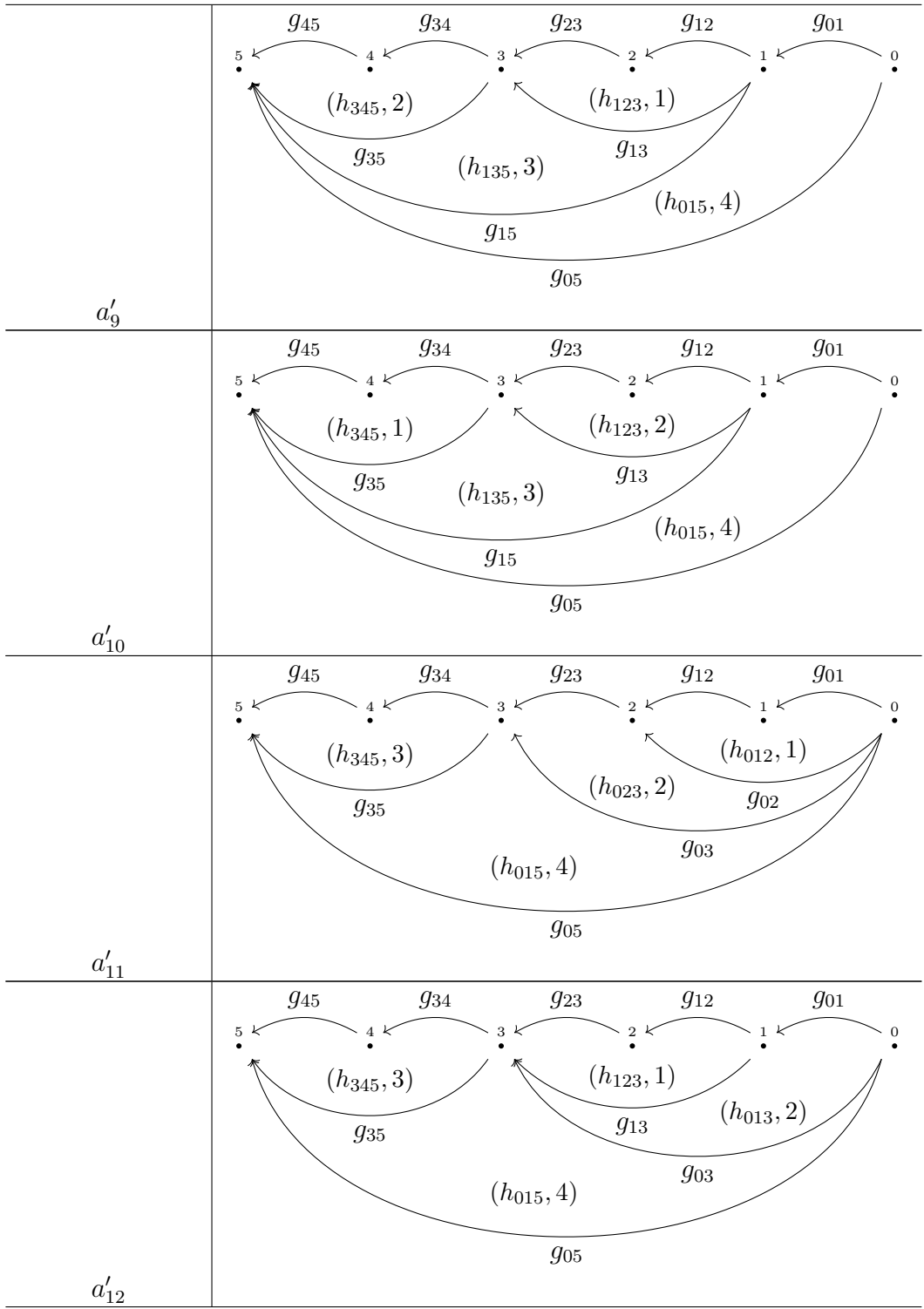
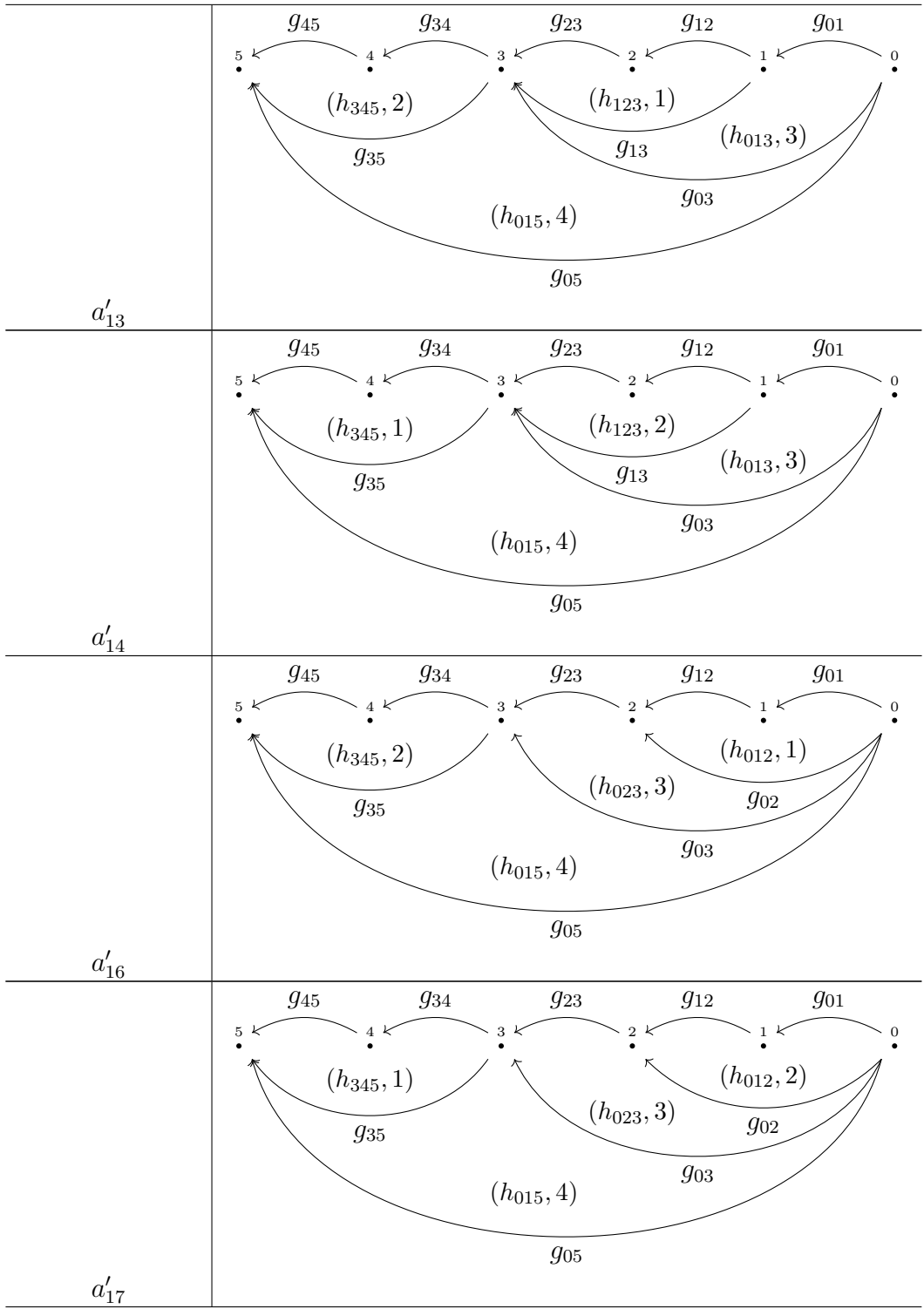
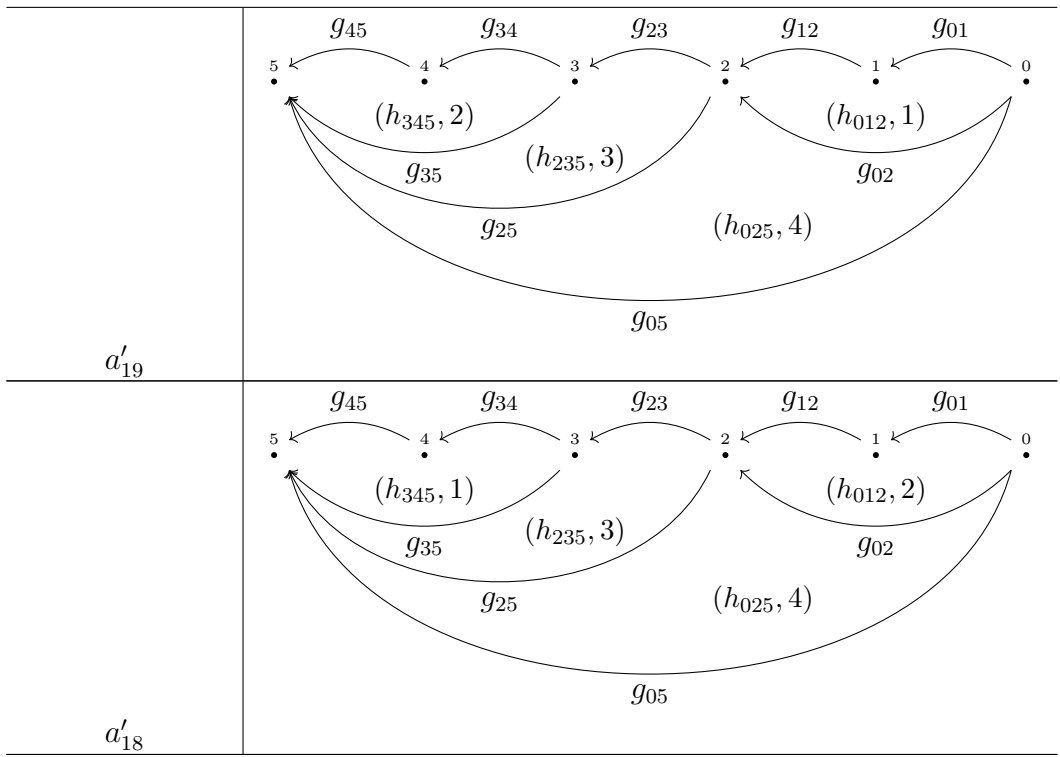
\begin{center}
\begin{longtable}{c|c}
\caption{Correspondence between vertices and the figure} \label{VandD} \\
\hline
Label of vertex & Figure\\
\hline

\end{longtable}
\end{center}

\end{appendices}

\newpage

\bibliographystyle{plain}
\nocite{*} 
\bibliography{refs}

\begin{thebibliography}{10}

\bibitem{AkcaPak2010}
{\.I}.~Ak{\c c}a and S.~Pak.
\newblock {P}seudo simplicial groups and crossed modules.
\newblock {\em Turkish Journal of Mathematics}, 34:475--488, 2010.

\bibitem{Arvasi2009}
Z.~Arvasi, T.~S. Kuzpinari, and E.~{\"O}. Uslu.
\newblock Three-crossed modules.
\newblock {\em Homology, Homotopy and Applications}, 11:161--187, 2009.

\bibitem{BaezSchreiber2004}
J.~Baez and U.~Schreiber.
\newblock Higher gauge theory: 2-connections on 2-bundles, 2004.

\bibitem{Conduche1984}
Daniel Conduch{\'e}.
\newblock Modules crois{\'e}s g{\'e}n{\'e}ralis{\'e}s de longueur 2.
\newblock {\em Journal of Pure and Applied Algebra}, 34:155--178, December
  1984.

\bibitem{HidakaNittaYokokura2021}
Y.~Hidaka, M.~Nitta, and R.~Yokokura.
\newblock Topological axion electrodynamics and 4-group symmetry.
\newblock {\em Physics Letters B}, 823:136762, 2021.

\bibitem{HidakaNittaYokokura2022}
Y.~Hidaka, M.~Nitta, and R.~Yokokura.
\newblock Global 4-group symmetry and ’t hooft anomalies in topological axion
  electrodynamics.
\newblock {\em Progress of Theoretical and Experimental Physics}, 2022:04A109,
  2022.

\bibitem{MartinsPorter2007}
J.~F. Martins and T.~Porter.
\newblock On yetter's invariant and an extension of the dijkgraaf--witten
  invariant to categorical groups.
\newblock {\em Theory and Applications of Categories}, 18:118--150, 2007.

\bibitem{MartinsPicken2011}
Jo{\~a}o~Faria Martins and Roger Picken.
\newblock {T}he fundamental {G}ray 3-groupoid of a smooth manifold and local
  3-dimensional holonomy based on a 2-crossed module.
\newblock {\em Differential Geometry and its Applications}, 29:179--206, Mar
  2011.

\bibitem{Noohi2007}
B.~Noohi.
\newblock Notes on 2-groupoids, 2-groups and crossed-modules.
\newblock {\em Homology, Homotopy and Applications}, 9(1):75--106, 2007.

\bibitem{Porst2008}
Sven-S. Porst.
\newblock Strict 2-groups are crossed modules, 2008.

\bibitem{RadenkovicVojinovic2022}
T.~Radenkovic and M.~Vojinovic.
\newblock Topological invariant of 4-manifolds based on a 3-group.
\newblock {\em Journal of High Energy Physics}, 07:105, 2022.

\bibitem{RezkIntroQCats}
Charles Rezk.
\newblock Introduction to quasicategories.
\newblock Lecture notes, University of Illinois, 2022.

\bibitem{SarikayaUlualan2024}
M.~Sarikaya and E.~Ulualan.
\newblock Comparing 2-crossed modules with gray 3-groups.
\newblock {\em Theory and Applications of Categories}, 41(45):1557--1595, 2024.

\bibitem{SchreiberWaldorf2007}
U.~Schreiber and K.~Waldorf.
\newblock Parallel transport and functors, 2007.

\bibitem{SchreiberWaldorf2008a}
U.~Schreiber and K.~Waldorf.
\newblock Smooth functors vs. differential forms, 2008.
\newblock arXiv e-prints (Feb., 2008).

\bibitem{SchreiberWaldorf2013}
U.~Schreiber and K.~Waldorf.
\newblock Connections on non-abelian gerbes and their holonomy.
\newblock {\em Theory and Applications of Categories}, 28(17):476--540, 2013.

\bibitem{Wang2014}
W.~Wang.
\newblock On 3-gauge transformations, 3-curvatures, and gray-categories.
\newblock {\em Journal of Mathematical Physics}, 55:043506, 2014.

\bibitem{Whitehead1949}
J.~H.~C. Whitehead.
\newblock Combinatorial homotopy. {II}.
\newblock {\em Bulletin of the American Mathematical Society}, 55:453--496, May
  1949.

\end{thebibliography}

\end{document}